\documentclass[11pt]{article}
\usepackage{amsmath,amsfonts,amssymb,amsthm,enumerate,graphicx}
\usepackage{tikz}
\usepackage{subfigure}
\usepackage{url}
\usepackage{amsmath}
\usepackage{amssymb}
\usepackage[german,english]{babel}
\usepackage{url}
\usepackage{graphicx}
\usepackage{gastex}
\usepackage{longtable}
\usepackage{lscape}
\usepackage{multicol}
\usepackage{latexsym}
\usepackage{verbatim}
\usepackage{multirow}
\usepackage{authblk}

\usepackage[usenames,dvipsnames]{pstricks}
 \usepackage{pstricks-add}
 \usepackage{epsfig}
 \usepackage{pst-grad} 
 \usepackage{pst-plot} 
 \usepackage[space]{grffile} 
 \usepackage{etoolbox}

\usepackage{amssymb,latexsym,amsmath,epsfig,amsthm,mathrsfs}
\usepackage{amsfonts}
\usepackage{amscd}
\usepackage{dsfont}
\usepackage{bbm}
\usepackage{rotating}
\usepackage{newlfont}
\usepackage{enumitem}
\usepackage{lineno,hyperref}
\usepackage[english]{babel}
\usepackage{tabularx}
\usepackage{tikz}
\usepackage{tkz-graph}
\usepackage{longtable}
\usepackage{tabularx,ragged2e,booktabs,caption}
\tikzstyle{vertex}=[ draw, inner sep=0pt, minimum size=0pt]

\usepackage{epstopdf}
\usepackage{url}
\usepackage{hyperref}
\usepackage{float}
\usepackage{enumitem}
\newlist{subquestion}{enumerate}{1}
\date{}

\newtheorem{theorem}{{\bf Theorem}}[section]

\newtheorem{definition}[theorem]{{\bf Definition}}
\newtheorem{example}[theorem]{{\bf Example}}
\newtheorem{lemma}[theorem]{{\bf Lemma}}

\newtheorem{proposition}[theorem]{{\bf Proposition}}

\newtheorem{Case}{Case}
\newtheorem{Case1}{Case}
\newtheorem{Case2}{Case}
\newtheorem{Subcase}{Case}
\numberwithin{Subcase}{Case}
\newtheorem{Subsubcase}{Case}
\numberwithin{Subsubcase}{Subcase}
\newtheorem{claim}{\textit{Claim}}
\newcommand{\va}[1]{ a_{#1} }
\newcommand{\vb}[1]{ b_{#1} }

\newcommand{\vv}[1]{ v_{#1} }
\newcommand{\vl}[1]{ lk(#1) }

\newcommand{\5}{5-gon}

\newcommand{\Echar}[1]{Euler characteristic #1}
\newcommand{\vr}[1]{vertices #1}
\newcommand{\vx}[1]{vertex #1}

\newcommand{\lnk}[4]{ $lk(#1)=C_8([#2],[#3],#4)  $}
\newcommand{\klk}[5]{ $lk(#1)=C_9([#2],[#3],[#4],[#5])  $}
\newcommand{\lkk}[5]{$ lk(#1)=C_9([#2],#3,#4,#5) $}

\begin{document}
\author[1] {Debashis Bhowmik}
\author[2] {Dipendu Maity}
\author[3] {Ashish Kumar Upadhyay}
\author[4] {Bhanu Pratap Yadav}
\affil[1, 3, 4]{Department of Mathematics, Indian Institute of Technology Patna, Patna 801\,106, India.
 \{debashis.pma15, upadhyay, bhanu.pma16\}@iitp.ac.in.}
\affil[2]{Department of Sciences and Mathematics,
	Indian Institute of Information Technology Guwahati, Bongora, Assam-781\,015, India.
	dipendu@iiitg.ac.in/dipendumaity@gmail.com.}
	
\title{Semi-equivelar maps on the surface of Euler genus 3}


\date{April 30, 2020}

\maketitle

\vspace{-10mm}

\begin{abstract}
If the cyclic sequence of faces for all the vertices in a map are of same type, then the map is said to be a semi-equivelar map. In this article, we classify all the types of semi-equivelar maps on the surface of Euler genus 3, $i.e.$, on the surface of Euler characteristic  $-1$. That is, we present {a complete map types of} semi-equivelar maps (if exist)  on the surface of Euler char.  $-1$. We know the complete list of semi-equivelar maps (upto isomorphism) for some types. Here, we also present a complete list of semi-equivelar maps for one type and for other types, similar steps can be followed.  
\end{abstract}

\noindent {\small {\em MSC 2010\,:} 52C20, 52B70, 51M20, 57M60.

}

\section{Introduction}

Recall that a map is a connected 2-dimensional cell complex, or, equivalently, a cellular embedding of a connected graph in a surface. In this article, a map will mean a polyhedral map on a surface. Two maps are said to be isomorphic if there exists a bijective map between the vertex sets, which induced a bijective map between the face sets. In particular, an automorphism of a map is a permutation of the set of vertices that preserves incidence with edges and faces of the embedding. A map  $X $ is said to be {\em vertex-transitive} if the group of all automorphism, denoted ${ Aut}(X) $, acts transitively on the set  $V(X) $ of vertices of  $X $.

For a vertex  $u $ in a map  $X $, the faces containing  $u $ form a cycle (called the {\em face-cycle} at  $u $)  $C_u $ in the dual graph  of  $X $. That is,  $C_u $ is of the form  $(F_{1,1}\mbox{-}\cdots \mbox{-}F_{1,n_1})\mbox{-}\cdots\mbox{-}(F_{k,1}\mbox{-}\cdots \mbox{-}F_{k,n_k})\mbox{-}$ $F_{1,1} $, where  $F_{i,\ell} $ is a  $p_i $-gon for  $1\leq \ell \leq n_i $,  $1\leq i \leq k $,  $p_r\neq p_{r+1} $ for  $1\leq r\leq k-1 $ and  $p_n\neq p_1 $. A map  $X $ is said to be {\em semi-equivelar} (see \cite{datta2017}) if  $C_u $ and  $C_v $ are of same type for all  $u, v \in V(X) $. That is, there exist integers  $p_1, \dots, p_k\geq 3 $ and  $n_1, \dots, n_k\geq 1 $,  $p_i\neq p_{i+1} $ (addition in the suffix is modulo  $k $) such that  $C_u $ is of the form as above for all  $u\in V(X) $. In such a case,  $X $ is called a semi-equivelar map of type  $[p_1^{n_1}, \dots, p_k^{n_k}] $ (or, a map of type  $[p_1^{n_1}, \dots, p_k^{n_k}] $). Clearly, a
vertex-transitive map is semi-equivelar. 

It is a classical problem to ask which are the types of semi-equivelar maps exist on a surface. In this article, we discuss the same on the  surface of Euler char.  $-1$. On this surface, Upadhyay, Tiwari and Maity \cite{utm2014} have given a complete list of maps of type $[3^5, 4^1]$, and shown that there are exactly three such maps. Tiwari and Upadhyay \cite{tu2018} have also studied maps of the types whose at least one map of the type is also available on the orientable surface of genus 2 and vertex-transitive. (The list of vertex-transitive semi-equivelar maps on the surface of orientable genus 2, 3, 4 can be found in  \cite{kn2012, k2011}.) In particular, they have shown that $(i)$ there are no such maps of type $[3^4, 8^1]$, $[3^2, 4^1, 3^1, 6^1]$, $[3^1, 6^1, 4^1, 6^1]$, $[4^3, 6^1]$
or $[4^1, 8^1, 12,^1]$, and $(ii)$ there are exactly two such maps of each of type $[3^1, 4^1, 8^1, 4^1]$, $[4^1, 6^1, 16^1]$ and $[6^1, 8^1]$. It is also shown that none of the above-mentioned nine semi-equivelar maps are vertex-transitive. 
 Bhowmik and Upadhyay \cite{bu2019} have extended this studied and shown that there is exactly one map of type  $[3^1, 4^1, 3^1, 4^2] $ on the surface of Euler char.  $-1 $. In this article, we attempt to give the complete list types of semi-equivelar maps which exist on the surface of Euler char. $-1 $. More precisely, we present the complete list (upto isomorphism) of semi-equivelar maps for the vertices upto twenty, and for {vertices more than twenty}, we show the existence of semi-equivelar map (if exists) on the surface of Euler characteristics $-1 $. (Here, we have done the computation for one type only,  and for other types, one can follow similar steps to get the complete list.) Thus, we have the following results.

\begin{theorem} \label{theo:sf}
Let  $X $ be a semi-equivelar map on the surface of Euler char.  $-1 $ then  $X $ is  of type  $[3^1, 4^1, 3^1, 4^2], [3^5, 4^1], [3^1, 4^1, 8^1, 4^1], [4^1, 6^1, 16^1], [6^2, 8^1] $,  $[4^3, 5^1] $, $[6^2, 7^1]$, $[3^1, 4^1, 7^1, 4^1]$, $[4^1, 6^1, 14^1]$ or $[4^1, 8^1, 10^1]$.
\end{theorem} 

\begin{theorem} \label{theo:sf1}
	Let  $X $ be a semi-equivelar map of type  $[4^3, 5^1]$ on the surface of Euler char.  $-1 $. Then, $X \cong$ $\cal{K}_1, \cal{K}_2$ or $\cal{K}_3$ ($\cal{K}_1, \cal{K}_2, \cal{K}_3$ are given in Example \ref{eg:8maps-torus}).
\end{theorem}

 From \cite{bu2019, utm2014, tu2018} we know that the maps {$\cal{K}_{8}, \cal{K}_{9}, \dots, \cal{K}_{17}$} are not vertex-transitive. Thus, we have the following.

\begin{theorem} \label{theo:sf2}
	The maps $\cal{K}_1, \cal{K}_2, \dots, \cal{K}_{17}$ are not vertex-transitive.
\end{theorem}




\section{Some semi-equivelar maps}

\begin{example} \label{eg:8maps-torus}
	{\rm Semi-equivelar maps of types $[4^3, 5^1]$, $[6^2, 7^1],$ $[3^1, 4^1, 7^1, 4^1],$ $[4^1, 6^1, 14^1]$ and $[4^1, 8^1, 10^1]$ on the surface of Euler char. -1.}

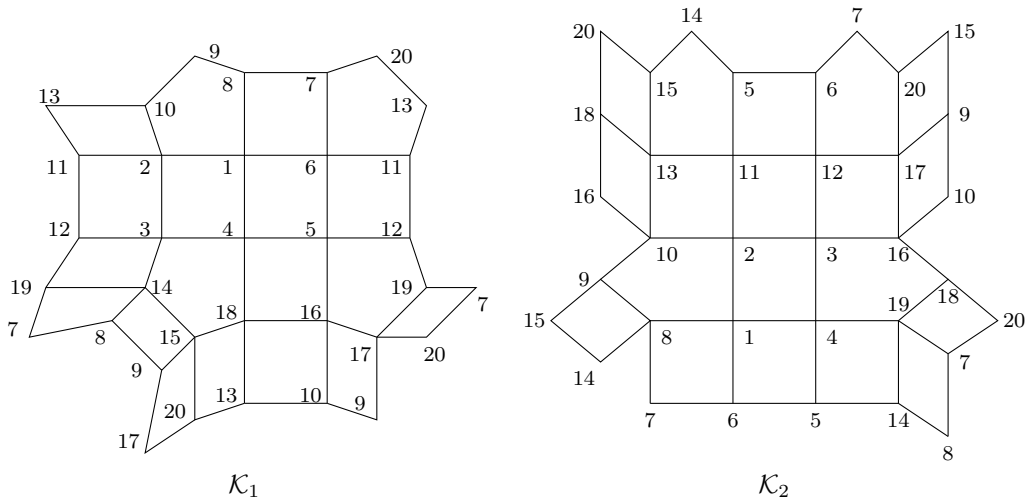
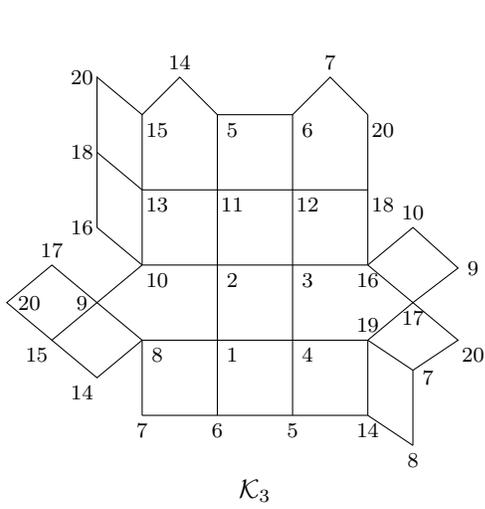
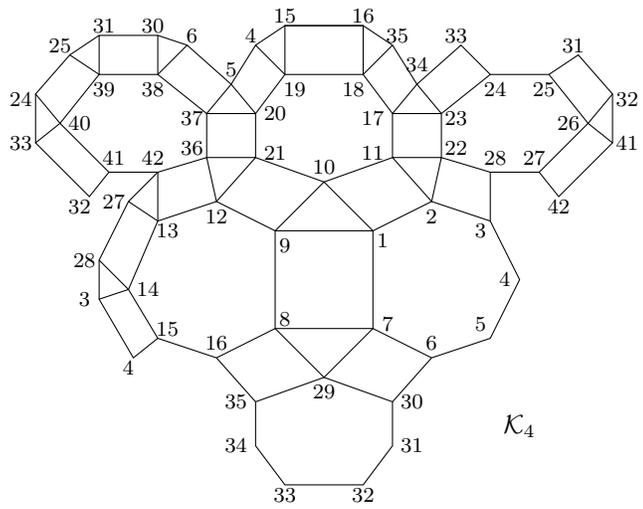
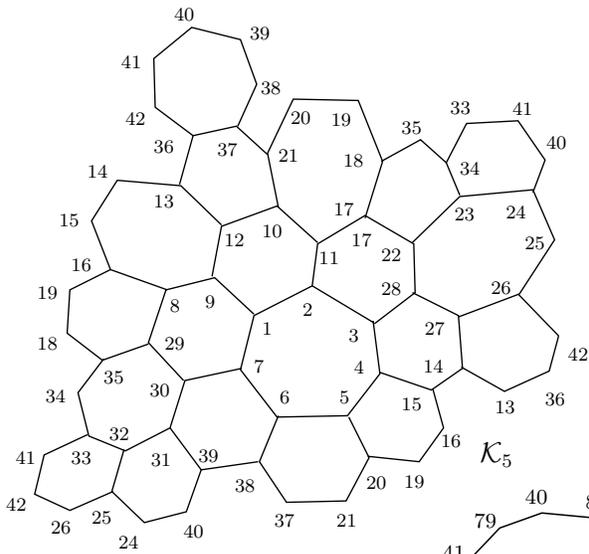
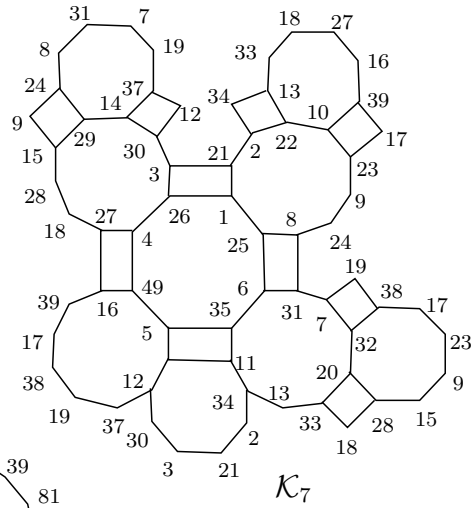
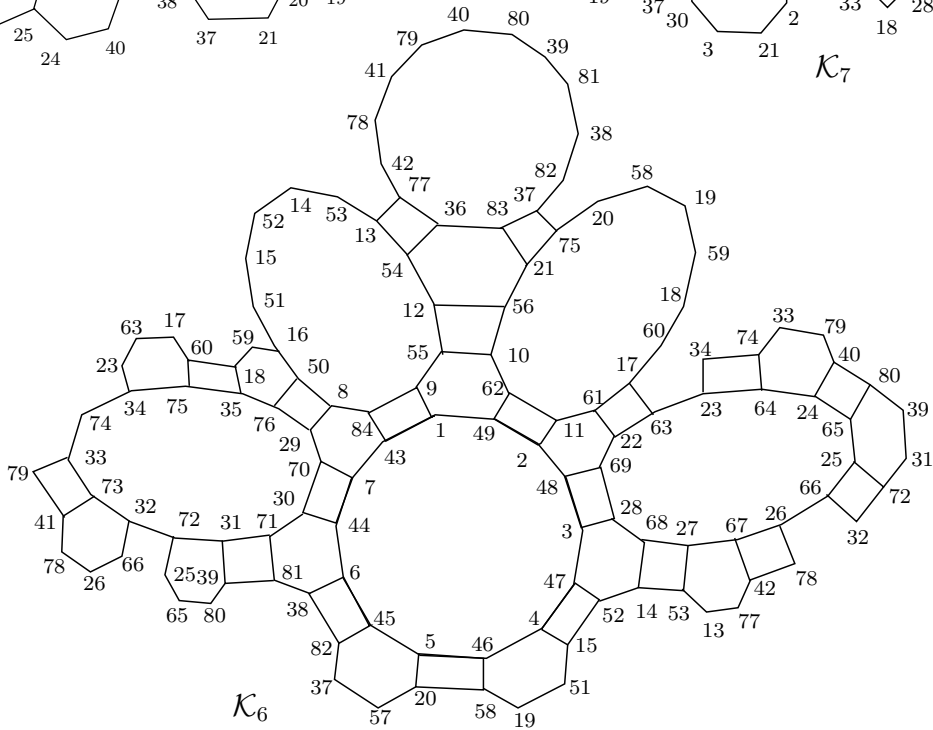
\begin{figure}[H]
	\tiny
	\tikzstyle{ver}=[]
	\tikzstyle{vert}=[circle, draw, fill=black!100, inner sep=0pt, minimum width=4pt]
	\tikzstyle{vertex}=[circle, draw, fill=black!00, inner sep=0pt, minimum width=4pt]
	\tikzstyle{edge} = [draw,thick,-]
	\centering
	\begin{tikzpicture}[scale=0.22]
	
	\draw[edge, thin](10,5)--(30,5);
	\draw[edge, thin](10,10)--(30,10);
	\draw[edge, thin](10,5)--(10,10);
	\draw[edge, thin](15,5)--(15,10);
	\draw[edge, thin](20,-5)--(20,15);
	\draw[edge, thin](25,-5)--(25,15);
	\draw[edge, thin](30,5)--(30,10);
	\draw[edge, thin](20,15)--(25,15);
	\draw[edge, thin](20,-5)--(25,-5);
	\draw[edge, thin](20,0)--(25,0);
	
	\draw[edge, thin](15,5)--(14,2)--(17,-1)--(20,0);
	\draw[edge, thin](30,5)--(31,2)--(28,-1)--(25,0);
	\draw[edge, thin](30,10)--(31,13)--(28,16)--(25,15);
	\draw[edge, thin](15,10)--(14,13)--(17,16)--(20,15);
	\draw[edge, thin](14,2)--(12,0)--(15,-3)--(17,-1);
	
	\draw[edge, thin](12,0)--(7,-1)--(8,2)--(10,5);
	\draw[edge, thin](8,2)--(14,2);
	\draw[edge, thin](10,10)--(8,13)--(14,13);
	\draw[edge, thin](20,-5)--(17,-6)--(17,-1);
	\draw[edge, thin](17,-6)--(14,-8)--(15,-3);
	\draw[edge, thin](28,-1)--(31,-1)--(34,2)--(31,2);
	
	\draw[edge, thin](25,-5)--(28,-6)--(28,-1);
	
	\node[ver] () at (8.8,5.5){\scriptsize ${12} $};
	\node[ver] () at (14,5.5){\scriptsize ${3} $};
	\node[ver] () at (19,5.5){\scriptsize ${4} $};
	\node[ver] () at (24,5.5){\scriptsize ${5} $};
	\node[ver] () at (28.9,5.5){\scriptsize ${12} $};
	
	\node[ver] () at (8.7,9.3){\scriptsize ${11} $};
	\node[ver] () at (14,9.3){\scriptsize ${2} $};
	\node[ver] () at (19,9.3){\scriptsize ${1} $};
	\node[ver] () at (24,9.3){\scriptsize ${6} $};
	\node[ver] () at (28.9,9.3){\scriptsize ${11} $};
	
	\node[ver] () at (24,14.3){\scriptsize ${7} $};
	\node[ver] () at (19,14.3){\scriptsize ${8} $};
	
	\node[ver] () at (24,0.5){\scriptsize ${16} $};
	\node[ver] () at (18.9,0.5){\scriptsize ${18} $};
	
	\node[ver] () at (24,-4.5){\scriptsize ${10} $};
	\node[ver] () at (18.9,-4.5){\scriptsize ${13} $};
	
	\node[ver] () at (6.5,2){\scriptsize ${19} $};
	\node[ver] () at (15,2){\scriptsize ${14} $};
	
	\node[ver] () at (6,-.5){\scriptsize ${7} $};
	\node[ver] () at (11.3,-1){\scriptsize ${8} $};
	\node[ver] () at (13.5,-3){\scriptsize ${9} $};
	\node[ver] () at (13,-7.3){\scriptsize ${17} $};
	\node[ver] () at (15.8,-5.5){\scriptsize ${20} $};
	
	\node[ver] () at (27,-5){\scriptsize ${9} $};
	
	\node[ver] () at (29.5,2){\scriptsize ${19} $};
	\node[ver] () at (27,-2){\scriptsize ${17} $};
	\node[ver] () at (31.5,-2){\scriptsize ${20} $};
	\node[ver] () at (34.4,1){\scriptsize ${7} $};
	
	\node[ver] () at (29.5,13){\scriptsize ${13} $};
	\node[ver] () at (29.5,16){\scriptsize ${20} $};
	
	\node[ver] () at (18.2,16.3){\scriptsize ${9} $};
	\node[ver] () at (15.2,12.8){\scriptsize ${10} $};
	
	\node[ver] () at (8.2,13.6){\scriptsize ${13} $};
	
	
	\node[ver] () at (15.5,-1){\scriptsize ${15} $};
	
	\node[ver] () at (20, -10){\small {$ \mathcal{K}_1 $}};
	
	\end{tikzpicture}\hspace{2mm}
	\begin{tikzpicture}[scale=0.22]
	\draw[edge, thin](10,0)--(25,0)--(28,-2);
	\node[ver] () at (10,-1){\scriptsize ${7} $};
	\node[ver] () at (15,-1){\scriptsize ${6} $};
	\node[ver] () at (20,-1){\scriptsize ${5} $};
	\node[ver] () at (25,-1){\scriptsize ${14} $};
	\node[ver] () at (28,-3){\scriptsize ${8} $};
	
	\draw[edge, thin](10,5)--(7,2.5)--(4,5)--(7,7.5)--(10,5)--(25,5)--(28,3);
	\node[ver] () at (11,4){\scriptsize ${8} $};
	\node[ver] () at (16,4){\scriptsize ${1} $};
	\node[ver] () at (21,4){\scriptsize ${4} $};
	\node[ver] () at (25,6){\scriptsize ${19} $};
	\node[ver] () at (29,2.5){\scriptsize ${7} $};
	\node[ver] () at (6,7.5){\scriptsize ${9} $};
	\node[ver] () at (6,1.5){\scriptsize ${14} $};
	\node[ver] () at (3,5){\scriptsize ${15} $};
	
	\draw[edge, thin](25,5)--(28,7.5)--(31,5)--(28,3);
	\node[ver] () at (28,6.5){\scriptsize ${18} $};
	\node[ver] () at (32,5){\scriptsize ${20} $};
	
	\draw[edge, thin](7,7.5)--(10,10)--(25,10)--(28,7.5);
	\node[ver] () at (11,9){\scriptsize ${10} $};
	\node[ver] () at (16,9){\scriptsize ${2} $};
	\node[ver] () at (21,9){\scriptsize ${3} $};
	\node[ver] () at (25,9){\scriptsize ${16} $};
	
	\draw[edge, thin](10,10)--(7,12.5)--(7,17.5)--(10,15)--(25,15)--(28,17.5)--(28,12.5)--(25,10);
	\node[ver] () at (6,12.5){\scriptsize ${16} $};
	\node[ver] () at (6,17.5){\scriptsize ${18} $};
	\node[ver] () at (11,14){\scriptsize ${13} $};
	\node[ver] () at (16,14){\scriptsize ${11} $};
	\node[ver] () at (21,14){\scriptsize ${12} $};
	\node[ver] () at (26,14){\scriptsize ${17} $};
	\node[ver] () at (29,17.5){\scriptsize ${9} $};
	\node[ver] () at (29,12.5){\scriptsize ${10} $};
	
	\draw[edge, thin](7,17.5)--(7,22.5)--(10,20)--(12.5,22.5)--(15,20)--(20,20)--(22.5,22.5)--(25,20)--(28,22.5)--(28,17.5);
	\node[ver] () at (6,22.5){\scriptsize ${20} $};
	\node[ver] () at (11,19){\scriptsize ${15} $};
	\node[ver] () at (12.5,23.5){\scriptsize ${14} $};
	\node[ver] () at (16,19){\scriptsize ${5} $};
	\node[ver] () at (21,19){\scriptsize ${6} $};
	\node[ver] () at (22.5,23.5){\scriptsize ${7} $};
	\node[ver] () at (26,19){\scriptsize ${20} $};
	\node[ver] () at (29,22.5){\scriptsize ${15} $};
	
	\draw[edge, thin](15,0)--(15,20);
	\draw[edge, thin](20,0)--(20,20);
	\draw[edge, thin](10,10)--(10,20);
	\draw[edge, thin](25,10)--(25,20);
	\draw[edge, thin](10,0)--(10,5);
	\draw[edge, thin](25,0)--(25,5);
	\draw[edge, thin](28,-2)--(28,3);
	
	\node[ver] () at (17.5, -5){\small {$ \mathcal{K}_2 $}};
	\end{tikzpicture}
\end{figure}

\begin{figure}[H]
		\scriptsize
		\tikzstyle{ver}=[]
		\tikzstyle{vert}=[circle, draw, fill=black!100, inner sep=0pt, minimum width=4pt]
		\tikzstyle{vertex}=[circle, draw, fill=black!00, inner sep=0pt, minimum width=4pt]
		\tikzstyle{edge} = [draw,thick,-]
		\centering
	\tikzstyle{ver}=[]
	\tikzstyle{vert}=[circle, draw, fill=black!100, inner sep=0pt, minimum width=4pt]
	\tikzstyle{vertex}=[circle, draw, fill=black!00, inner sep=0pt, minimum width=4pt]
	\tikzstyle{edge} = [draw,thick,-]
	\centering
	\begin{tikzpicture}[scale=0.2]
	\draw[edge, thin](10,0)--(25,0)--(28,-2);
	\node[ver] () at (10,-1){\scriptsize ${7} $};
	\node[ver] () at (15,-1){\scriptsize ${6} $};
	\node[ver] () at (20,-1){\scriptsize ${5} $};
	\node[ver] () at (25,-1){\scriptsize ${14} $};
	\node[ver] () at (28,-3){\scriptsize ${8} $};
	
	\draw[edge, thin](10,5)--(7,2.5)--(4,5)--(7,7.5)--(10,5)--(25,5)--(28,3);
	\node[ver] () at (11,4){\scriptsize ${8} $};
	\node[ver] () at (16,4){\scriptsize ${1} $};
	\node[ver] () at (21,4){\scriptsize ${4} $};
	\node[ver] () at (25,6){\scriptsize ${19} $};
	\node[ver] () at (29,2.5){\scriptsize ${7} $};
	\node[ver] () at (6,7.5){\scriptsize ${9} $};
	\node[ver] () at (6,1.5){\scriptsize ${14} $};
	\node[ver] () at (3,4){\scriptsize ${15} $};
	
	\draw[edge, thin](25,5)--(28,7.5)--(31,5)--(28,3);
	\node[ver] () at (28,6.5){\scriptsize ${17} $};
	\node[ver] () at (32,4){\scriptsize ${20} $};
	
	\draw[edge, thin](7,7.5)--(10,10)--(25,10)--(28,7.5);
	\node[ver] () at (11,9){\scriptsize ${10} $};
	\node[ver] () at (16,9){\scriptsize ${2} $};
	\node[ver] () at (21,9){\scriptsize ${3} $};
	\node[ver] () at (25,9){\scriptsize ${16} $};
	
	\draw[edge, thin](10,10)--(7,12.5)--(7,17.5)--(10,15)--(25,15);
	\node[ver] () at (6,12.5){\scriptsize ${16} $};
	\node[ver] () at (6,17.5){\scriptsize ${18} $};
	\node[ver] () at (11,14){\scriptsize ${13} $};
	\node[ver] () at (16,14){\scriptsize ${11} $};
	\node[ver] () at (21,14){\scriptsize ${12} $};
	\node[ver] () at (26,14){\scriptsize ${18} $};
	
	\draw[edge, thin](7,17.5)--(7,22.5)--(10,20)--(12.5,22.5)--(15,20)--(20,20)--(22.5,22.5)--(25,20);
	\node[ver] () at (6,22.5){\scriptsize ${20} $};
	\node[ver] () at (11,19){\scriptsize ${15} $};
	\node[ver] () at (12.5,23.5){\scriptsize ${14} $};
	\node[ver] () at (16,19){\scriptsize ${5} $};
	\node[ver] () at (21,19){\scriptsize ${6} $};
	\node[ver] () at (22.5,23.5){\scriptsize ${7} $};
	\node[ver] () at (26,19){\scriptsize ${20} $};
	
	\draw[edge, thin](15,0)--(15,20);
	\draw[edge, thin](20,0)--(20,20);
	\draw[edge, thin](10,10)--(10,20);
	\draw[edge, thin](25,10)--(25,20);
	\draw[edge, thin](10,0)--(10,5);
	\draw[edge, thin](25,0)--(25,5);
	\draw[edge, thin](28,-2)--(28,3);
	
	\draw[edge, thin](25,10)--(28,12.5)--(31,9.8)--(28,7.5);
	\node[ver] () at (28,13.5){\scriptsize ${10} $};
	\node[ver] () at (32,9.8){\scriptsize ${9} $};
	
	\draw[edge, thin](7,7.5)--(4,10)--(1,7.5)--(4,5);
	\node[ver] () at (2.5,7.5){\scriptsize ${20} $};
	\node[ver] () at (4,11){\scriptsize ${17} $};
	
	\node[ver] () at (17.5, -5){\small {$ \mathcal{K}_3 $}};
	\end{tikzpicture}
	\begin{tikzpicture}[scale=1.3]
		\draw[edge, thin] (0,0) -- (0.6,0.3) -- (1.2,0.1) -- (1.5,-0.5) -- (1.2,-1.1) -- (0.6,-1.3) -- (0,-1) -- (0,0);
		
		\node[ver] () at (0.1,-0.1){\scriptsize{${1} $} };
		\node[ver] () at (0.6,0.15){\scriptsize{${2} $} };
		\node[ver] () at (1.1,0){\scriptsize{${3} $} };
		\node[ver] () at (1.35,-0.5){\scriptsize{${4} $} };
		\node[ver] () at (1.1,-0.95){\scriptsize{${5} $} };
		\node[ver] () at (0.6,-1.15){\scriptsize{${6} $} };
		\node[ver] () at (0.15,-0.95){\scriptsize{${7} $} };

		\draw[edge, thin] (0,0) -- (-1,0) -- (-1,-1) -- (0,-1);
		
		\node[ver] () at (-0.9,-0.15){\scriptsize{${9} $} };
		\node[ver] () at (-0.9,-0.9){\scriptsize{${8} $} };

		\draw[edge, thin] (-1,0) -- (-1.6,0.3) -- (-2.2,0.1) -- (-2.5,-0.6) -- (-2.2,-1.1) -- (-1.6, -1.3) -- (-1,-1);
		
		\node[ver] () at (-1.6,0.15){\scriptsize{${12} $} };
		\node[ver] () at (-2.1,-0.0){\scriptsize{${13} $} };
		\node[ver] () at (-2.3,-0.6){\scriptsize{${14} $} };
		\node[ver] () at (-2.1,-1){\scriptsize{${15} $} };
		\node[ver] () at (-1.6,-1.15){\scriptsize{${16} $} };
		
		\draw[edge, thin] (-2.2,0.1) -- (-2.5,0.3) -- (-2.2,0.6);
		
		\node[ver] () at (-2.65,0.3){\scriptsize{${27} $} };
		\node[ver] () at (-2.95,-0.3){\scriptsize{${28} $} };
		\node[ver] () at (-2.95,-0.7){\scriptsize{${3} $} };
		\node[ver] () at (-2.5,-1.4){\scriptsize{${4} $} };
		
		\draw[edge, thin] (-2.5,0.3) -- (-2.8,-0.3) -- (-2.5,-0.6);
		\draw[edge, thin] (-2.8,-0.3) -- (-2.8,-0.7);
		\draw[edge, thin] (-2.5,-0.6) -- (-2.8,-0.7) -- (-2.45,-1.3) -- (-2.2,-1.1);

		\draw[edge, thin] (-1,-1) -- (-0.5,-1.5) -- (0,-1);
		
		\draw[edge, thin] (-0.5,-1.5) -- (0.2,-1.75) -- (0.6,-1.3);
		
		\draw[edge, thin] (-0.5,-1.5) -- (-1.2,-1.75) -- (-1.6,-1.3);
		
		\draw[edge, thin] (-1.2,-1.75) -- (-1.2, -2.2) -- (-0.9,-2.6) -- (-0.1,-2.6) -- (0.2,-2.2) -- (0.2, -1.75);
		
		\node[ver] () at (-0.5,-1.65){\scriptsize{${29} $} };
		\node[ver] () at (-1.4,-1.75){\scriptsize{${35} $} };
		\node[ver] () at (-1.4,-2.2){\scriptsize{${34} $} };
		\node[ver] () at (-0.9,-2.7){\scriptsize{${33} $} };
		\node[ver] () at (-0.1,-2.7){\scriptsize{${32} $} };
		\node[ver] () at (0.4,-2.2){\scriptsize{${31} $} };
		\node[ver] () at (0.4,-1.75){\scriptsize{${30} $} };
		
		\draw[edge, thin] (-1,0) -- (-0.5,0.5) -- (0,0);
		
		\draw[edge, thin] (-0.5,0.5) -- (-1.2,0.75) -- (-1.2, 1.2) -- (-0.9,1.6) -- (-0.1,1.6) -- (0.2,1.2) -- (0.2, 0.75) -- (-0.5,0.5);
		
		\node[ver] () at (-0.5,0.65){\scriptsize{${10} $} };
		\node[ver] () at (-1,0.8){\scriptsize{${21} $} };
		\node[ver] () at (-1,1.2){\scriptsize{${20} $} };
		\node[ver] () at (-0.8,1.45){\scriptsize{${19} $} };
		\node[ver] () at (-0.2,1.45){\scriptsize{${18} $} };
		\node[ver] () at (0,1.15){\scriptsize{${17} $} };
		\node[ver] () at (0,0.8){\scriptsize{${11} $} };
		\node[ver] () at (-1.45,1.65){\scriptsize{${5} $} };
		\node[ver] () at (-1.25,2){\scriptsize{${4} $} };
		\node[ver] () at (-0.9,2.2){\scriptsize{${15} $} };
		\node[ver] () at (-0.1,2.2){\scriptsize{${16} $} };
		\node[ver] () at (0.25,2){\scriptsize{${35} $} };
		\node[ver] () at (0.45,1.7){\scriptsize{${34} $} };
		
		\draw[edge, thin] (-1.2, 0.75) -- (-1.7,0.75) -- (-1.6,0.3) -- (-1.2,0.75);
		
		\draw[edge, thin] (0.2, 0.75) -- (0.7,0.75) -- (0.6,0.3) -- (0.2,0.75);
		
		\draw[edge, thin] (-1.7,0.75) -- (-1.7,1.2) -- (-1.2,1.2);
		
		\draw[edge, thin] (0.7,0.75) -- (0.7,1.2) -- (0.2,1.2);
		
		\draw[edge, thin] (-1.7,0.75) -- (-2.2,0.6) -- (-2.2,0.1);
		
		\draw[edge, thin] (0.7,0.75) -- (1.2,0.6) -- (1.2,0.1);
		
		\draw[edge, thin] (-2.2, 0.6) -- (-2.7, 0.6) -- (-3.2, 1.1) -- (-2.8,1.6) -- (-2.2,1.6) -- (-1.7, 1.2);

		\draw[edge, thin] (1.2, 0.6) -- (1.7, 0.6) -- (2.2, 1.1) -- (1.8,1.6) -- (1.2,1.6) -- (0.7, 1.2); 
		
		\node[ver] () at (1.25,0.75){\scriptsize{${28} $} };
		\node[ver] () at (1.65,0.75){\scriptsize{${27} $} };
		\node[ver] () at (2,1.1){\scriptsize{${26} $} };
		\node[ver] () at (1.75,1.45){\scriptsize{${25} $} };
		\node[ver] () at (1.25,1.45){\scriptsize{${24} $} };
		\node[ver] () at (0.85,1.15){\scriptsize{${23} $} };
		\node[ver] () at (0.85,0.85){\scriptsize{${22} $} };
		\node[ver] () at (0.85,2){\scriptsize{${33} $} };
		\node[ver] () at (2.05,1.9){\scriptsize{${31} $} };
		\node[ver] () at (2.6,1.35){\scriptsize{${32} $} };
		\node[ver] () at (2.6,0.9){\scriptsize{${41} $} };
		\node[ver] () at (1.9,0.25){\scriptsize{${42} $} };
		
		\node[ver] () at (-2.25,0.75){\scriptsize{${42} $} };
		\node[ver] () at (-2.65,0.75){\scriptsize{${41} $} };
		\node[ver] () at (-3,1.1){\scriptsize{${40} $} };
		\node[ver] () at (-2.75,1.45){\scriptsize{${39} $} };
		\node[ver] () at (-2.25,1.45){\scriptsize{${38} $} };
		\node[ver] () at (-1.85,1.15){\scriptsize{${37} $} };
		\node[ver] () at (-1.85,0.85){\scriptsize{${36} $} };
		\node[ver] () at (-1.85,2){\scriptsize{${6} $} };
		\node[ver] () at (-3.2,1.9){\scriptsize{${25} $} };
		\node[ver] () at (-3.6,1.35){\scriptsize{${24} $} };
		\node[ver] () at (-3.6,0.9){\scriptsize{${33} $} };
		\node[ver] () at (-3,0.25){\scriptsize{${32} $} };
		\node[ver] () at (-2.75,2.1){\scriptsize{${31} $} };
		\node[ver] () at (-2.25,2.1){\scriptsize{${30} $} };

		\draw[edge, thin] (-1.7, 1.2) -- (-1.45, 1.5) -- (-1.2,1.2);

		\draw[edge, thin] (-1.45,1.5) -- (-1.9,1.9) -- (-2.2,1.6);
		
		\draw[edge, thin] (-2.2,1.6) -- (-2.2,2) -- (-2.8,2) -- (-2.8,1.6);
		\draw[edge, thin] (-1.9,1.9) -- (-2.2,2);
		
		\draw[edge, thin] (-2.8,1.6) -- (-3.1,1.8) -- (-3.45,1.4) -- (-3.2,1.1);
		\draw[edge, thin] (-2.8,2) -- (-3.1,1.8);
		
		\draw[edge, thin] (-3.2,1.1) -- (-3.45, 0.9) -- (-2.9,0.35) -- (-2.7,0.6);
		\draw[edge, thin] (-3.45,1.4) -- (-3.45,0.9);

		\draw[edge, thin] (0.45,1.5) -- (0.9,1.9) -- (1.2,1.6);
		
		
		\draw[edge, thin] (1.8,1.6) -- (2.1,1.8) -- (2.45,1.4) -- (2.2,1.1);
		
		\draw[edge, thin] (2.2,1.1) -- (2.45, 0.9) -- (1.9,0.35) -- (1.7,0.6);
		\draw[edge, thin] (2.45,1.4) -- (2.45,0.9);

		\draw[edge, thin] (-1.45,1.5) -- (-1.2,1.9) -- (-0.9,1.6);
		
		\draw[edge, thin] (0.7, 1.2) -- (0.45, 1.5) -- (0.2,1.2);
		
		\draw[edge, thin] (0.45,1.5) -- (0.2,1.9) -- (-0.1,1.6);
		
		\draw[edge, thin] (-0.9,1.6) -- (-0.9,2.1) -- (-0.1,2.1) -- (-0.1,1.6);
		
		\draw[edge, thin] (-1.2,1.9) -- (-0.9,2.1) -- (-0.1,2.1) -- (0.2,1.9);
		
	\node[ver] () at (1.5, -2){\small {$ \mathcal{K}_{4} $}};	
		\end{tikzpicture}
	\label{exm1}

	\begin{center}
		\psscalebox{0.9 0.9} 
		{
			\begin{pspicture}(6.5,23.455)(10.49,31.205)
			\pspolygon[linecolor=black, linewidth=0.02](4.52,27.175)(3.66,26.735)(3.46,25.935)(4.0,25.235)(5.04,25.255)(5.52,25.875)(5.42,26.655)(5.42,26.635)
			\psline[linecolor=black, linewidth=0.02](4.52,27.175)(4.6,27.795)(5.3,28.215)(6.02,27.795)(6.04,27.075)(5.44,26.615)
			\psline[linecolor=black, linewidth=0.02](6.0,27.075)(6.68,26.735)(6.74,25.955)(6.28,25.635)(5.52,25.895)
			\psline[linecolor=black, linewidth=0.02](6.3,25.655)(6.46,25.075)(6.1,24.575)(5.34,24.655)
			\psline[linecolor=black, linewidth=0.02](5.06,25.235)(5.36,24.635)(5.02,23.995)(4.14,23.975)(3.74,24.575)(4.02,25.255)
			\psline[linecolor=black, linewidth=0.02](3.72,24.575)(2.88,24.455)(2.42,25.075)(2.64,25.775)(3.48,25.955)
			\psline[linecolor=black, linewidth=0.02](2.62,25.775)(2.1,26.335)(2.36,27.115)(3.08,27.295)(3.66,26.755)
			\psline[linecolor=black, linewidth=0.02](3.04,27.315)(3.18,28.055)(4.0,28.355)(4.6,27.815)
			\rput[bl](4.38,26.735){2}
			\rput[bl](3.78,26.455){1}
			\rput[bl](3.66,25.895){7}
			\rput[bl](4.04,25.435){6}
			\rput[bl](4.92,25.435){5}
			\rput[bl](5.14,25.875){4}
			\rput[bl](5.06,26.355){3}
			\rput[bl](4.62,27.495){11}
			\rput[bl](5.1,27.755){17}
			\rput[bl](5.54,27.535){22}
			\rput[bl](5.54,27.035){28}
			\rput[bl](6.18,26.455){27}
			\rput[bl](5.84,25.335){15}
			\rput[bl](6.16,25.895){14}
			\rput[bl](6.42,24.795){16}
			\rput[bl](5.88,24.195){19}
			\rput[bl](5.32,24.175){20}
			\rput[bl](4.88,23.615){21}
			\rput[bl](3.96,23.615){37}
			\rput[bl](3.38,24.135){38}
			\rput[bl](2.84,24.575){39}
			\rput[bl](2.14,24.515){31}
			\rput[bl](2.12,25.535){30}
			\rput[bl](2.34,26.235){29}
			\rput[bl](2.42,26.795){8}
			\rput[bl](2.94,26.855){9}
			\rput[bl](3.22,27.755){12}
			\rput[bl](3.78,27.895){10}
			\psline[linecolor=black, linewidth=0.02](6.0,27.815)(6.7,28.495)(7.78,28.595)(8.1,27.855)(7.58,27.055)(6.68,26.715)
			\psline[linecolor=black, linewidth=0.02](4.02,28.335)(3.86,29.135)(4.24,29.935)(5.2,29.915)(5.56,29.015)(5.3,28.175)
			\psline[linecolor=black, linewidth=0.02](2.38,27.115)(1.54,27.415)(1.26,28.135)(1.64,28.735)(2.58,28.655)(3.18,28.075)
			\psline[linecolor=black, linewidth=0.02](6.7,28.515)(6.5,28.995)(6.8,29.575)(7.56,29.615)(7.96,29.035)(7.78,28.555)
			\psline[linecolor=black, linewidth=0.02](6.72,25.975)(7.36,25.615)(8.02,25.915)(8.16,26.435)
			\rput[bl](4.02,29.015){21}
			\rput[bl](4.2,29.575){20}
			\rput[bl](4.8,29.535){19}
			\rput[bl](4.98,28.935){18}
			\rput[bl](4.84,28.215){17}
			\rput[bl](6.62,28.135){23}
			\rput[bl](7.38,28.175){24}
			\rput[bl](7.66,27.715){25}
			\rput[bl](7.16,27.055){26}
			\rput[bl](2.18,28.335){13}
			\rput[bl](1.2,28.775){14}
			\rput[bl](0.78,28.055){15}
			\rput[bl](0.96,27.395){16}
			\rput[bl](7.22,25.295){13}
			\rput[bl](7.96,25.415){36}
			\rput[bl](8.3,26.075){42}
			\rput[bl](7.98,29.195){40}
			\rput[bl](7.48,29.695){41}
			\rput[bl](6.56,29.715){33}
			\rput[bl](6.7,28.815){34}
			\psline[linecolor=black, linewidth=0.02](1.54,27.415)(0.94,27.115)(0.9,26.475)(1.42,26.075)(2.1,26.315)
			\rput[bl](0.46,26.975){19}
			\rput[bl](0.46,26.235){18}
			\rput[bl](1.44,25.735){35}
			\psline[linecolor=black, linewidth=0.02](2.56,28.675)(2.76,29.395)(3.4,29.535)(3.86,29.115)
			\psline[linecolor=black, linewidth=0.02](5.54,29.015)(6.12,29.335)(6.5,28.995)
			\rput[bl](5.84,29.415){35}
			\rput[bl](3.12,29.095){37}
			\rput[bl](2.18,29.135){36}
			\psline[linecolor=black, linewidth=0.02](2.76,29.395)(2.2,29.815)(2.18,30.535)(2.74,30.995)(3.46,30.815)(3.7,30.155)(3.42,29.535)
			\rput[bl](1.76,29.555){42}
			\rput[bl](1.7,30.395){41}
			\rput[bl](2.48,31.095){40}
			\rput[bl](3.6,30.815){39}
			\rput[bl](3.76,30.015){38}
			\psline[linecolor=black, linewidth=0.02](1.44,26.095)(1.06,25.595)(1.22,24.955)(1.76,24.735)(2.42,25.075)
			\rput[bl](0.58,25.455){34}
			\psline[linecolor=black, linewidth=0.02](1.74,24.755)(1.56,24.135)(2.04,23.675)(2.66,23.835)(2.88,24.455)
			\rput[bl](1.52,24.895){32}
			\rput[bl](1.26,23.655){25}
			\rput[bl](1.66,23.295){24}
			\rput[bl](2.62,23.455){40}
			\psline[linecolor=black, linewidth=0.02](1.2,24.975)(0.56,24.675)(0.42,24.095)(0.94,23.855)(1.58,24.135)
			\rput[bl](0.14,24.535){41}
			\rput[bl](0.0,23.855){42}
			\rput[bl](0.66,23.515){26}
			\rput[bl](0.96,24.535){33}
			\psline[linecolor=black, linewidth=0.02](7.56,27.055)(8.16,26.435)
			\rput[bl](7,24.5){\large{$ \mathcal{K}_{5} $}}
			\end{pspicture}
		}
		
		\psscalebox{1 1} 
	{
		\begin{pspicture}(-7.5,-7.165)(6.21,-10.91)
		\pspolygon[linecolor=black, linewidth=0.02](2.1,-3.425)(1.6,-3.905)(1.6,-4.685)(2.08,-5.185)(2.92,-5.185)(3.36,-4.685)(3.34,-3.945)(2.92,-3.425)
		\psline[linecolor=black, linewidth=0.02](2.08,-3.425)(2.1,-3.025)(2.92,-3.025)(2.92,-3.425)
		\psline[linecolor=black, linewidth=0.02](3.32,-3.925)(3.8,-3.945)(3.8,-4.685)(3.34,-4.685)
		\psline[linecolor=black, linewidth=0.02](2.08,-5.165)(2.08,-5.605)(2.92,-5.625)(2.92,-5.185)
		\psline[linecolor=black, linewidth=0.02](1.6,-3.885)(1.18,-3.885)(1.18,-4.685)(1.6,-4.685)
		\psline[linecolor=black, linewidth=0.02](2.9,-3.025)(3.2,-2.585)(3.66,-2.445)(4.2,-2.545)(4.5,-2.905)(4.5,-3.405)(4.24,-3.785)(3.78,-3.945)
		\psline[linecolor=black, linewidth=0.02](2.9,-5.605)(3.14,-6.025)(3.6,-6.245)(4.12,-6.185)(4.5,-5.785)(4.52,-5.225)(4.18,-4.805)(3.78,-4.685)
		\psline[linecolor=black, linewidth=0.02](1.2,-3.885)(0.76,-3.665)(0.58,-3.225)(0.58,-2.785)(0.96,-2.405)(1.54,-2.385)(1.92,-2.625)(2.1,-3.025)
		\psline[linecolor=black, linewidth=0.02](1.2,-4.685)(0.76,-4.865)(0.56,-5.265)(0.54,-5.705)(0.84,-6.105)(1.4,-6.245)(1.86,-6.005)(2.08,-5.585)
		\psline[linecolor=black, linewidth=0.02](3.64,-2.445)(3.4,-2.025)(3.42,-1.585)(3.72,-1.245)(4.28,-1.245)(4.6,-1.625)(4.62,-2.165)(4.2,-2.545)
		\psline[linecolor=black, linewidth=0.02](0.96,-2.405)(0.64,-2.005)(0.62,-1.525)(1.0,-1.145)(1.58,-1.165)(1.88,-1.485)(1.88,-2.025)(1.52,-2.385)
		\psline[linecolor=black, linewidth=0.02](4.5,-5.225)(4.88,-4.905)(5.42,-4.925)(5.76,-5.265)(5.76,-5.785)(5.42,-6.145)(4.84,-6.145)(4.48,-5.765)
		\psline[linecolor=black, linewidth=0.02](4.6,-2.185)(4.92,-2.565)(4.5,-2.905)
		\psline[linecolor=black, linewidth=0.02](3.38,-2.025)(2.92,-2.205)(3.18,-2.605)
		\psline[linecolor=black, linewidth=0.02](1.86,-2.045)(2.24,-2.225)(1.92,-2.645)
		\psline[linecolor=black, linewidth=0.02](0.64,-1.985)(0.24,-2.365)(0.58,-2.785)
		\psline[linecolor=black, linewidth=0.02](4.82,-6.145)(4.46,-6.505)(4.12,-6.165)
		\psline[linecolor=black, linewidth=0.02](4.86,-4.925)(4.56,-4.525)(4.18,-4.805)
		\psline[linecolor=black, linewidth=0.02](1.84,-5.985)(1.86,-6.425)(2.2,-6.825)(2.78,-6.845)(3.12,-6.445)(3.12,-6.005)
		\rput[bl](2.08,-3.785){\scriptsize{26}}
		\rput[bl](2.74,-3.805){\scriptsize{1}}
		\rput[bl](2.86,-4.165){\scriptsize{25}}
		\rput[bl](3.0,-4.745){\scriptsize{6}}
		\rput[bl](2.62,-5.025){\scriptsize{35}}
		\rput[bl](1.72,-5.345){\scriptsize{5}}
		\rput[bl](1.72,-4.725){\scriptsize{49}}
		\rput[bl](1.72,-4.085){\scriptsize{4}}
		\rput[bl](3.64,-3.825){\scriptsize{8}}
		\rput[bl](4.22,-4.125){\scriptsize{24}}
		\rput[bl](4.56,-3.585){\scriptsize{9}}
		\rput[bl](4.58,-3.165){\scriptsize{23}}
		\rput[bl](3.92,-2.405){\scriptsize{10}}
		\rput[bl](3.5,-2.805){\scriptsize{22}}
		\rput[bl](3.16,-2.925){\scriptsize{2}}
		\rput[bl](2.6,-2.945){\scriptsize{21}}
		\rput[bl](1.82,-3.225){\scriptsize{3}}
		\rput[bl](1.48,-2.925){\scriptsize{30}}
		\rput[bl](1.16,-2.285){\scriptsize{14}}
		\rput[bl](0.82,-2.765){\scriptsize{29}}
		\rput[bl](0.12,-2.985){\scriptsize{15}}
		\rput[bl](0.16,-3.525){\scriptsize{28}}
		\rput[bl](0.42,-3.985){\scriptsize{18}}
		\rput[bl](1.1,-3.785){\scriptsize{27}}
		\rput[bl](1.12,-5.025){\scriptsize{16}}
		\rput[bl](0.34,-4.925){\scriptsize{39}}
		\rput[bl](0.12,-5.465){\scriptsize{17}}
		\rput[bl](0.14,-6.005){\scriptsize{38}}
		\rput[bl](0.48,-6.445){\scriptsize{19}}
		\rput[bl](1.2,-6.585){\scriptsize{37}}
		\rput[bl](1.46,-5.985){\scriptsize{12}}
		\rput[bl](1.52,-6.765){\scriptsize{30}}
		\rput[bl](2.0,-7.145){\scriptsize{3}}
		\rput[bl](2.74,-7.165){\scriptsize{21}}
		\rput[bl](3.14,-6.745){\scriptsize{2}}
		\rput[bl](2.96,-5.745){\scriptsize{11}}
		\rput[bl](2.66,-6.265){\scriptsize{34}}
		\rput[bl](3.4,-6.145){\scriptsize{13}}
		\rput[bl](3.82,-6.545){\scriptsize{33}}
		\rput[bl](4.3,-6.865){\scriptsize{18}}
		\rput[bl](4.78,-6.565){\scriptsize{28}}
		\rput[bl](5.36,-6.485){\scriptsize{15}}
		\rput[bl](5.86,-6.005){\scriptsize{9}}
		\rput[bl](5.82,-5.465){\scriptsize{23}}
		\rput[bl](5.5,-4.965){\scriptsize{17}}
		\rput[bl](4.9,-4.785){\scriptsize{38}}
		\rput[bl](4.56,-5.485){\scriptsize{32}}
		\rput[bl](4.04,-5.865){\scriptsize{20}}
		\rput[bl](4.42,-4.445){\scriptsize{19}}
		\rput[bl](4.04,-5.225){\scriptsize{7}}
		\rput[bl](3.58,-5.065){\scriptsize{31}}
		\rput[bl](4.96,-2.765){\scriptsize{17}}
		\rput[bl](4.72,-2.245){\scriptsize{39}}
		\rput[bl](4.72,-1.725){\scriptsize{16}}
		\rput[bl](4.24,-1.205){\scriptsize{27}}
		\rput[bl](3.54,-1.145){\scriptsize{18}}
		\rput[bl](2.96,-1.605){\scriptsize{33}}
		\rput[bl](3.54,-2.125){\scriptsize{13}}
		\rput[bl](2.6,-2.185){\scriptsize{34}}
		\rput[bl](2.22,-2.425){\scriptsize{12}}
		\rput[bl](1.46,-2.085){\scriptsize{37}}
		\rput[bl](2.0,-1.525){\scriptsize{19}}
		\rput[bl](1.68,-1.125){\scriptsize{7}}
		\rput[bl](0.76,-1.085){\scriptsize{31}}
		\rput[bl](0.36,-1.565){\scriptsize{8}}
		\rput[bl](0.16,-2.045){\scriptsize{24}}
		\rput[bl](0.0,-2.545){\scriptsize{9}}
		\rput[bl](3.5,-7.5){\large{$ \mathcal{K}_{7} $}}
		\end{pspicture}

	}
%
%
			\psscalebox{1.0 1.0} 
			{
				\begin{pspicture}(1.0,25.150082)(12.204464,30.55)
				\pspolygon[linecolor=black, linewidth=0.02](5.6961904,29.57)(4.9961905,29.21)(4.5761905,28.71)(4.3761907,28.09)(4.4761906,27.39)(4.8361907,26.77)(5.4961905,26.37)(6.3761907,26.33)(7.1161904,26.73)(7.5361905,27.31)(7.6561904,28.03)(7.4361906,28.73)(7.0561905,29.19)(6.4961905,29.51)(6.4961905,29.51)
				\pspolygon[linecolor=black, linewidth=0.02](5.038254,29.22)(4.8063498,29.592962)(5.4525356,29.932844)(5.6461906,29.540361)(5.6886783,29.561037)
				\pspolygon[linecolor=black, linewidth=0.02](6.4761906,29.515429)(6.6842985,29.86)(7.3061905,29.496286)(7.100515,29.19)(7.059704,29.17)
				\pspolygon[linecolor=black, linewidth=0.02](7.4161906,28.75)(7.8961906,28.87)(8.076191,28.19)(7.6361904,28.07)
				\pspolygon[linecolor=black, linewidth=0.02](7.5561905,27.33)(7.8761907,27.11)(7.4561906,26.51)(7.1161904,26.71)(7.1161904,26.73)
				\pspolygon[linecolor=black, linewidth=0.02](5.4761906,26.39)(5.4361906,25.95)(6.3361907,25.91)(6.3361907,26.33)
				\pspolygon[linecolor=black, linewidth=0.02](4.4761906,27.41)(4.0161905,27.19)(4.4161906,26.53)(4.8361907,26.77)(4.8161907,26.77)
				\pspolygon[linecolor=black, linewidth=0.02](4.5961905,28.73)(4.1761904,28.95)(3.9361906,28.27)(4.3761907,28.13)(4.3761907,28.13)
				\psline[linecolor=black, linewidth=0.02](5.4361906,29.93)(5.7761903,30.41)(6.4361906,30.37)(6.6761904,29.85)
				\psline[linecolor=black, linewidth=0.02](7.2961903,29.49)(7.8161907,29.63)(8.076191,29.29)(7.8761907,28.87)
				\psline[linecolor=black, linewidth=0.02](8.03619,28.19)(8.476191,27.87)(8.396191,27.27)(7.8761907,27.09)
				\psline[linecolor=black, linewidth=0.02](6.3361907,25.93)(6.7961903,25.67)(7.4161906,25.99)(7.4561906,26.51)
				\psline[linecolor=black, linewidth=0.02](4.4161906,26.55)(4.3561907,26.05)(4.9361906,25.67)(5.4361906,25.95)
				\psline[linecolor=black, linewidth=0.02](4.0361905,27.19)(3.5561905,27.37)(3.4961905,27.95)(3.9361906,28.25)
				\psline[linecolor=black, linewidth=0.02](4.1761904,28.95)(4.0361905,29.37)(4.3161907,29.69)(4.8161907,29.61)
				\rput[bl](5.6961904,29.27){\scriptsize{1}}
				\rput[bl](5.0161905,28.91){\scriptsize{$ 43 $}}
				\rput[bl](4.7561903,28.51){\scriptsize{7}}
				\rput[bl](4.5361905,28.01){\scriptsize{44}}
				\rput[bl](4.5561905,27.39){\scriptsize{6}}
				\rput[bl](4.8761907,26.77){\scriptsize{45}}
				\rput[bl](5.5561905,26.47){\scriptsize{5}}
				\rput[bl](6.1761904,26.43){\scriptsize{46}}
				\rput[bl](6.1961904,29.23){\scriptsize{49}}
				\rput[bl](6.7761903,28.87){\scriptsize{2}}
				\rput[bl](7.0361905,28.51){\scriptsize{48}}
				\rput[bl](7.3761907,27.93){\scriptsize{3}}
				\rput[bl](7.1361904,27.27){\scriptsize{47}}
				\rput[bl](6.9361906,26.77){\scriptsize{4}}
				\rput[bl](5.5761905,29.83){\scriptsize{9}}
				\rput[bl](6.3161907,29.83){\scriptsize{62}}
				\rput[bl](4.5761905,29.29){\scriptsize{84}}
				\rput[bl](7.3961906,29.27){\scriptsize{11}}
				\psline[linecolor=black, linewidth=0.02](6.4361906,30.37)(6.6161904,30.99)(6.9161906,31.57)(7.3161907,32.03)(7.8561907,32.41)(8.516191,32.61)(9.016191,32.35)(9.156191,31.73)(8.99619,31.01)(8.696191,30.49)(8.276191,29.99)(7.7961903,29.61)
				\psline[linecolor=black, linewidth=0.02](3.4961905,27.97)(2.8761904,27.89)(2.2361906,27.93)(1.6361905,28.15)(1.1161904,28.51)(0.8161905,28.97)(0.9961905,29.57)(1.6561905,29.89)(2.3761904,29.95)(3.1361904,29.85)(3.6161904,29.65)(4.0161905,29.37)(4.0161905,29.37)
				\psline[linecolor=black, linewidth=0.02](8.076191,29.27)(8.616191,29.61)(9.25619,29.85)(10.03619,29.91)(10.03619,29.91)(10.75619,29.81)(11.21619,29.51)(11.276191,28.95)(10.91619,28.49)(10.29619,28.11)(9.71619,27.91)(9.0561905,27.83)(8.436191,27.91)(8.436191,27.91)
				\psline[linecolor=black, linewidth=0.02](4.2961903,29.69)(3.8761904,30.05)(3.5561905,30.49)(3.2761905,31.01)(3.1761904,31.65)(3.2961905,32.25)(3.7761905,32.59)(4.3961906,32.47)(4.9161906,32.15)(5.3161907,31.71)(5.6761904,31.05)(5.7961903,30.37)(5.7961903,30.37)
				\psline[linecolor=black, linewidth=0.02](5.6761904,31.03)(6.6161904,31.01)(6.6161904,31.01)
				\psline[linecolor=black, linewidth=0.02](5.3161907,31.69)(5.7161903,32.09)(6.5961905,32.05)(6.9161906,31.57)(6.8961906,31.55)
				\psline[linecolor=black, linewidth=0.02](5.7361903,32.11)(5.2161903,32.47)(4.9761906,32.91)(4.8961906,33.49)(5.1161904,34.07)(5.5161905,34.49)(6.0761905,34.69)(6.7161903,34.63)(7.1561904,34.33)(7.4561906,33.97)(7.5961905,33.31)(7.3961906,32.69)(7.0561905,32.29)(6.5561905,32.05)
				\psline[linecolor=black, linewidth=0.02](4.9161906,32.15)(5.2361903,32.47)
				\psline[linecolor=black, linewidth=0.02](7.0361905,32.29)(7.3161907,32.01)
				\psline[linecolor=black, linewidth=0.02](8.276191,29.99)(8.59619,29.57)
				\psline[linecolor=black, linewidth=0.02](3.5561905,29.69)(3.8561904,30.05)
				\rput[bl](6.6561904,30.27){\scriptsize{10}}
				\rput[bl](6.7161903,30.93){\scriptsize{56}}
				\rput[bl](6.9961905,31.43){\scriptsize{21}}
				\rput[bl](7.3361907,31.73){\scriptsize{75}}
				\rput[bl](7.7761903,32.09){\scriptsize{20}}
				\rput[bl](8.29619,32.71){\scriptsize{}58}
				\rput[bl](9.13619,32.37){\scriptsize{19}}
				\rput[bl](9.316191,31.63){\scriptsize{59}}
				\rput[bl](8.67619,31.09){\scriptsize{18}}
				\rput[bl](8.45619,30.61){\scriptsize{60}}
				\rput[bl](8.09619,30.15){\scriptsize{17}}
				\rput[bl](7.6561904,29.71){\scriptsize{61}}
				\rput[bl](8.016191,28.81){\scriptsize{69}}
				\rput[bl](8.156191,29.13){\scriptsize{22}}
				\rput[bl](8.5561905,29.31){\scriptsize{63}}
				\rput[bl](9.21619,29.55){\scriptsize{23}}
				\rput[bl](9.936191,29.57){\scriptsize{64}}
				\rput[bl](10.49619,29.51){\scriptsize{24}}
				\rput[bl](10.83619,29.29){\scriptsize{65}}
				\rput[bl](10.816191,28.89){\scriptsize{25}}
				\rput[bl](10.53619,28.51){\scriptsize{66}}
				\rput[bl](10.076191,28.17){\scriptsize{26}}
				\rput[bl](9.5561905,28.01){\scriptsize{67}}
				\rput[bl](8.896191,27.97){\scriptsize{27}}
				\rput[bl](8.49619,28.05){\scriptsize{68}}
				\rput[bl](8.156191,28.25){\scriptsize{28}}
				\rput[bl](8.356191,26.87){\scriptsize{14}}
				\rput[bl](7.9161906,26.81){\scriptsize{52}}
				\rput[bl](7.5561905,26.41){\scriptsize{15}}
				\rput[bl](7.4761906,25.85){\scriptsize{51}}
				\rput[bl](6.7361903,25.41){\scriptsize{19}}
				\rput[bl](6.2161903,25.53){\scriptsize{58}}
				\rput[bl](5.4161906,25.71){\scriptsize{20}}
				\rput[bl](4.8161907,25.45){\scriptsize{57}}
				\rput[bl](4.0361905,25.85){\scriptsize{37}}
				\rput[bl](4.0361905,26.37){\scriptsize{82}}
				\rput[bl](3.7161906,26.89){\scriptsize{38}}
				\rput[bl](3.6561904,27.39){\scriptsize{81}}
				\rput[bl](3.5361905,28.37){\scriptsize{30}}
				\rput[bl](3.3161905,28.07){\scriptsize{71}}
				\rput[bl](2.8361905,28.05){\scriptsize{31}}
				\rput[bl](2.2761905,28.09){\scriptsize{72}}
				\rput[bl](1.6961905,28.27){\scriptsize{32}}
				\rput[bl](1.2561905,28.53){\scriptsize{73}}
				\rput[bl](1.0361905,28.91){\scriptsize{33}}
				\rput[bl](1.0961905,29.37){\scriptsize{74}}
				\rput[bl](1.5561905,29.57){\scriptsize{34}}
				\rput[bl](2.1161904,29.59){\scriptsize{75}}
				\rput[bl](2.8361905,29.55){\scriptsize{35}}
				\rput[bl](3.2761905,29.35){\scriptsize{76}}
				\rput[bl](3.6161904,29.09){\scriptsize{29}}
				\rput[bl](3.7361906,28.75){\scriptsize{70}}
				\rput[bl](4.3961906,29.79){\scriptsize{8}}
				\rput[bl](3.9961905,30.13){\scriptsize{50}}
				\rput[bl](3.7161906,30.55){\scriptsize{16}}
				\rput[bl](3.4161904,31.01){\scriptsize{51}}
				\rput[bl](3.2961905,31.57){\scriptsize{15}}
				\rput[bl](3.3961904,32.05){\scriptsize{52}}
				\rput[bl](3.7561905,32.27){\scriptsize{14}}
				\rput[bl](4.2361903,32.15){\scriptsize{53}}
				\rput[bl](4.6161904,31.87){\scriptsize{13}}
				\rput[bl](4.9761906,31.41){\scriptsize{54}}
				\rput[bl](5.2561903,30.87){\scriptsize{12}}
				\rput[bl](5.3161907,30.31){\scriptsize{55}}
				\rput[bl](5.8161907,32.23){\scriptsize{36}}
				\rput[bl](5.3361907,32.55){\scriptsize{77}}
				\rput[bl](5.1161904,32.83){\scriptsize{42}}
				\rput[bl](4.5161905,33.37){\scriptsize{78}}
				\rput[bl](4.7561903,34.07){\scriptsize{41}}
				\rput[bl](5.1761904,34.49){\scriptsize{79}}
				\rput[bl](5.8561907,34.81){\scriptsize{40}}
				\rput[bl](6.6561904,34.75){\scriptsize{80}}
				\rput[bl](7.1761904,34.39){\scriptsize{39}}
				\rput[bl](7.5961905,33.99){\scriptsize{81}}
				\rput[bl](7.7561903,33.23){\scriptsize{38}}
				\rput[bl](7.0161905,32.73){\scriptsize{82}}
				\rput[bl](6.7361903,32.41){\scriptsize{37}}
				\rput[bl](6.3761907,32.23){\scriptsize{83}}
				\psline[linecolor=black, linewidth=0.02](9.0561905,27.83)(8.99619,27.23)(8.396191,27.27)
				\psline[linecolor=black, linewidth=0.02](9.67619,27.93)(9.87619,27.37)(10.476191,27.59)(10.276191,28.09)
				\psline[linecolor=black, linewidth=0.02](10.896191,28.51)(11.29619,28.15)(11.656191,28.61)(11.25619,28.95)
				\psline[linecolor=black, linewidth=0.02](10.736191,29.81)(10.99619,30.27)(11.476191,29.97)(11.21619,29.51)
				\psline[linecolor=black, linewidth=0.02](9.25619,29.87)(9.25619,30.31)(9.99619,30.37)(10.03619,29.89)
				\psline[linecolor=black, linewidth=0.02](9.99619,30.37)(10.276191,30.73)(10.856191,30.63)(10.99619,30.27)
				\psline[linecolor=black, linewidth=0.02](11.476191,29.97)(11.91619,29.63)(11.95619,28.99)(11.656191,28.61)
				\psline[linecolor=black, linewidth=0.02](8.978095,27.250834)(9.301905,26.946072)(9.720952,27.003214)(9.892381,27.384167)
				\psline[linecolor=black, linewidth=0.02](2.8638096,27.898453)(2.9019048,27.327024)(3.5495238,27.36512)
				\psline[linecolor=black, linewidth=0.02](2.2161906,27.955595)(2.1019049,27.441309)(2.292381,27.098452)(2.7114286,27.060358)(2.9019048,27.327024)
				\psline[linecolor=black, linewidth=0.02](1.6257143,28.16512)(1.5304762,27.688929)(1.1114286,27.479404)(0.7304762,27.74607)(0.7495238,28.222261)(1.1685715,28.488928)
				\psline[linecolor=black, linewidth=0.02](0.8066667,29.003214)(0.3495238,28.812738)(0.76857144,28.203215)
				\psline[linecolor=black, linewidth=0.02](3.1114285,29.860357)(3.035238,30.222261)(3.2638094,30.469881)(3.6257143,30.412739)
				\psline[linecolor=black, linewidth=0.02](2.4257143,29.955595)(2.4066668,30.298452)(3.0542858,30.203215)
				\psline[linecolor=black, linewidth=0.02](1.6257143,29.879404)(1.5304762,30.222261)(1.7209524,30.584167)(2.2161906,30.603214)(2.4257143,30.279406)
				\rput[bl](2.9923809,30.50988){\scriptsize{59}}
				\rput[bl](3.1457143,29.998453){\scriptsize{18}}
				\rput[bl](2.444762,30.374643){\scriptsize{}60}
				\rput[bl](2.0828571,30.736547){\scriptsize{17}}
				\rput[bl](1.5114286,30.64131){\scriptsize{63}}
				\rput[bl](1.1838095,30.14607){\scriptsize{23}}
				\rput[bl](9.073334,30.355595){\scriptsize{34}}
				\rput[bl](9.7,30.511786){\scriptsize{74}}
				\rput[bl](10.179048,30.809881){\scriptsize{33}}
				\rput[bl](10.86381,30.656548){\scriptsize{79}}
				\rput[bl](11.054286,30.279406){\scriptsize{40}}
				\rput[bl](11.587619,30.012737){\scriptsize{80}}
				\rput[bl](11.96381,29.55369){\scriptsize{39}}
				\rput[bl](12.025714,28.88893){\scriptsize{31}}
				\rput[bl](11.724762,28.391787){\scriptsize{72}}
				\rput[bl](11.165714,27.856548){\scriptsize{32}}
				\rput[bl](10.482857,27.307976){\scriptsize{78}}
				\rput[bl](9.930476,27.174643){\scriptsize{42}}
				\rput[bl](9.720952,26.736547){\scriptsize{77}}
				\rput[bl](9.244761,26.618452){\scriptsize{13}}
				\rput[bl](8.785714,26.844166){\scriptsize{53}}
				\rput[bl](2.5219047,27.316547){\scriptsize{39}}
				\rput[bl](2.6152382,26.809881){\scriptsize{80}}
				\rput[bl](2.117143,26.812738){\scriptsize{65}}
				\rput[bl](2.2161906,27.355595){\scriptsize{25}}
				\rput[bl](1.5304762,27.498453){\scriptsize{66}}
				\rput[bl](1.0152382,27.23369){\scriptsize{26}}
				\rput[bl](0.48285717,27.460358){\scriptsize{78}}
				\rput[bl](0.36476192,28.052738){\scriptsize{41}}
				\rput[bl](0.0,28.704166){\scriptsize{79}}
				
		\rput[bl](3.0,25.5){\large{$ \mathcal{K}_{6} $}}
				\end{pspicture}
			}

		\end{center}
	\end{figure}

\end{example}

\bigskip

A list of semi-equivelar maps on the surface of Euler genus $ 3 $ given in tabular form\,:

\begin{center}
\begin{tabular}{ |p{1cm}|p{2cm}|p{3cm}|p{2cm}|  }
 \hline
S. No. & Type & Maps & Ref.\\
 \hline\hline
 1 & $[4^3, 5^1] $ &  $ \mathcal{K}_{1} $,  $ \mathcal{K}_{2} $,  $ \mathcal{K}_{3} $ &\\
 \hline
 2 & $[3^1, 4^1, 7^1, 4^1] $ &  $ \mathcal{K}_{4} $ &\\
 \hline
 3 & $[6^2, 7^1] $ &  $ \mathcal{K}_{5} $ &\\
 \hline
 4 & $[4^1, 6^1, 14^1] $ &  $ \mathcal{K}_{6} $ &\\
 \hline
 5 & $[4^1, 8^1, 10^1] $ &  $ \mathcal{K}_{7} $ &\\
 \hline
6 & $[4^1, 6^1, 16^1] $ &    $ \mathcal{K}_{8} $,  $ \mathcal{K}_{9} $ & (see \cite{tu2018})\\
 \hline
7 &  $[3^1, 4^1, 8^1, 4^1] $    &  $ \mathcal{K}_{10} $,  $ \mathcal{K}_{11} $ & (see \cite{tu2018})\\
\hline
8 & $[6^2, 8^1] $ &  $ \mathcal{K}_{12} $,  $ \mathcal{K}_{13} $  &  (see \cite{tu2018}) \\
\hline
9 & $[3^1, 4^1, 3^1, 4^2] $ &  $ \mathcal{K}_{14} $ & (see \cite{bu2019}) \\
 \hline
10 &  $[3^5, 4^1] $ &    $ \mathcal{K}_{15} $,  $ \mathcal{K}_{16} $,  $ \mathcal{K}_{17} $ & (see \cite{utm2014})\\
\hline

\end{tabular}
\begin{center}
{\small Table 1}
\end{center}
\end{center}

\bigskip
This leads to the following \,:

\begin{theorem} \label{theo:sf3}
	There are at least 17 non-isomorphic semi-equivelar maps on the surface of Euler genus $ 3 $, listed in Table 1. None of these are vertex transitive.
\end{theorem}
{The proof of this theorem is follows by Theorem \ref{theo:sf2} and \cite{bu2019,utm2014,tu2018}.}

\section{Proofs} \label{sec:proofs-1}

Let  $F_1\mbox{-}\cdots\mbox{-}F_m\mbox{-}F_1 $ be the face-cycle of a vertex  $u $ in a map. Then  $F_i \cap F_j $ is either  $u $ or an edge through  $u $. Thus the face  $F_i $ must be of the form  $u_{i+1}\mbox{-}u\mbox{-}u_i\mbox{-}P_i\mbox{-}u_{i+1} $, where  $P_i = \emptyset $ or a path  $\& $  $P_i \cap P_j = \emptyset $ for  $i \neq j $. Here addition in the suffix is modulo  $m $. So,  $u_{1}\mbox{-}P_1\mbox{-}u_2\mbox{-}\cdots\mbox{-}u_m\mbox{-}P_m\mbox{-}u_1 $ is a cycle and said to be the {\em link-cycle} of  $u $. For a simplicial complex,  $P_i = \emptyset $ for all  $i $, and the link-cycle of a vertex is the link of that vertex.

A face in a map of the form  $u_1\mbox{-}u_2\mbox{-}\cdots\mbox{-}u_n\mbox{-}u_1 $ is also denoted by  $u_1u_2\cdots u_n $. The faces with 3, 4, \dots, 10 vertices are called {\em triangle}, {\em square}, \dots, {\em decagon} respectively.

We need the following technical proposition from \cite{dm2017} to prove Lemma  \ref{lem1} which follows it.

\begin{proposition} [Datta  $\& $ Maity] \label{prop}
If  $[p_1^{n_1}, \dots, p_k^{n_k}] $ satisfies any of the following three properties then  $[p_1^{n_1} $,  $\dots, p_k^{n_k}] $ can not be the type of any semi-equivelar map on a surface.
\begin{enumerate}[label=\roman*]
\item There exists  $i $ such that  $n_i=2 $,  $p_i $ is odd and  $p_j\neq p_i $ for all  $j\neq i $.

\item There exists  $i $ such that  $n_i=1 $,  $p_i $ is odd,  $p_j\neq p_i $ for all  $j\neq i $ and  $p_{i-1}\neq p_{i+1} $.

\item   $[p_1^{n_1} $,  $\dots, p_k^{n_k}] $ is of the form  $[p^1,q^m,p^1,r^n] $, where  $p $,  $q $,  $r $ are distinct and  $p $ is odd.
\end{enumerate}
(Here, addition in the subscripts are modulo  $k $.)
\end{proposition}

\begin{lemma} \label{lem1}
Let  $X $ be an  $n $-vertex  map on the surface of Euler char.  $\chi = -1 $ of type   $[p_1^{n_1}, \dots,  $  $p_{\ell}^{n_{\ell}}] $.  Then  $(n, [p_1^{n_1}, \dots, p_{\ell}^{n_{\ell}}])  $  $= (12, [3^5, 4^1]), (42, [3^4, 7^1]), $  $ (20, [3^2, $  $4^1, 3^1, 5^1]), (12, $  $[3^1, 4^1, 3^1, 4^2]), $  $ (24, [3^4, 8^1]), $  $  (42, [3^1, 4^1, 7^1, 4^1]), $  $(24, [3^1, 4^1, 8^1, 4^1]), (15, $  $[3^1, 5^3]) $,  $(12, [3^1, 6^1, $  $4^1, 6^1]) $,  $(21,[3^1, 7^1, 3^1, 7^1]), $   $(84, [4^1, $  $ 6^1, 14^1]), $  $(48, [4^1, 6^1, 16^1]), (40,[4^1, $  $8^1, 10^1]), (24, [6^2,  8^1]), $  $(42, [3^1, 14^2]), (42, [7^1, 6^2]), (20, $  $[4^3, $  $5^1]) $.
\end{lemma}

\begin{proof}
Let  $f_1, f_2 $ be the no. of edges and faces of  $X $ respectively. Let  $f_0 = n $. Let  $d $ be the degree of each vertex. Consider the  $k $-tuple  $(q_1^{m_1}, \dots, q_k^{m_k}) $, where  $3 \le q_1 < \dots <q_k $, for each  $i=1, \dots,k $,  $q_i = p_j $ for some  $j $,  $m_i = \sum_{p_j = q_i}n_j $. Let  $x_i $ be the no. of  $i $-gons in  $X $. So,  $\sum_i m_i = \sum_{j}n_j = d $ and  $3f_2 \le 2f_1 = f_0d $. Clearly, (the no. of  $q_i $-gons)  $\times $  $q_i $ =  $n \times m_i $ and hence  $f_2 = n \times (\frac{m_1}{q_1} + \dots + \frac{m_k}{q_k}) $. Since  $\chi = -1 $, we get  $-1 -f_0 = f_2 - f_1 \le (f_0 \times d) /3-f_1 = -(f_0 \times d)/6 $. Thus,  $(d-6)f_0 \le -6\chi $ and hence  $d \le 6 $ since  $f_0 \geq 7 $. So,  $d = 3, 4, 5 $ or  $6 $.

\begin{Case1} \textnormal{Assume  $d = 6 $. Then,  $f_1 = 3 \times f_0, \sum n_i = 6, 6f_0 = 2f_1 = \sum(i \times x_i) $. Therefore,  $f_2 = \sum x_i \le 2f_0 $ and  $f_2 = -1 -f_0 +f_1 = -1 -f_0 +3f_0 = 2f_0 -1 $. So,  $\sum x_i = 2f_0-1 $ and  $3\sum x_i+3 = 6f_0 = \sum ix_i $, i.e.,  $3 = x_4+2x_5+3x_6+\dots $. Hence,  $x_i = 0 $ for  $i \geq 7 $. So,  $(x_4, x_5, x_6) = (3, 0, 0), (1, 1, 0), (0, 0, 1) $.
If  $(x_4, x_5, x_6) = (3, 0, 0) $ then  $(q_1^{m_1}, \dots, q_k^{m_k}) = (3^{\ell_1}, 4^{\ell_2}) $ and  $\ell_1 + \ell_2 = 6 $. So,  $-1 = n - 6n/2 + n\ell_1/3 + n\ell_2/4 = -2n + n(\ell_1/3 + \ell_2 / 4) $, i.e,  $12 = 6n - n\ell_1 $, i.e.,  $12 = n(6 - \ell_1) $. Thus,  $\ell_1 = 5 $ since  $n \geq 7 $. Hence,  $[p_1^{n_1}, \dots, p_{\ell}^{n_{\ell}}] = [3^5, 4^1] $  $\& $  $n = 12 $.
 $(x_4, x_5, x_6) = {(1, 1, 0)} $ implies  $(q_1^{m_1}, \dots, q_k^{m_k}) = (3^{4}, 4^{1}, 5^1) $. This gives,  $n(X) = 60/13 $, which is not an integer. So,  $(x_4, x_5, x_6) \neq (1, 1, 0) $. Clearly,  $(x_4, x_5, x_6) \neq (0, 0, 1) $ as  $x_i \geq 3 $. }
\end{Case1}

\begin{Case1} \textnormal{Assume  $d = 5 $. Then,  $f_1 = 5f_0/2 ~\&~ 6f_1 -10 = 10f_2 $, i.e.,  $3 \times\sum i \times x_i -10 = 10 \sum x_i $. Therefore,  $\sum_{i \ge 4}(3 \times i -10) x_i -10 = x_3 $.}

\textnormal{Let   $(q_1^{m_1}, \dots, q_k^{m_k}) = (3^4, i^1) $. Then,  $4f_0 = 3x_3, f_0 = ix_i $. Hence,  $x_3 = (3i-10)x_i -10 $, i.e.,  $30i = (5i-30)f_0 $. So,  $f_0 = 6i/(i-6) $. Hence,  $(i, f_0, x_i, x_3) = (7, 42, 7, 56), (8, 24, 3, 32) $ as  $x_i \geq 3 $. So,  $(n, [p_1^{n_1}, \dots, p_{\ell}^{n_{\ell}}]) = (42, [3^4, 7^1]), (24, [3^4, 8^1]) $.}

\textnormal{Let   $(q_1^{m_1}, \dots, q_k^{m_k}) = (3^3, i^2) $. Then,  $f_0=x_3, 2f_0 = ix_i $. So,  $x_3  = (3i-10) \times x_i -10 $, i.e.,  $f_0 = 2i/(i-4) $. Hence,  $(i, f_0, x_i, x_3) = (5, 10, 4, 10) $ as  $x_i \geq 3 $. Clearly,  $(i, f_0, x_i, x_3) \neq (5, 10, 4, 10) $ since face-cycle of a vertex in  $X $ contains  $10 ~(=f_0) $ vertices, which is not possible as edge graph of  $X $ is not complete.}

\textnormal{Let   $(q_1^{m_1}, \dots, q_k^{m_k}) = (3^3, i^1, j^1) $,  $i < j $. Then,  $f_0=x_3, f_0 = ix_i, f_0 = jx_j $. So,  $x_3  = (3i-10) \times x_i+(3j-10) \times x_j -10 $, i.e.,  $f_0 = 2ij/(ij-2j-2i), i<j $. Hence,  $(i, j, f_0, x_i, x_j, x_3) = (4, 5, 20, 5, 4, 20) $ as  $x_i, x_j \geq 3 $. Thus,  $(n, [p_1^{n_1}, \dots, p_{\ell}^{n_{\ell}}]) = (20, [3^2, 4^1, 3^1, 5^1]) $ by Prop. \ref{prop}. }

\textnormal{Let   $(q_1^{m_1}, \dots, q_k^{m_k}) = (3^2, i^1, j^1, k^1) $,  $i < j < k $. By Prop. \ref{prop},  $[p_1^{n_1}, \dots, p_{\ell}^{n_{\ell}}] = [3^1, i^1, 3^1, $  $j^1, i^1], [3^1, i^1, 3^1, i^2] $.
Let  $[p_1^{n_1}, \dots, p_{\ell}^{n_{\ell}}] = [3^1, i^1, 3^1, j^1, i^1] $. Then,
 $2f_0=3x_3, 2f_0 = ix_i, f_0 = jx_j $. So,  $x_3  = (3i-10) \times x_i+(3j-10) \times x_j -10 $, i.e.,  $f_0 = 6ij/(5ij-12j-6i), i<j $. Hence,  $(i, j, f_0, x_3, x_i, x_j) = (3, 7, 42, 28, 28, 6), (3, 8, 24, 16, 16, 3), (4, 4, 12, 8, 6, 3) $ as  $x_i, x_j \geq 3 $. Thus,  $(n, [p_1^{n_1}, \dots, p_{\ell}^{n_{\ell}}]) = (42, [3^4, 7^1]), (24, [3^4, 8^1]), (12, [3^1, 4^1, 3^1, 4^2]) $.
Let  $[p_1^{n_1}, \dots, p_{\ell}^{n_{\ell}}] $  $= [3^1, i^1, 3^1, i^2] $. Then,
 $2f_0=3x_3, 3f_0 = ix_i $. So,  $x_3  = (3i-10) \times x_i -10 $, i.e.,  $f_0 = 6i/(5i-18) $. Hence,  $(i, f_0, x_i, x_3) = (4, 12, 9, 8) $. Thus,  $(n, [p_1^{n_1}, \dots, p_{\ell}^{n_{\ell}}]) = (12, [3^1, 4^1, 3^1, 4^2]) $.}

\textnormal{Let   $(q_1^{m_1}, \dots, q_k^{m_k}) = (3^1, i^1, j^1, k^1, \ell^1) $,  $i < j < k < \ell $. By Prop. \ref{prop},  $[p_1^{n_1}, \dots, p_{\ell}^{n_{\ell}}] = [3^1, i^1, j^1, k^1, i^1] $.
So,  $f_0=3x_3, 2f_0 = ix_i, f_0 = jx_j, f_0 = kx_k $. Therefore,  $x_3  = (3i-10) \times x_i+(3j-10) \times x_j+(3k-10) \times x_k -10 $, i.e.,  $f_0 = 6ijk/(7ijk-12jk-6ik-6ij) $. Hence,  $(i, j, k, f_0, x_3, x_i, x_j, x_k) = (4, 4, 3, 12, 4, 6, 3, 4), $  $(4, 3, 4, 12, 4, 6, 4, 3), $  $(3, 8, 3, 24, 8, 16, 3, 8), $  $(3, 7, 3, 42, 14, 28, 6, 14), (3, 3, 8, 24, 8, 16, 8, 3), (3, 3, 7, 42, 14, 28, 14, 6) $ as  $x_i, x_j, x_k \geq 3 $. Thus,  $(n, [p_1^{n_1}, \dots, p_{\ell}^{n_{\ell}}]) = (12, [3^1, 4^2, 3^1, 4^1]), (24, [3^4, 8^1]), (42, [3^4, 7^1]) $.}
\end{Case1}

\begin{Case1} \textnormal{Assume  $d = 4 $. Then,  $f_1 = 2f_0 ~\&~ f_1 -2f_2 = 2 $. Thus,  $\sum_{i \ge 5}(i -4) x_i = x_3 +4 $. }

\textnormal{If   $(q_1^{m_1}, \dots, q_k^{m_k}) = (3^3, i^1) $. Then,  $3f_0 = 3x_3, f_0 = ix_i $. Hence,  $4+x_3 = (i-4)x_i $, i.e.,  $f_0 = -i $. So,  $(q_1^{m_1}, \dots, q_k^{m_k}) \neq (3^3, i^1) $.}

\textnormal{Let   $(q_1^{m_1}, \dots, q_k^{m_k}) = (3^1, i^1, j^1, k^1) $. By Prop. \ref{prop},  $[p_1^{n_1}, \dots, p_{\ell}^{n_{\ell}}] = [3^1, i^1, j^1, i^1] $.
Let  $[p_1^{n_1}, \dots, p_{\ell}^{n_{\ell}}] = [3^1, i^1, j^1, i^1] $. Then,
 $f_0=3x_3, 2f_0 = ix_i, f_0 = jx_j $. So,  $4 + x_3  = (i-4) \times x_i+(j-4) \times x_j $, i.e.,  $f_0 = 3ij/(2ij-3i-6j) $. Hence,  $(i, j, f_0, x_3, x_i, x_j) = (4, 7, 42, 14, 21, 6), (4, 8, $  $24, 8, 12, 3), (5, 4, 60, 20, 24, 15), (5, 5, 15, 5, 6, 3), (6, 4, 12, 4, 4, 3), (7, 3, 21, 7, 6, 7), (8, 3, 12, 4, 3, $  $4) $ as  $x_i, x_j \geq 3 $. Thus,  $(n, [p_1^{n_1}, \dots, p_{\ell}^{n_{\ell}}]) = (42, [3^1, 4^1, 7^1, 4^1]), (24, [3^1, 4^1, 8^1, 4^1]), (60, [3^1, 5^1, $  $4^1, 5^1]), (15, [3^1, 5^3]), (12, [3^1, 6^1, 4^1, 6^1]), (21, [3^1, 7^1, 3^1, 7^1]), (12, [3^1, 8^1, 3^1, 8^1]) $. Observe that  $[p_1^{n_1}, \dots, p_{\ell}^{n_{\ell}}] \neq [3^1, 8^1, 3^1, 8^1] $ since face-cycle of a vertex contains  $15 $ vertices and  $15 > 12 $,  $\& $  $[p_1^{n_1}, \dots, p_{\ell}^{n_{\ell}}] \neq (60, [3^1, 5^1, $  $4^1, 5^1]) $ by Prop. \ref{prop}.}

\textnormal{Let   $(q_1^{m_1}, \dots, q_k^{m_k}) = (i^1, j^1, k^1, t^1) $. Then,  $f_0 = ix_i, f_0 = jx_j, f_0 = kx_k, f_0 = tx_t $. Thus,  $f_1 - 2f_2 = 2 $, i.e.,  $4 = (i-4)m_i + (j-4)m_j + (k-4)m_k + (t-4)m_t $, i.e.,  $f_0 = ijkt/(ijkt - jkt - ikt - ijt - ijk) $. Hence,  $(i, j, k, t, f_0) = (4, 4, 4, 5, 20) $. Therefore,  $(n, [p_1^{n_1}, \dots, p_{\ell}^{n_{\ell}}]) = (20, [4^3, 5^1]) $.}
\end{Case1}

\begin{Case1} \textnormal{Assume  $d = 3 $. Then,  $2f_1 = 3f_0 ~\&~ f_1 -3f_2 = 3 $. Thus,  $(i-6)m_i+(j-6)m_j+(k-6)m_k = 6 $.
Let   $(q_1^{m_1}, \dots, q_k^{m_k}) = (i^1, j^1, k^1) $. By Prop. \ref{prop},  $[p_1^{n_1}, \dots, p_{\ell}^{n_{\ell}}] = [i^1, j^1, k^1] $. So,  $f_0 = im_i, f_0 = jm_j, f_0 = km_k $. Thus,  $f_0 = 2ijk/(ijk-2ij-2ik-2jk) $. If  $i, j, k $ are even then  $(i, j, k, f_0, m_i, m_j, m_k) = (4, 6, 14, 84, 21, 14, 6), (4, 6, 16, 48, 12, 8, 3), (4, 8, 10, 40, 10, 5, 4), $  $(6, $  $6, 8, 24, 4, 4, 3) $.
Thus,  $(n, [p_1^{n_1}, \dots, p_{\ell}^{n_{\ell}}]) = (84, [4^1, 6^1, 14^1]),$ $(48, [4^1, 6^1, 16^1]),$ $(40, [4^1, 8^1,$ $10^1]), $  $(24, [6^2, 8^1]) $. If  $i, j $ or  $k $ is odd then by Prop. \ref{prop},  $[p_1^{n_1}, \dots, p_{\ell}^{n_{\ell}}] = [i^1, j^2] $ if  $i $ is odd. So,  $f_0 = ix_i, 2f_0 = jx_j $. Hence,  $6 = (i-6)m_i + (j-6)m_j $, i.e.,  $f_0 = 2ij/(ij-2j-4i) $. Thus,  $(i, j, f_0) = (3, 14, 42),  (7, 6, 42), (8, 6, 24), (9, 6, 18)$ using Prop. \ref{prop}.  Therefore,  $(n, [p_1^{n_1}, \dots, p_{\ell}^{n_{\ell}}]) = (42, [3^1, 14^2]), (42, [7^1, 6^2]), (24, [8^1, 6^2])$.
This completes the proof. }
\end{Case1}\vspace{-7mm}
\end{proof}

We recall (from \cite{DM2018}) some combinatorial versions of truncation and rectification on maps.

\begin{definition}\label{dfn1}
{\rm Let  $W $ be a semi-equivelar map on surface  $S $. Let  $V(W) = \{u_1, \dots, u_n\} $. Consider a new set (of nodes)  $V := \{v_{ij} \colon u_iu_j $ is an edge of  $W\} $. So, if  $v_{ij} \in V $ then  $i \neq j $  $\& $  $v_{ji} $ is also in  $V $.  Let  $E := \{v_{ij}v_{ji} \colon v_{ij} \in V\} \sqcup \{v_{ij}v_{ik} \colon u_j, u_k $  are in a face containing  $u_i, 1\le i \le n\} $. Then  $(V, E) $ is a graph on  $S $.  Thus  $(V, E) $ gives a map  $T(W) $ on  $S $. This map  $T(W) $ is said to be the} truncation {\rm of  $W $}.
\end{definition}

\begin{definition}\label{dfn2}
{\rm Let  $W $ be a semi-equivelar map on surface  $S $. Let  $V(W) = \{u_1, \dots, u_n\} $. Consider the graph  $(V, E) $, where  $V $ is the edge set  $E(W) $ of  $W $ and  $E := \{ef \colon e, f $ are two adjacent edges in a face of  $W\} $. Then  $(V, E) $ is a graph on  $S $.  Thus  $(V, E) $ gives a map say  $R(X) $ on  $S $, which is said to be the} rectification {\rm of  $W $}.
\end{definition}

\begin{proposition}\label{lem4}
Let  $X $  $\& $  $T(X) $ be as in Definition \ref{dfn1}. Then, if  $X $ is semi-equivelar of type  $[q^p] $  $( $resp.,  $[p^1, q^1, p^1, q^1]) $, then  $T(X) $ is also semi-equivelar  $\& $ of type  $[p^1, (2q)^2] $  $( $resp.,  $[4^1, (2p)^1, (2q)^1]) $.
\end{proposition}

\begin{proposition}\label{lem5}
Let  $X $  $\& $  $R(X) $ be as in Definition \ref{dfn2}. Then, if  $X $ is semi-equivelar of type  $[q^p] $  $( $resp.,  $[p^1, q^1, p^1, q^1]) $,  then  $R(X) $ is also semi-equivelar
and of type  $[p^1, q^1, p^1, q^1] $  $( $resp.,  $[4^1, p^1, 4^1, q^1]) $.
\end{proposition}

\begin{proposition} (\cite{tu2018}) \label{prop16}
Let  $X $ be a semi-equivelar map of type $Y$ on the surface of Euler char.  $-1 $. If $Y = [4^1, 6^1, 16^1] $ then $X \cong \cal{K}_{8}$ or $\cal{K}_{9}$. If $Y = [3^1, 4^1, 8^1, 4^1] $ then $X \cong \cal{K}_{10}$ or $\cal{K}_{11}$. If $Y = [6^2, 8^1] $ then $X \cong \cal{K}_{12}$ or $\cal{K}_{13}$.
\end{proposition}

\begin{proposition} (\cite{bu2019}) \label{prop14}
	Let  $X $ be a semi-equivelar map of type  $[3^1, 4^1, 3^1, 4^2] $ on the surface of Euler char.  $-1 $. Then, $X \cong \cal{K}_{14}$.
\end{proposition}

\begin{proposition} (\cite{utm2014}) \label{prop15}
Let  $X $ be a semi-equivelar map of type  $[3^5, 4^1]$ on the surface of Euler char.  $-1 $. Then, $X \cong \cal{K}_{15}, \cal{K}_{16}$ or $\cal{K}_{17}$.
\end{proposition}

\begin{proposition} (\cite{tu2018}) \label{prop20}
Let  $X $ be a semi-equivelar map of type  $[p_1^{n_1}, \dots, p_k^{n_k}] $ on the surface of Euler char.  $-1 $. Then,  $[p_1^{n_1}, \dots, p_k^{n_k}] \neq [3^4, 8^1], [3^1, 6^1, 4^1, 6^1], [3^2, 4^1, 3^1, 6^1],$ $[4^3, 6^1]$ $\&$ $[4^1, 8^1, 12^1]$.
\end{proposition}

\begin{proposition} (\cite{am2019})\label{prop19}
Let  $X $ be a semi-equivelar map of type  $[p_1^{n_1}, \dots, p_k^{n_k}] $ on the surface of Euler char.  $-1 $. Then,  $[p_1^{n_1}, \dots, p_k^{n_k}] \neq [3^1, 5^3]$.
\end{proposition}

\begin{lemma}\label{lem6}
	Let  $X $ be a semi-equivelar map of type  $[p_1^{n_1}, \dots, p_k^{n_k}] $ on the surface of Euler char.  $-1 $. Then,  $[p_1^{n_1}, \dots, p_k^{n_k}] \neq [3^2, 4^1, 3^1, 5^1]$.
\end{lemma}
\vspace{-5mm}
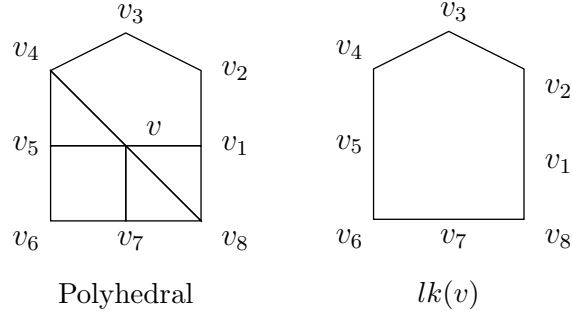
\begin{figure}[H]
	\begin{center}
		\begin{tikzpicture}[scale=1,line width=.5pt]
		\draw (0,0) node[anchor=south west]{\large{ $ v  $}}
		-- (1,0) node[anchor=west]{\large{ $ v_1  $}}
		-- (1,1) node[anchor=west]{\large{ $v_2 $}}
		-- (0,1.5) node[anchor=south]{\large{ $v_3 $}}
		-- (-1,1) node[anchor=south east]{\large{ $ v_4  $}}
		-- cycle;
		\draw (0,0)
		-- (-1,1)
		-- (-1,0) node[anchor=east]{\large{ $ v_5  $}}
		-- cycle;
		\draw (0,0)
		-- (-1,0)
		-- (-1,-1) node[anchor=north east]{\large{ $ v_6  $}}
		-- (0,-1) node[anchor=north]{\large{ $ v_7  $}}
		-- cycle;
		\draw (0,0)
		-- (0,-1)
		-- (1,-1) node[anchor=north west]{\large{ $ v_8  $}}
		-- cycle;
		\draw (0,0)
		-- (1,-1)
		-- (1,0)
		-- cycle;
		\node at (0,-2){Polyhedral};
		\end{tikzpicture}
		\hspace{0.5cm}
		\begin{tikzpicture}[scale=1,line width=.5pt]
			\draw (0,-1) node[anchor=north east]{\large{ $ v_6  $}}
			-- (1,-1) node[anchor=north]{\large{ $ v_7  $}}
			-- (2,-1) node[anchor=north west]{\large{ $ v_8  $}}
			-- (2,0) node[anchor=north west]{\large{ $ v_1  $}}
			-- (2,1) node[anchor=north west]{\large{ $ v_2  $}}
			-- (1,1.5) node[anchor=south]{\large{ $ v_3  $}}
			-- (0,1) node[anchor=south east]{\large{ $ v_4  $}}
			-- (0,0) node[anchor=east]{\large{ $ v_5  $}}
			-- cycle;
			\node at (1,-2){$ lk(v) $};
		\end{tikzpicture}
	\end{center}
	\caption{Example of  $ lk(v)  $ of type  $ [3^2,4^1,3^1,5^1]  $}
	\label{example_lk(v)_3(2)435}
\end{figure}
\begin{proof}
		Let  $X $ be a map of type  $[3^2,4^1,3^1,5^1] $ on the surface of  $\chi =-1 $. Let  $ V(X), F_3(X) $,  $F_4(X), F_5(X)  $ be \vx{}set, 3-gon face set, 4-gon face set and 5-gon face set respectively. Then  $ |V(X)|=20 $,  $|F_3(X)|=20 $,  $ |F_4(X)|=5 $ and  $ |F_5(X)|=4  $. Let  $V(X)=\{1,2,3,...,20\} $. We write  $lk(v) $ by the notation  $lk(v)=C_{8}([v_1,v_2,v_3,v_4],[v_5,v_6,v_7],v_8) $ where  {'$ [\ ] $' to indicate} $[v,v_1,v_2,v_3,v_4] $ form a 5-gon face,  $[v,v_{5},v_{6},v_{7}]  $ form a 4-gon face and  $ [v,v_4,v_5]  $,  $ [v,v_7,b_8]  $,  $ [v,v_1,v_8]  $ form 3-gon faces, see Fig. \ref{example_lk(v)_3(2)435}. Therefore always edges  $ [v_1,v_8]  $,  $ [v_7,v_8]  $ will adjacent to 4-gon and 5-gon respectively. {Without loss of generality}, let \lnk{\allowbreak 1}{2,3,4,5}{6,7,8}{9}, then  {\lnk{2}{3,4,5,1}{9,a,b}{c}}, \lnk{\allowbreak 9}{8,a_1,a_2,a_3}{a, b,2}{1}, \lnk{\allowbreak 8}{a_1,\allowbreak a_2,a_3,9}{1,6,7}{a_4}. Therefore without loss of generality, we can assume  $ (a_1,a_2,a_3)=(10,11,$ $12)  $,  $ (a,b,c)=(13,14,15)  $ and  $ a_4=16  $. Now \lnk{\allowbreak 15}{14,a_6,a_7,a_8}{a_9,a_{10},3}{2}, \lnk{\allowbreak 14}{a_6,a_7,a_8,15}{2,9,13}{a_5}, \lnk{\allowbreak 7}{a_{12},\allowbreak a_{13},a_{14},a_4}{8,1,6}{a_{11}} and \allowbreak \lnk{\allowbreak 13}{a_{17},a_{16},a_{15},a_5}{14,2,9}{12}. Let  $ \alpha=[1,2,3,4,5]  $,  $ \beta=[8,9,12,11,10]  $,  $ \gamma=[14,15,a_8,a_7,a_6]  $,  $ \delta=[13,a_5,a_{15},a_{16},a_{17}]  $,  $ \eta=[7,a_{12},a_{13},a_{14},4]  $. Then there are two cases, one is  $ F_5(X)=\{\alpha,\beta, \gamma=\eta,\delta\}  $ and other is  $ F_5(X)=\{\alpha,\beta,\gamma,\delta=\eta \}  $.

\smallskip
		
\begin{Case2} \textnormal{$F_5(X)=\{\alpha,\beta, \gamma=\eta,\delta\}  $ implies  $ \{a_6,a_7,a_8\}=\{7,16,17\}  $,  $ \{a_{12},a_{13},a_{14}\}=\{14,15,17\}  $ and  $ \{a_5,a_{15},a_{16},a_{17}\}=\{6,18,19,20\}  $. By considering  $ lk(7)  $, we see that  $ a_6\neq 7  $ and  $ a_5,a_{17}\neq 6  $, so with out loss of generality we can assume  $ a_5=18  $, $a_{17}=19 $, therefore $ a_{11}\in\{19,20\} $. If $a_{11}=19 $ then $a_{16}=6 $, $ a_{15}=20 $ and \lnk{\allowbreak 19}{6,20,18,13}{12,a_{18},\allowbreak a_{12}}{7}, which implies $ a_{12}\in\{15,17\} $. If $ a_{12}=17 $ then remaining 4-gons are $ [12,19,17,a_{18}] $, $ [3,15,a_9,a_{10}] $, $ [18,a_6,b_1,b_2] $ and we have $ [10,16] $ adjacent to 4-gon, which make a contradiction. Therefore $ a_{12}=15 $ and then $ a_{18}=3 $, $ (a_9,a_{10})=(12,19) $, $ (a_{13},a_{14})=(14,17) $, $ (a_6,a_7,a_8)=(17,16,7) $. Now \lnk{\allowbreak 3}{4,5,1,2}{15,19,12}{11}, \lnk{\allowbreak 12}{19,8,10,11}{3,15,19}{13} and remaining two 4-gons are $ [10,16,b_1,b_2] $ and $ [4,11,b_3,b_4] $ and $ \{b_1,b_2,b_3,b_4\}=\{5,17,18,20\} $. But none of $ b_3 $, $ b_4 $ is 5, therefore one of $ b_1 $, $ b_2 $ is 5. Now \lnk{\allowbreak 11}{12,9,8,10}{b_3,b_4,4}{3}, \lnk{\allowbreak 4}{5,1,2,3}{11,b_3,b_4}{a_{19}} which implies $ a_{19}\in\{10,16\} $, $ b_4\in\{11,17\} $. From here we see that $ b_4=17 $, $ a_{19}=19 $, $ b_1=5 $ which contradict to $ lk(16) $, therefore $ a_{11}=20 $. Now either $ a_6=a_{12} $ or $ a_6\neq a_{12} $. If $ a_6\neq a_{12} $ then $ [3,15,a_9,a_{10}] $, $ [18,a_6,b_1,b_2] $, $ [10,a_{12},b_3,b_4] $ are distinct 4-gons and we have $ [10,16] $, $ [12,19] $ are adjacent to 4-gons, which make contradiction. Therefore $ a_6=a_{12}=17 $ and then $ (a_{13},a_{14})=(14,15) $, $ (a_7,a_8)=(7,16) $ and remaining 4-gons are $ [3,15,a_9,a_{10}] $, $ [17,18,b_1,20] $, $ [10,16,b_2,b_3] $. Therefore $ \{a_9,a_{10}\}=\{12,19\} $, $ \{b_1,b_2,b_3\}=\{4,5,11\} $ which implies one of $ b_2 $, $ b_3 $ is 11. By considering $ lk(6) $, we get $ b_1=5 $, therefore $ \{b_2,b_3\}=\{4,11\} $ and $ \{a_{15},a_{16}\}={\{6,20\}} $. Now \lnk{\allowbreak 6}{18,13,19,20}{7,8,1}{5} and \lnk{18}{13,19,20,6}{5,20,17}{14}, this implies $ C(7,16,15,14,18,5,20)\in lk(17) $, a contradiction.}
\end{Case2}
	
\begin{Case2} \textnormal{$F_5(X)=\{\alpha,\beta,\gamma,\delta=\eta \} $ then one of $ a_6 $, $ a_7 $, $ a_8 $ is 6 and $ \{a_{12},a_{13},a_{14}\}=\{13,18,19\} $, $ \{a_6,a_7,a_8\}=\{6,17,20\} $ and $ \{a_{15},a_{16}\}=\{7,16\} $. This implies $ a_{13}=13 $, $ \{a_{12},a_{14}\}=\{18,19\} $ and $ a_{11}\in\{15,17,20\} $. Since $ [12,19] $ is adjacent to 4-gon, therefore if $ a_{12}=19 $ then $ lk(19) $ will not possible, so $ a_{12}=18 $, $ a_{14}=19 $ and $ a_{15}=7 $, $ a_{16}=16 $. Now remaining three distinct 4-gons are $ [10,16,b_1,b_2] $, $ [19,12,b_3,b_4] $, $ [18,a_6,b_5,a_{11}] $ and we also have another incomplete 4-gon $ [3,15,a_9,a_{10}] $. Therefore either $ \{b_1,b_2\}=\{3,15\} $ or $ \{b_3,b_4\}=\{3,15\} $ and $ b_5=11 $, $ a_6\in \{17,20\} $. Since $ [6,a_{11}] $ is adjacent to 5-gon, therefore $ a_{11}\neq 4,5 $ which implies $ [4,5] $ is adjacent to 4-gon, make a contradiction.}
\end{Case2}\vspace{-7mm}
\end{proof}

\begin{lemma}\label{lemma1}
	Let $K$ be a semi-equivelar map of type $ [4^3,5^1] $ on the surface of \Echar{-1}. Then, $K$ is isomorphic to $  \mathcal{K}_1 $, $  \mathcal{K}_2 $ or $  \mathcal{K}_3 $.
\end{lemma}
\begin{figure}[H]
	\begin{center}
		\begin{tikzpicture}[scale=1.5,line width=.5pt]
			\draw (0,0) node[anchor=west]{\large{ $ v $}}
			-- (0.5,0.5) node[anchor=south west]{\large{ $ v_1 $}}
			-- (0,1) node[anchor=south]{\large{ $v_2 $}}
			-- (-0.5,0.5) node[anchor=south east]{\large{ $v_3 $}}
			-- cycle;
			\draw (0,0)
			-- (-0.5,0.5)
			-- (-1,0) node[anchor=east]{\large{ $v_4 $}}
			-- (-0.5,-0.5) node[anchor=north east]{\large{ $v_5 $}}
			-- cycle;
			\draw (0,0)
			-- (-0.5,-0.5)
			-- (0,-1) node[anchor=north]{\large{ $v_6 $}}
			-- (0.5,-0.5) node[anchor=north west]{\large{ $v_7 $}}
			-- cycle;
			\draw (0,0)
			-- (0.5,-0.5)
			-- (1,-0.5) node[anchor=west]{\large{ $v_8 $}}
			-- (1,0.5) node[anchor=west]{\large{ $v_9 $}}
			-- (0.5,0.5)
			-- cycle;
			\node at (0,-1.5){Polyhedral};
		\end{tikzpicture}
		\hspace{0.5cm}
		\begin{tikzpicture}[scale=1.5,line width=.5pt]
			\draw (-1,0) node[anchor=east]{\large{ $v_4 $}}
			-- (-0.5,-0.5) node[anchor=east]{\large{ $v_5 $}}
			-- (0,-1) node[anchor=north]{\large{ $v_6 $}}
			-- (0.5,-0.5) node[anchor=north]{\large{ $v_7 $}}
			-- (1,-0.5) node[anchor=west]{\large{ $v_8 $}}
			-- (1,0.5) node[anchor=west]{\large{ $v_9 $}}
			-- (0.5,0.5) node[anchor=south]{\large{ $v_1 $}}
			-- (0,1) node[anchor=south]{\large{ $v_2 $}}
			-- (-0.5, 0.5) node[anchor=east]{\large{ $v_3 $}}
			-- cycle;
			\node at (0,-1.5){$ lk(v) $};
		\end{tikzpicture}
	\end{center}
      \caption{Example of $ lk(v) $ of type $ [4^3,5^1] $ }
      \label{example_lk(v)_4(3)5}
\end{figure}
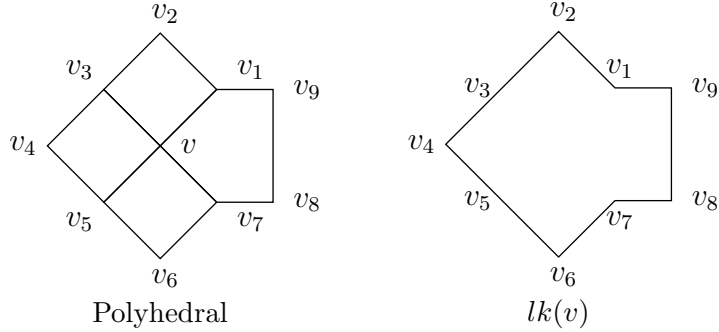
\begin{proof}
		 Let $ V(K) $, $F_4(K) $, $F_5(K) $ be \vx{}set, 4-gon face set and 5-gon face set respectively. Then $ |V(K)|=20 $, $ |F_4(K)|=15 $ and $ |F_5(K)|=4 $. Let $V(K)=\{1,2,\dots,20\} $. We write $lk(v) $ by the notation \klk{v\allowbreak}{\allowbreak v_1,v_2,v_3\allowbreak}{\allowbreak v_3,v_4,v_5\allowbreak}{\allowbreak v_5,v_6,v_7\allowbreak}{\allowbreak v_7,v_8,v_9,v_1} where  {'$ [\ ] $' to indicate} $ [v,v_1,v_9,v_8,v_7] $ form a 5-gon face and $ [v,v_1,v_2,v_3] $, $ [v,v_3,v_4,v_5] $, $ [v,v_5,v_6,v_7] $ form 4-gon faces, see  Fig. \ref{example_lk(v)_4(3)5}. Without loss{ of generality}, let \klk{1\allowbreak}{\allowbreak 2,3,4\allowbreak}{\allowbreak 4,5,6}{\allowbreak 6,7,8\allowbreak}{\allowbreak 8,9,10,2} and \klk{2\allowbreak}{\allowbreak 1,4,3\allowbreak}{\allowbreak 3,a,b\allowbreak}{\allowbreak b,c,10\allowbreak}{\allowbreak 10,9,8,1}.
	
		If $ b=5 $ then either $ a=6 $ or $ c=6 $. Now consider $ a=6 $. Then \klk{5\allowbreak}{\allowbreak 4,1,6\allowbreak}{\allowbreak 6,3,\allowbreak2\allowbreak}{\allowbreak 2,10,\allowbreak 3\allowbreak}{\allowbreak c,a_1,b_1,4} and \klk{6\allowbreak}{\allowbreak 3,2,5\allowbreak}{\allowbreak 5,4,1\allowbreak}{\allowbreak 1,8,7\allowbreak}{\allowbreak 7,a_2,b_2,3}. So without loss, assume $ (c,a_1,b_1)=(11,12,13) $ and $ (a_2,b_2)=(14,15) $. Therefore \klk{3\allowbreak}{\allowbreak 6,2,5\allowbreak}{\allowbreak 5,1,4}{4,13,15\allowbreak}{\allowbreak 15,14,7,6}, \klk{4\allowbreak}{\allowbreak 5,6,1\allowbreak}{\allowbreak 1,2,3\allowbreak}{\allowbreak 3,15,13\allowbreak}{\allowbreak 13,12,11,\allowbreak 5}, \klk{7\allowbreak}{\allowbreak 6,1,8\allowbreak}{\allowbreak 8,a_4,a_5\allowbreak}{\allowbreak a_5,a_6,14\allowbreak}{\allowbreak 14,15,3,6},  \klk{8\allowbreak}{\allowbreak 1,6,7\allowbreak}{\allowbreak 7,a_5,a_4\allowbreak}{\allowbreak a_4,\allowbreak a_7,9\allowbreak}{\allowbreak 9,10,2,1}, \klk{10\allowbreak}{\allowbreak 2,5,11\allowbreak}{\allowbreak 11,a_8,a_9\allowbreak}{\allowbreak a_9, a_{10},9\allowbreak}{\allowbreak 9,8,1,2}, \klk{11\allowbreak}{\allowbreak 5,2,\allowbreak 10\allowbreak}{\allowbreak 10,a_9,a_8\allowbreak}{\allowbreak a_8,a_{11},12\allowbreak}{\allowbreak 12,13,4,5}, \klk{13\allowbreak}{\allowbreak 4,3,15\allowbreak}{\allowbreak 15,a_{12},a_{13}\allowbreak}{\allowbreak a_{13},a_{14},12\allowbreak}{\allowbreak \allowbreak 12,11,5,4}, \klk{15\allowbreak}{\allowbreak 3,4,13\allowbreak}{\allowbreak 13,a_{13},a_{12}\allowbreak}{\allowbreak a_{12},a_{15},14\allowbreak}{\allowbreak 14,7,6,3}, \klk{12\allowbreak}{\allowbreak 13,\allowbreak a_{13},a_{14}\allowbreak}{\allowbreak a_{14},d_1,a_{11}\allowbreak}{\allowbreak a_{11},a_8,11\allowbreak}{\allowbreak 11,5,4,13}, \klk{9\allowbreak}{\allowbreak 8,a_4,a_7\allowbreak}{\allowbreak a_7,d_2,a_{10}\allowbreak}{\allowbreak a_{10},a_9,10\allowbreak\allowbreak}{\allowbreak 10,2,1,8} and \klk{14\allowbreak}{\allowbreak 15,a_{12},a_{15}\allowbreak}{\allowbreak a_{15},d_3,a_6\allowbreak}{\allowbreak a_6,a_5,7\allowbreak}{\allowbreak 7,6,3,15}, which implies $ a_5\in\{11,12\}\cup N $ and $ a_9\in\{14,15\}\cup N $ where $ N=\{16,17,18,19,20\} $.

		If $ a_5=11 $, then by considering $ lk(10) $ and $ lk(11) $, we get $ a_4=12 $, $ a_6=10 $ and then  $ a_{15}=a_{10}=16 $, $ a_7=a_{14}=17 $, $ a_{12}=18 $. Now \vr{}of remaining \5 are 16, 17, 18, 19, 20 and we have $ [9,16,d_2,17] $ is a face, which make a contradiction.
		
		If $ a_5=12 $ then either $ a_{11}=7 $ or $ a_{14}=7 $. $ a_{11}=7 $ implies $ \{a_8,d_1\}=\{8,14\} $. If $ (a_8,d_1)=(8,14) $ then $ a_4=11 $ and then from $ lk(11) $ we get $ a_9=9 $, $ a_7=10 $, which is a contradiction, therefore $ (a_8,d_1)=(14,8) $. Now from $ lk(14) $ we get $ a_6=11 $, $ a_9=a_{15} $, $ d_3=10 $, so without loss, we can assume $ a_9=16 $. Therefore \klk{16\allowbreak}{\allowbreak a_{10},9,10\allowbreak}{\allowbreak 10,11,14\allowbreak}{\allowbreak 14,15,a_{12}\allowbreak}{\allowbreak a_{12},\allowbreak b_1, b_2,a_{10}}, this implies $ a_{12}=17 $, $ [a_{10}=20] $, $ (b_1,b_2)=(18,19) $ and then after completing $ lk(17) $ we see that $ lk(9) $ will not possible. $ a_{14}=7 $ implies $ \{a_{13},d_1\}=\{8,14\} $. If $ (a_{13},d_1)=(8,14) $ then $ a_4=13 $, $ a_{14}=7 $, $ a_{12}=9 $, $ a_7=15 $. Now we have twelve \vr{}remain to complete link of 16, 17, 18, 19, 20, which is not possible, therefore $ (a_{13},d_1)=(14,8) $ and then $ a_6=13 $, $ a_{14}=7 $ and $ a_4=a_{11} $. So $ a_4 $ must be a new \vx{}, say 16. Now \klk{16\allowbreak}{\allowbreak a_8,11,12\allowbreak}{\allowbreak 12,7,8\allowbreak}{\allowbreak 8,9,a_7\allowbreak}{\allowbreak a_7,b_1,b_2,a_8}. So without loss, assume $ (a_7,a_8,b_1,b_2)=(17,18,19,20) $ and then after completing $ lk(20) $, we get $ a_9\notin V(K) $. Therefore $ a_5\neq 12 $.
		
		Hence $ a_5=16 $ and from $ lk(8) $, $ lk(11) $ and $ lk(13) $ we have $ a_4\in \{12,17\} $. If $ a_4=12 $ then either $ a_{11}=8 $ or $ a_{14}=8 $. $ a_{11}=8 $ implies $ [12,a_{{11}},d_1,a_{14}]=[12,8,7,16] $ and $ a_8=9 $, $ a_7=11 $, which implies $ [9,11,10,a_9] $, $ [9,10,a_9,a_{10}]\in F_4(K) $ which is not possible, therefore $ a_{14}=8 $. Then $ [12,a_{14},d_1,a_{11}]=[12,8,7,16] $ and $ a_{13}=9 $, $ a_7=13 $, $ d_2=15 $, $ a_{10}=a_{12} $, which implies $ a_{10}\in\{14,17\} $. $ a_{10}=14 $ implies $ a_{12}=13 $, $ a_{15}=9 $, $ d_3=10 $, $ a_6=a_9=17 $. Now \klk{16\allowbreak}{\allowbreak 17,14,7\allowbreak}{\allowbreak 7,8,12\allowbreak}{\allowbreak 12,11,20\allowbreak}{\allowbreak 20,19,18,17} i.e. $ a_8=20 $, $ a_{11}=16 $ which implies $ C(10,11,12,16,17)\in lk(20) $, a contradiction, therefore $ a_{10}=17 $. Now \klk{17\allowbreak}{\allowbreak a_9,10,9\allowbreak}{\allowbreak 9,13,\allowbreak 15\allowbreak}{\allowbreak 15,14,a_{15}\allowbreak}{\allowbreak a_{15},b_1,b_2,a_9} which implies $ \{a_9,a_{15}\}=\{16,18\} $. But $ a_{15}\neq16 $ as then $ [14,15,16,\allowbreak 17],[7,14,a_6,16]\in F_4(K) $, which is not possible and $ a_9\neq 16 $ as then $ lk(16) $ will not possible, therefore $ a_4\neq 12 $ i.e. $ a_4=17 $. Now \klk{16\allowbreak}{\allowbreak 17,8,7\allowbreak}{\allowbreak 7,14,a_6\allowbreak}{\allowbreak a_6,d_4,20\allowbreak}{\allowbreak 20,19,18,17} and \klk{17\allowbreak}{\allowbreak 16,7,8\allowbreak}{\allowbreak 8,9,a_7\allowbreak}{\allowbreak a_7,d_5,18\allowbreak}{\allowbreak 18,19,20,16}. So $ a_6=9 $ which implies $ [9,a_{10},d_2,a_7]=[9,16,7,14] $ and $ a_9=20 $, $ d_4=10 $, $ a_{15}=17 $, $ d_3=8 $, $ a_7=14 $, $ a_{12}=18 $, $ d_5=15 $. Now \klk{20\allowbreak}{\allowbreak 16,9,10\allowbreak}{\allowbreak 10,11,a_8\allowbreak}{\allowbreak a_8,d_6,19\allowbreak}{\allowbreak 19,18,17,16} i.e. $ a_8=12 $ and then \klk{18\allowbreak}{\allowbreak 17,14,15\allowbreak}{\allowbreak 15,\allowbreak 13,a_{13}\allowbreak}{\allowbreak a_{13},d_7,19\allowbreak}{\allowbreak 19,20,16,17} which implies $ a_{13}\notin V(K) $. So $ a_6\neq 9 $ and therefore $ a\neq6 $ for $ b=5 $. In the similar manner we see that $ c\neq6 $ for $ b=5 $ and therefore $ b\neq 5 $.
		
		So without loss, we assume $ b=11 $ and then $ (a,c)\in\{(5,7), (7,5), (5,12), (7,12), (12,5), \allowbreak (12,7), (12,13)\} $. $ (a,c)=(7,12) $ is isomorphic to $ (a,c)=(12,5) $ under the map $ (1\ 2)(3\ 4)\allowbreak (7\ 12)(8\  10) $, see \cite{datta2006}.
		\begin{Case}[$(a,c)=(5,7) $]\textnormal{
		 $ [5,6] $ will adjacent to \5, as if $ [6,7] $ is adjacent to \5 then $ lk(11) $ will not possible. Now \klk{5\allowbreak}{\allowbreak 6,1,4\allowbreak}{\allowbreak 4,a_1,11\allowbreak}{\allowbreak 11,2,3\allowbreak}{\allowbreak 3,a_2,a_3,6}, \klk{7\allowbreak}{\allowbreak 8,1,6\allowbreak}{\allowbreak 6,a_4,\allowbreak 10\allowbreak}{\allowbreak 10,2,11\allowbreak}{\allowbreak 11,a_5,a_6,8} and \klk{11\allowbreak}{\allowbreak 7,10,2\allowbreak}{\allowbreak 2,3,5\allowbreak}{\allowbreak 5,4,a_5\allowbreak}{\allowbreak a_5,\allowbreak  a_6,8,7} which implies $ a_1=a_5=12 $. From $ lk(4) $ we see that $ [1,4] $ can not be adjacent to \5, therefore $ [4,11] $ will adjacent to \5 which implies $ a_6=4 $ and then $ C(5,11,7,8)\in lk(12) $, is a contradiction.}
		\end{Case}
		\begin{Case}[$(a,c)=(7,5) $]\textnormal{
			 $ [5,6] $ is not adjacent to \5 as then $ lk(7) $ will not possible, therefore $ [6,7] $ and $ [4,5] $ are adjacent to \5. Now \klk{5\allowbreak}{\allowbreak 4,1,6\allowbreak}{\allowbreak 6,a_1,10\allowbreak}{\allowbreak 10,2,11\allowbreak}{\allowbreak 11,\allowbreak  a_2,a_3,4} which implies \klk{7\allowbreak}{\allowbreak 6,1,8\allowbreak}{\allowbreak 8,a_4,11\allowbreak}{\allowbreak 11,2,3\allowbreak}{\allowbreak 3,a_5,a_6} and \klk{11\allowbreak}{\allowbreak \allowbreak 5,10,2\allowbreak}{\allowbreak 2,3,7\allowbreak}{\allowbreak 7,8,\allowbreak  a_2\allowbreak}{\allowbreak a_2,a_3,4,5} and therefore $ a_2=a_4=12 $ and \klk{8\allowbreak}{\allowbreak 1,6,7\allowbreak}{\allowbreak 7,\allowbreak 11,12\allowbreak}{\allowbreak 12,a_7,9\allowbreak}{\allowbreak 9,10,2,1}, \klk{12\allowbreak}{\allowbreak 11,7,8\allowbreak}{\allowbreak 8,9,a_7\allowbreak}{\allowbreak a_7,a_8,a_3\allowbreak}{\allowbreak a_3,4,5,11} which implies $ a_3=13 $. Now $ \va{5} $, $ \va{6} $ must be two new \vr{}, so let $ \va{5}=14 $, $ \va{6}=15 $. Now \klk{6\allowbreak}{\allowbreak 7,8,1\allowbreak}{\allowbreak 1,4,5\allowbreak}{\allowbreak 5,10,15\allowbreak}{\allowbreak 15,14,3,7} i.e. $ \va{1}=15 $. From here we see that $\va{7}\in\{14,16\} $ and $[\va{7},\va{8}] $ is adjacent to \5. If $\va{7}=14 $ then after completing $\vl{14} $ and $\vl{3} $ we see that $\vl{4} $ will not possible, therefore $\va{7}=16 $. Without loss{ of generality}, assume $ [16,17,18,19,20]\in F_5(K) $ and $\va{8}=17 $. Then \klk{16\allowbreak}{\allowbreak 17,13,12\allowbreak}{\allowbreak 12,8,9\allowbreak}{\allowbreak 9,\va{9},20\allowbreak}{\allowbreak 20,19,\allowbreak 18,17}, \klk{9\allowbreak}{\allowbreak ,12,16\allowbreak}{\allowbreak 16,20,\va{9}\allowbreak}{\allowbreak \va{9},15,10\allowbreak}{\allowbreak 10,2,1,5}, \klk{10\allowbreak}{\allowbreak 2,11,5\allowbreak}{\allowbreak 5,6,15\allowbreak}{\allowbreak 15,\allowbreak  \va{9},9\allowbreak}{\allowbreak 9,\allowbreak 8,1,2} and \klk{3\allowbreak}{\allowbreak 7,11,2\allowbreak}{\allowbreak 2,1,4\allowbreak}{\allowbreak 4,\va{10},14\allowbreak}{\allowbreak 14,15,6,7}. As we have $ [16,20] $ adjacent to a \5, therefore $ [\va{9},15] $ will adjacent to a \5. So $ \va{9}=14 $ and then after completing $ \vl{14} $ we see that $ \vl{4} $ will not possible. }
		\end{Case}
	\begin{Case}[$(a,c)=(5,12) $]\textnormal{
		If $ [5,6] $ adjacent to \5 then \klk{5\allowbreak}{\allowbreak 6,1,4\allowbreak}{\allowbreak 4,\va{1},11\allowbreak}{\allowbreak 11,2,\allowbreak 3\allowbreak}{\allowbreak 3,\allowbreak \va{2},\va{3},6} which implies $ [4,5] $ to \5, which is not possible, therefore $ [6,7] $ is adjacent to \5. Now \klk{5\allowbreak}{\allowbreak 4,1,6\allowbreak}{\allowbreak 6,\va{1},3\allowbreak}{\allowbreak 3,2,11\allowbreak}{\allowbreak 11,\va{2},\va{3},4}, \klk{6\allowbreak}{\allowbreak 7,8,1\allowbreak}{\allowbreak 1,4,\allowbreak 5\allowbreak}{\allowbreak 5,3,\va{1}\allowbreak}{\allowbreak \va{1},\allowbreak \va{4},\va{5},7} and \klk{3\allowbreak}{\allowbreak \va{1},6,5\allowbreak}{\allowbreak 55,11,2\allowbreak}{\allowbreak 2,1,4\allowbreak}{\allowbreak 4,\vb{1},\vb{2},\allowbreak\va{1}} which implies $ [6,7,\va{5},\va{4},\va{1}]=[3,4,\vb{1},\vb{2},\va{1}] $, is a contradiction.}
	\end{Case}
\begin{Case}[$(a,c)=(7,12) $]\textnormal{
	 $ [6,7] $ is adjacent to \5 as if $ [5,6] $ adjacent to \5 then $ \vl{7} $ will not possible. Now \klk{6\allowbreak}{\allowbreak 7,8,1\allowbreak}{\allowbreak 1,4,5\allowbreak}{\allowbreak 5,\va{1},\va{2}\allowbreak}{\allowbreak \va{2},\va{3},\va{4},7} and from $ \vl{7} $, we get $ \va{4}\in\{3,11\} $.}
	\begin{Subcase}[$\va{4}=3 $]\textnormal{
		Then \klk{7\allowbreak}{\allowbreak 6,1,8\allowbreak}{\allowbreak 8,\va{5},11\allowbreak}{\allowbreak 11,2,3\allowbreak}{\allowbreak 3,\va{3},\va{2},6}, \klk{3\allowbreak}{\allowbreak 7,11,2\allowbreak}{\allowbreak 2,\allowbreak 1,4\allowbreak}{\allowbreak 4,\va{6},\va{3}\allowbreak}{\allowbreak \va{3},\va{2},6,7}, \klk{4\allowbreak}{\allowbreak 5,6,1\allowbreak}{\allowbreak 1,2,3\allowbreak}{\allowbreak 3,\va{3},\va{6}\allowbreak}{\allowbreak \va{6},\va{7},\va{8},\allowbreak 5} and \klk{11\allowbreak}{\allowbreak 12,\allowbreak 10,2\allowbreak}{\allowbreak 2,3,7\allowbreak}{\allowbreak 7,8,\va{5}\allowbreak}{\allowbreak \va{5},\va{9},\va{10},12}. From here we see that two \5 $ [3,7,6,\va{2},\va{3}] $ and $ [11,12,\allowbreak \va{10},\va{9},\va{5}] $ are not same, therefore $ \va{2}=13 $, $ \va{3}=14 $. Now \klk{8\allowbreak}{\allowbreak 1,6,7\allowbreak}{\allowbreak 7,11,\va{5}\allowbreak}{\allowbreak \va{5},\va{11},\allowbreak 9\allowbreak}{\allowbreak 9,10,2,1} and \klk{\va{5}\allowbreak}{\allowbreak 11,7,8\allowbreak}{\allowbreak 8,9,\va{11}\allowbreak}{\allowbreak \va{11},\allowbreak \va{12},\va{9}\allowbreak}{\allowbreak \va{9},\va{10},12,11} which implies $ \va{5}=15 $. Now two \4s $ [4,5,\va{8},\va{7},\va{6}] $ and $ [11,12,\allowbreak \va{10},\va{9},15] $ are either same of distinct.}
		
\begin{Subsubcase}\textnormal{If $[4,5,\va{8},\va{7},\va{6}] $ and $ [11,12,\va{10},\va{9},15] $ are same then $ \{\va{6},\va{7},\va{8}\}=\{11,12,$ $15\} $, $ \{\va{9},\va{10}\}=\{4,5\} $ and $ [16,17,18,19,20] $ is a \5 face. If $(\va{9},\va{10})=(4,5) $ then $ (\va{6},\va{7},\va{8})=(15,11,12)$, $ \va{11}=14 $, $ \va{12}=3 $ and then \klk{14\allowbreak}{\allowbreak 3,4,15\allowbreak}{\allowbreak 15,8,9\allowbreak}{\allowbreak 9,\va{13},\allowbreak 13\allowbreak}{\allowbreak 13,6,7,3}, \klk{9\allowbreak}{\allowbreak 8,15,14\allowbreak}{\allowbreak 14,13,\va{13}\allowbreak}{\allowbreak \va{13},\va{14},10\allowbreak}{\allowbreak 10,2,1,8} and \klk{13\allowbreak}{\allowbreak 14,9,\va{13}\allowbreak}{\allowbreak \va{13},\allowbreak \va{15},\va{1}\allowbreak}{\allowbreak \va{1},5,6\allowbreak}{\allowbreak 6,7,3,14}. From here we see that $ \va{1}\in\{12,16\} $. But if $ \va{1}=12 $ then after completing $ \vl{12} $ we get $ C(9,14,13,12,10)\in \vl{\va{13}} $, is a contradiction, therefore $ \va{1}=16 $. This implies $ \va{13}=12 $ and then after completing $ \vl{12} $, $ \vl{5} $ will not possible. Therefore $ (\va{9},\va{10})=(5,4) $ and then $ (\va{6},\va{7},\va{8})=(12,11,15) $. Now \klk{5\allowbreak}{\allowbreak 4,1,6\allowbreak}{\allowbreak 6,13,\va{12}\allowbreak}{\allowbreak \va{12},\va{11},15\allowbreak}{\allowbreak 15,11,12,4} which implies $ \va{1}=\va{12}=16 $ and $ \va{11}=17 $. Now \klk{17\allowbreak}{\allowbreak 16,5,15\allowbreak}{\allowbreak 15,8,9\allowbreak}{\allowbreak 9,\va{13},18\allowbreak}{\allowbreak 18,19,20,16} and \klk{9\allowbreak}{\allowbreak 8,15,17\allowbreak}{\allowbreak \allowbreak 17,18,\va{13}\allowbreak}{\allowbreak \va{13},\va{14},10\allowbreak}{\allowbreak 10,2,1,8} and from here we see that for any $ \va{13} $, $ \vl{\va{13}} $ will not possible.}
\end{Subsubcase}
\begin{Subsubcase}\textnormal{If $ [4,5,\va{8},\va{7},\va{6}] $ and $ [11,12,\va{10},\va{9},15] $ are distinct then without loss{ of generality}, let $ (\va{6},\va{7},\va{8},\allowbreak \va{9},\va{10})=(16,17,18,19,20) $ and then $ \va{11}\in\{13,17,18\} $ and \klk{5\allowbreak}{\allowbreak 4,1,6\allowbreak}{\allowbreak 6,13,\va{1}\allowbreak}{\allowbreak \allowbreak \va{1},\va{13},18\allowbreak}{\allowbreak 18,17,\allowbreak 16,4}. $ \va{11}=13 $ implies $ \va{1}=9 $, $ \va{12}=14 $ and then \klk{9\allowbreak}{\allowbreak 8,15,13\allowbreak}{\allowbreak \allowbreak 13,6,5\allowbreak}{\allowbreak 5,18,10\allowbreak}{\allowbreak 10,2,\allowbreak 1,8} i.e. $ \va{13}=10 $. Now \klk{10\allowbreak}{\allowbreak 2,11,12\allowbreak}{\allowbreak 12,\va{14},18\allowbreak}{\allowbreak 18,5,\allowbreak 9\allowbreak}{\allowbreak 9,8,1,2}, \klk{14\allowbreak}{\allowbreak 3,4,\allowbreak 16\allowbreak}{\allowbreak 16,\va{15},19\allowbreak}{\allowbreak 19,15,13\allowbreak}{\allowbreak 13,6,7,3}, \klk{18\allowbreak}{\allowbreak 5,9,\allowbreak 10\allowbreak}{\allowbreak 10,12,\va{14}\allowbreak}{\allowbreak \va{14},\va{16},17\allowbreak}{\allowbreak 1,16,4,5} and \klk{12\allowbreak}{\allowbreak 11,2,10\allowbreak}{\allowbreak 10,18,\va{14}\allowbreak}{\allowbreak \va{14},\va{17},20\allowbreak \allowbreak}{\allowbreak 20,19,15,11} which implies $ \va{14}\notin V(K) $. If $ \va{11}=17 $ then $ \va{12}\in\{16,18\} $. $ \va{12}=16 $ implies \klk{16\allowbreak}{\allowbreak 17,15,19\allowbreak}{\allowbreak 19,\va{14},14\allowbreak}{\allowbreak 14,3,4\allowbreak}{\allowbreak 4,5,\allowbreak 18,17}, \klk{17\allowbreak}{\allowbreak 16,19,15\allowbreak}{\allowbreak 15,\allowbreak 8,9\allowbreak}{\allowbreak 9,\va{15},18\allowbreak}{\allowbreak 18,5,4,16}, \klk{9\allowbreak}{\allowbreak 8,15,17\allowbreak}{\allowbreak 17,18,\allowbreak \va{15}\allowbreak}{\allowbreak \va{15},\va{16},10\allowbreak}{\allowbreak 10,2,1,8}, \klk{\allowbreak 19\allowbreak}{\allowbreak 15,17,16\allowbreak}{\allowbreak 16,14,\va{14}\allowbreak}{\allowbreak \va{14},\va{17},20\allowbreak}{\allowbreak 20,12,11,15} which implies $ \va{15}=20 $ and then $ C(11,2,1,8,9,20,12)\in\vl{10} $, which is a contradiction. Similarly $ \va{12}\neq 18 $. If $ \va{11}=18 $ then $ \va{12}=5 $ as if $ \va{12}=17 $, after completing $ \vl{18} $, we get $ \va{1}=10 $ which implies $ \vl{10} $ will not possible. Now $ \va{1}=19 $, $ \va{13}=15 $ and \klk{18\allowbreak}{\allowbreak 5,19,15\allowbreak}{\allowbreak 15,8,9\allowbreak}{\allowbreak 9,\va{14},17\allowbreak}{\allowbreak 17,\allowbreak 16,4,5}, \klk{19\allowbreak}{\allowbreak 15,18,5\allowbreak}{\allowbreak 5,6,13\allowbreak}{\allowbreak 13,\va{15},20\allowbreak}{\allowbreak 20,12,11,15} and \klk{9\allowbreak}{\allowbreak 8,15,\allowbreak 18\allowbreak}{\allowbreak 18,\allowbreak 17,\va{14}\allowbreak}{\allowbreak \va{14},\va{16},10\allowbreak}{\allowbreak 10,2,1,8} which implies $ \va{14}\notin V(K) $.}
\end{Subsubcase}
\end{Subcase}
\begin{Subcase}[$(\va{4}=11)$]\textnormal{Then \klk{7\allowbreak}{\allowbreak 6,1,8\allowbreak}{\allowbreak 8,\va{5},3\allowbreak}{\allowbreak 3,2,11\allowbreak}{\allowbreak 11,\va{3},\va{2}}, \klk{3\allowbreak}{\allowbreak 4,1,2\allowbreak}{\allowbreak 2,\allowbreak 11,7\allowbreak}{\allowbreak 7,8,\va{5}\allowbreak}{\allowbreak \va{5},\va{6},\va{7},4}, \klk{11\allowbreak}{\allowbreak 7,3,2\allowbreak}{\allowbreak 2,10,12\allowbreak}{\allowbreak 12,\va{7},\va{3}\allowbreak}{\allowbreak \va{3},\va{2},\allowbreak 6,7} and \klk{4\allowbreak}{\allowbreak 3,2,1\allowbreak}{\allowbreak 1,6,5\allowbreak}{\allowbreak 5,\va{8},17\allowbreak}{\allowbreak 17,16,15,3}, this implies $ (\va{2},\va{3})=(13,14) $ and $ (\va{5},\va{6},\va{7})=(15,16,\allowbreak 17) $. From here we see that $ [5,\va{1}] $ and $ [12,\va{7}] $ adjacent to same \5 and \vr\ of this \5 are $ 5,12,18,19,20 $ and $ \va{1}\neq12 $, $ \va{7}\neq 5 $, so $ [5,18,19,12,20] $ form a face and therefore $ \va{1}=18 $, $ \va{8}=20 $. Now after completing $ \vl{5} $, $ \vl{8} $, $ \vl{15} $, $ \vl{19} $ and $ \vl{12} $ respectively, we get $ C(1,8,9,19,12,11,2)\in \vl{10} $, which is a contradiction.}
\end{Subcase}
\end{Case}

\begin{Case}[$(a,c)=(12,7)$]\textnormal{Then $ [6,7] $ will adjacent to \5 as if $ [5,6] $ is adjacent to \5 then $ \vl{7} $ will not complete. Now \klk{7\allowbreak}{\allowbreak 6,1,8\allowbreak}{\allowbreak 8,\va{1},10\allowbreak}{\allowbreak 10,2,11\allowbreak}{\allowbreak 11,\va{2},\va{3},6}, \klk{10\allowbreak}{\allowbreak 2,11,7\allowbreak}{\allowbreak 7,8,\va{1}\allowbreak}{\allowbreak \va{1},\va{4},9\allowbreak}{\allowbreak 9,8,1,2} and \klk{8\allowbreak}{\allowbreak 1,6,7\allowbreak}{\allowbreak 7,10,\va{1}\allowbreak}{\allowbreak \va{1},\va{5},\allowbreak 9\allowbreak}{\allowbreak 9,10,2,1} which implies $ C(8,7,10,9)\in\vl{\va{1}} $, which make a contradiction.}
\end{Case}
\begin{claim}\label{claim1}If $ \vl{v} $ is like in  Fig. \ref{example_lk(v)_4(3)5}, then $ \vv{4},\vv{5},\vv{6}\notin V_{lk(v_1)}(K) $ and $ \vv{2},\vv{3},\vv{4}\notin V_{lk(v_7)}(K) $, where $ V_{lk(v)}(K) $ is the set of all \vr\ of $ \vl{v} $.
\end{claim}
{The proof of the Claim \ref{claim1} is} {clear} from the above cases. Without that being mentioned, we will often use this result in the next section of the proof.
\begin{Case}[$(a,c)=(12,13) $]\textnormal{
	Then \klk{8\allowbreak}{\allowbreak 1,6,7\allowbreak}{\allowbreak 7,\va{1},\va{2}\allowbreak}{\allowbreak \va{2},\va{3},9\allowbreak}{\allowbreak 9,10,2,1}, \klk{10\allowbreak}{\allowbreak 2,11,13\allowbreak}{\allowbreak 13,\va{4},\va{5}\allowbreak}{\allowbreak \va{5},\va{6},9\allowbreak}{\allowbreak 9,8,1,2} and \klk{9\allowbreak}{\allowbreak 10,\va{5},\va{6}\allowbreak}{\allowbreak \va{6},\va{7},\allowbreak\va{3}\allowbreak}{\allowbreak \va{3}, \va{2},8\allowbreak}{\allowbreak 8,1,2,10} which implies $ \va{2}\in\{12,13,14\}$.}
	
\begin{Subcase}[$\va{2}=12 $]\textnormal{Then by considering $ \vl{3} $, we have $ \va{1}\in\{3,11\} $. $ \va{1}=3 $ implies \klk{3\allowbreak}{\allowbreak 7,8,12\allowbreak}{\allowbreak 12,11,2\allowbreak}{\allowbreak 2,1,4\allowbreak}{\allowbreak 4,\va{8},\va{9},7} and \klk{12\allowbreak}{\allowbreak \va{3},9,8\allowbreak}{\allowbreak 8,7,3\allowbreak}{\allowbreak 3,2,\allowbreak 11\allowbreak}{\allowbreak 11,\va{10},\va{11},\allowbreak \va{3}}. Therefore $ [3,4,\va{8},\va{9},7] $ and $ [11,12,\va{3},\va{11},\va{10}] $ are distinct, which means $ \va{3}=14 $. Now \klk{14\allowbreak}{\allowbreak 12,8,9\allowbreak}{\allowbreak 9,\va{6},\va{7}\allowbreak}{\allowbreak \va{7},\va{12},\va{11}\allowbreak}{\allowbreak \va{11},\allowbreak \va{10},11,12}, \klk{11\allowbreak}{\allowbreak 12,3,2\allowbreak}{\allowbreak 2,10,13\allowbreak}{\allowbreak 13, \va{13},\va{10}\allowbreak}{\allowbreak \va{10},\va{11},14,12}, \klk{13\allowbreak}{\allowbreak \va{5},\va{5},10\allowbreak}{\allowbreak 10,2,11\allowbreak}{\allowbreak 11,\allowbreak\va{10},\va{13}\allowbreak \allowbreak}{\allowbreak \va{13},\va{14},\va{15},\va{4}} and\  \klk{4\allowbreak}{\allowbreak 3,2,1\allowbreak}{\allowbreak 1,6,5\allowbreak}{\allowbreak 5,\va{16},\va{8}\allowbreak}{\allowbreak \va{8},\va{9},7,3}. So $ [11,12,\allowbreak 14,\va{11},\va{10}] $ and $ [13,\va{4},\va{15},\va{14},\va{13}] $ are distinct faces and $ \va{10}\in\{5,15\} $. If $ \va{10}=5 $ then either \klk{5\allowbreak}{\allowbreak 11,13,4\allowbreak}{\allowbreak 4,1,6\allowbreak}{\allowbreak 6,\va{17},\va{11}\allowbreak}{\allowbreak \va{11},\allowbreak 14,12,11} or \klk{5\allowbreak}{\allowbreak 11,13,6\allowbreak}{\allowbreak 6,1,\allowbreak 4\allowbreak}{\allowbreak 4,\va{8},\va{11}\allowbreak}{\allowbreak \va{11},\allowbreak 14,12,11}. If \klk{5\allowbreak}{\allowbreak 11,13,4\allowbreak}{\allowbreak 4,\allowbreak 1,6\allowbreak}{\allowbreak 6,\va{17},\va{11}\allowbreak}{\allowbreak \va{11},14,12,11} then $ \va{13}=4 $ and two \5s $ [3,4,\va{8},\va{9},7] $, $ [13,\va{4},\va{15},\va{14},\va{4}] $ are same i.e. $ \va{14}=3 $, $ \va{15}=7 $. Without loss{ of generality}, we {can} conclude $ \va{9}=15 $, $ \va{11}=17 $ and $ [6,17,18,19,20] $ is a face, which contradict to $ \vl{6} $, therefore \klk{5\allowbreak}{\allowbreak 11,13,6\allowbreak}{\allowbreak 6,1,4\allowbreak}{\allowbreak 4,\va{8},\va{11}\allowbreak}{\allowbreak \va{11},14,12,11} and then $ [3,4,\va{8},\va{9},7] $ and $ [13,\va{4},\va{15},\va{14},6] $ are distinct. Now we may suppose that $ (\va{4},\va{8},\va{9},\va{11},\va{14},\allowbreak \va{15})=(15,16,17,18,19,20) $ and then after completing $ \vl{7} $, $ \vl{6} $ will not possible. Therefore $ \va{10}=15 $, which conclude that $ \va{10}=15 $ and $ \va{8}\in\{13,16\} $. If $ \va{8}=13 $ implies $ \va{4}=4 $, $ \va{5}=5 $, $ \va{16}=10 $, $ \va{15}=3 $, $ \va{14}=7 $ and without loss{ of generality}, we may assume that $ (\va{6},\va{9},\va{11},\va{17}\va{18})=(16,17,18,19,20) $. Now after completing $ \vl{6} $, $ \vl{7} $, $ \vl{9} $, $\vl{14} $ and $ \vl{16} $ respectively, we see that $ \vl{17} $ is not possible, therefore $ \va{8}=16 $. $ \vl{7} $ is not possible for $ \va{9}=13 $, $ \va{9}=17 $, $ \va{11}=18 $. {Therefore} \klk{7\allowbreak}{\allowbreak 3,12,8\allowbreak}{\allowbreak 8,1,6\allowbreak}{\allowbreak 6,\va{18},17\allowbreak}{\allowbreak 17,16,4,3} and \klk{6\allowbreak}{\allowbreak 5,4,1\allowbreak}{\allowbreak 1,8,7\allowbreak}{\allowbreak 7,17,\va{18}\allowbreak}{\allowbreak \va{18},\allowbreak\va{19},\va{20}}. Therefore two faces $ [5,6,\va{18},\va{19},\va{20}] $ and $ [13,\va{4},\va{15},\va{14},\va{13}] $ are same, which conclude $ \va{18}=19 $ and then \klk{5\allowbreak}{\allowbreak 6,4,1\allowbreak}{\allowbreak 1,16,\va{16}\allowbreak}{\allowbreak \va{16},\allowbreak \va{21},\va{20}\allowbreak}{\allowbreak \va{20},\va{19},19,6} with $ \va{16}\in\{15,18\} $. $ \va{16}=15 $ implies $ [15,\va{21}] $ will adjacent to \5, therefore $ \va{21}=18 $ and then after completing $ \vl{15} $ and $ \vl{18} $, we see that $ \va{6}=14 $, make contradiction. If $ \va{16}=18 $ then $ [18,\va{21}] $ will adjacent to \5 which implies $ \va{21}\in\{14,15\} $. But for $ \va{21}=14,15 $, we acquire $ \va{5}=\va{6} $ and $ \va{13}=19 $, respectively, which are not possible. Therefore $ \va{1}\neq 3 $ which conclude $ \va{1}=11 $, $ \va{3}=14 $. Now after completing $ \vl{13} $, $ \vl{7} $ and $ \vl{6} $ respectively, we obtain $ C(1,2,11,12,14,\va{11},4)\in \vl{3} $, which make contradiction.}
\end{Subcase}
\begin{Subcase}[$\va{2}=13 $]\textnormal{Then \klk{13\allowbreak}{\allowbreak \va{3},9,8\allowbreak}{\allowbreak 8,7,10\allowbreak}{\allowbreak 10,2,11\allowbreak}{\allowbreak 11,\va{8},\va{9},\va{3}}, \klk{7\allowbreak}{\allowbreak \va{6},9,10\allowbreak}{\allowbreak 10,13,8\allowbreak}{\allowbreak 8,1,6\allowbreak}{\allowbreak 6,\va{10},\va{11},\va{6}} and \klk{\va{3}\allowbreak}{\allowbreak 13,8,9\allowbreak}{\allowbreak 9,\va{6},\va{7}\allowbreak\allowbreak}{\allowbreak \va{7},\vb{1},\va{9}\allowbreak}{\allowbreak \va{9},\va{8},11,13} which indicate $ (\va{4},\va{5})=(8,7) $, $ \va{1}=10 $ and $ \va{6}=14 $. Now \klk{14\allowbreak}{\allowbreak 7,10,9\allowbreak}{\allowbreak 9,\va{3},\va{7}\allowbreak}{\allowbreak \va{7},\va{12},\va{11}\allowbreak}{\allowbreak \va{11},\va{10},6,7} which signify $ \va{3}=13 $, $ \va{7}=16 $ and \klk{16\allowbreak}{\allowbreak \vb{1},\va{9},15\allowbreak}{\allowbreak 15,9,14\allowbreak}{\allowbreak 14,\va{11},\va{12}\allowbreak}{\allowbreak \va{12},\vb{2},\vb{3}, \vb{1}}, \klk{6\allowbreak}{\allowbreak 7,8,1\allowbreak}{\allowbreak 1,4,\allowbreak 5\allowbreak}{\allowbreak 5,\va{13},\va{10}\allowbreak}{\allowbreak \va{10},\va{11},14,7} with $ \va{11}\in\{3,12,17\} $. But for $ \va{11}=3,12 $, after completing $ \vl{5} $, $ \vl{11} $, $ \vl{12} $ respectively, $ \vl{15} $ will not conceivable. Therefore $ \va{11}=17 $ and $ \va{10}=18 $, which is also not conceivable as then after completing $ \vl{5} $, $ \vl{4} $ will not possible. }
\end{Subcase}
\begin{Subcase}[$(\va{2}=14) $]\textnormal{Then $ \va{3}\in\{12,15\} $. If $ \va{3}=12 $ then after calculating $ \vl{12} $, $ \vl{11} $, $ \vl{14} $ and $ \vl{15} $ respectively, we come to a contradiction that $ \va{4} $ occur in two \5, therefore $ \va{3}=15 $ which indicate $ \va{5}=16 $, $ \va{6}=17 $ and $ \va{4}\in\{7,18\} $. If {$ \va{7}=7 $} then after calculating $ \vl{7} $, $ \vl{13} $, $ \vl{11} $, $ \vl{6} $ and $ \vl{16} $ respectively, we observe that $ [5,14,13,11,17] $ form a face which implies $ C(9,17,11,13,14,15)\in\vl{5} $, which is not conceivable, therefore $ \va{4}=18 $ and $ (\va{1},\va{7})\in\{(12,5),(12,19),(18,5),(18,12),(18,19),(19,5),(19,20)\} $. }

\textnormal{Let $\va{1}=12 $ then if $ [4,5] $ adjacent to \5 then after completing $ \vl{12} $, $ \vl{11} $ and $ \vl{14} $ respectively, we observe $ \vl{7} $ is not possible, therefore $ [3,4] $ is adjacent to \5. We have $ [2,3,12,11] $ and $ [7,8,14,12] $ are two faces, therefore $ [11,12] $, $ [13,18] $, $ [16,17] $ are adjacent to \5. Now $ \va{7}\neq5 $ as then $ [5,15] $, $ [5,17] $ are not adjacent to \5 which conclude $ \vl{5} $ is not conceivable. Therefore $ (\va{1},\va{7})\neq (12,5) $. For $ \va{7}=19 $, if $ [6,7] $ adjacent to \5 then completing $ \vl{12} $, $ \vl{7} $ subsequently, $ \vl{3} $ can not construct, similarly if $ [7,12] $ adjacent to \5 then $ \vl{14} $ will not possible, therefore $ (\va{1},\va{7})\neq(12,19) $. For $ \va{1}=18 $, if $ [6,7] $ is adjacent to \5 then $ [15,\va{7}] $ will adjacent to \5 which conclude $ \vl{16} $ will not conceivable, and similarly $ [5,6] $ is not adjacent to \5. Therefore $ (\va{1},\va{7})\notin\{(18,5),(18,12),(18,19)\} $ and so $ \va{1}=19 $. Now if $ \va{7}=5 $ then it will contradict to $ \vl{4} $ or $ \vl{6} $ for any possible of $ \vl{5} $, therefore $ (\va{1},\va{7})\neq(19,5) $, so $ (\va{1},\va{7})=(19,20) $ and then either \klk{3\allowbreak}{\allowbreak 4,1,2\allowbreak}{\allowbreak 2,11,12\allowbreak}{\allowbreak 12,\va{8},\va{9}\allowbreak}{\allowbreak \va{9},\va{10},\va{11},4} or \klk{3\allowbreak}{\allowbreak 11,12,2\allowbreak}{\allowbreak 2,1,4\allowbreak}{\allowbreak 4,\va{8},\va{9}\allowbreak}{\allowbreak \va{9},\va{10},\va{11},12}.}
		
\begin{Subsubcase} \textnormal{\klk{\allowbreak3\allowbreak}{\allowbreak 4,1,2\allowbreak}{\allowbreak 2,11,12\allowbreak}{\allowbreak 12,\va{8},\va{9}\allowbreak}{\allowbreak \va{9},\va{10},\va{11},4} then \klk{\allowbreak4\allowbreak}{\allowbreak 3,\allowbreak 2,1\allowbreak}{\allowbreak 1,\allowbreak  6,5\allowbreak}{\allowbreak 5,\va{12},\va{11}\allowbreak}{\allowbreak \va{11},\va{10},\va{9},3} and $ \va{9}\in\{14,15,16,17,18,19,20\} $. $ \va{9}=14 $ indicate $ \{{\va{8}},\va{10}\}=\{15,19\} $. $ (\va{8},\va{10})=(15,19) $ implies \klk{\allowbreak14\allowbreak}{\allowbreak 19,7,8\allowbreak}{\allowbreak 8,9,\allowbreak 15\allowbreak}{\allowbreak 15,12,\allowbreak 3\allowbreak}{\allowbreak 3,4,\va{11},19} and \klk{\allowbreak15\allowbreak}{\allowbreak 12,3,14\allowbreak}{\allowbreak 14,8,9\allowbreak}{\allowbreak 9,17,20\allowbreak}{\allowbreak 20,\va{13},\va{14},12} which conclude $ \va{11}=18 $, $ \va{12}\in\{13,16\} $, $ \{\va{13},\va{14}\}=\{6,7\} $ and \vr of the remaining \5 are 5, 11, 13, 16, 17. But $ ({\va{13}},\va{14})\neq (6,7) $ as then $ \vl{19} $ will not possible, and if $ (\va{13},\va{14})=(7,6) $ then completing $ \vl{6} $, $ \vl{7} $, $ \vl{5} $, $ \vl{11} $, $ \vl{13} $, $ \vl{16} $, we obtain $ C(8,14,3,4,18,20,7)\in \vl{19} $, which is not conceivable. Therefore $ (\va{8},\va{10})=(19,15) $. Now \klk{\allowbreak14\allowbreak}{\allowbreak 15,9,8\allowbreak}{\allowbreak 8,7,19\allowbreak}{\allowbreak 19,12,3\allowbreak}{\allowbreak 3,4,\allowbreak \va{11},15} and \klk{\allowbreak15\allowbreak}{\allowbreak 14,8,9\allowbreak}{\allowbreak 9,17,\allowbreak 20\allowbreak}{\allowbreak 20,\va{13},11\allowbreak}{\allowbreak 11,4,3,14} which indicate $ \va{11}=18 $, $ \va{13}\in\{13,16\} $ and for any $ \va{13} $, $ [20,\va{13}] $ will adjacent to \5. If $ \va{13}=13 $ then \klk{\allowbreak13\allowbreak}{\allowbreak 20,15,18\allowbreak}{\allowbreak 18,16,10\allowbreak}{\allowbreak 10,2,11\allowbreak}{\allowbreak 11,\allowbreak \va{14},\va{15}} which indicate $ [7,19] $, $ [12,19] $ are adjacent to same \5, so $ \{\va{14},\va{15}\}=\{6,7\} $ and \vr of the remaining \5 are 5, 12, 16, 17, 19. Now \klk{\allowbreak11\allowbreak}{\allowbreak 13,10,2\allowbreak}{\allowbreak 2,3,12\allowbreak}{\allowbreak 12,\allowbreak \va{16},\va{14}\allowbreak}{\allowbreak \va{14},\va{15},20,13} which implies $ (\va{14},\va{15})=(6,7) $ and $ \va{16}=5 $ and then we obtained following $ lk $ subsequently:  \klk{\allowbreak12\allowbreak}{\allowbreak 19,14,3\allowbreak}{\allowbreak 3,2,11\allowbreak}{\allowbreak 11,6,5\allowbreak}{\allowbreak 5,16,17,\allowbreak 19}, \klk{\allowbreak19\allowbreak}{\allowbreak 12,3,14\allowbreak}{\allowbreak 14,8,7\allowbreak}{\allowbreak 7,20,17\allowbreak}{\allowbreak 17,16,5,12}, \klk{\allowbreak7\allowbreak}{\allowbreak 6,1,8\allowbreak}{\allowbreak 8,14,19\allowbreak}{\allowbreak 19,\allowbreak 17, 20\allowbreak}{\allowbreak 20,\allowbreak 13,11,6}, \klk{\allowbreak16\allowbreak}{\allowbreak 17,9,10\allowbreak}{\allowbreak 10,13,18\allowbreak}{\allowbreak 18,4,5\allowbreak}{\allowbreak 5,12,19,17}, \klk{\allowbreak18\allowbreak}{\allowbreak 4,5,16\allowbreak \allowbreak}{\allowbreak 16,10,13\allowbreak}{\allowbreak 13,20,15\allowbreak}{\allowbreak 15,14,3,4}, \klk{\allowbreak5\allowbreak}{\allowbreak 12,11,6\allowbreak}{\allowbreak 6,1,4\allowbreak}{\allowbreak 4,18,16\allowbreak}{\allowbreak \allowbreak 16,17,19, 12}, \klk{\allowbreak6\allowbreak \allowbreak}{\allowbreak 7,8,1\allowbreak}{\allowbreak 1,4,5\allowbreak}{\allowbreak 5,12,11\allowbreak}{\allowbreak 11,13,20,7}, \klk{\allowbreak20\allowbreak}{\allowbreak 7,19,17\allowbreak}{\allowbreak \allowbreak 19,9,15\allowbreak}{\allowbreak 15,18,\allowbreak 13\allowbreak}{\allowbreak 13,\allowbreak 11,6,7} and \klk{\allowbreak17\allowbreak}{\allowbreak 16,10,9\allowbreak}{\allowbreak 9,15,20\allowbreak}{\allowbreak 20,7,19\allowbreak}{\allowbreak 19,12,5,\allowbreak 16}. Let this be the semi-equivelar map ${ \mathcal{K}_1} $. If $ \va{13}=16 $ then this map will isomorphic to $  \mathcal{K}_1 $ under the map $ (1\ 3)(2\ 4)(5\ 11)(6\ 12)(7\ 19)(8\ 14)(9\ 15)\allowbreak (10\ 18)(13\ 16)(17\ 20) $.}
				
\textnormal{$\va{9}=15 $ indicate $ \{\va{8},\va{10}\}=\{14,20\} $, but if $ (\va{8},\va{10})=(14,20) $ then completing $ \vl{15} $, $ \vl{14} $, $ \vl{12} $ subsequently, we see that $ \vl{5} $ is not conceivable, therefore $ (\va{8},\va{10})=(20,14) $. Now after completing $ \vl{14} $, $ \vl{15} $, we obtain that $ \vl{13} $ will not possible.}

\textnormal{$\va{9}=16 $ implies $ \{\va{8},\va{10}\}=\{17,18\} $. If $ (\va{8},\va{10})=(17,18) $ then \klk{\allowbreak16\allowbreak}{\allowbreak 3,12,17\allowbreak}{\allowbreak \allowbreak 17,9,\allowbreak 10\allowbreak}{\allowbreak 10,13,18\allowbreak}{\allowbreak 18\va{11},4} which indicate $ \va{11}=19 $, $ \va{12}=14 $ and then \klk{\allowbreak14\allowbreak}{\allowbreak 5,4,19\allowbreak}{\allowbreak \allowbreak 19,\allowbreak 7,8\allowbreak \allowbreak}{\allowbreak 8,9,15\allowbreak}{\allowbreak 15,\va{13},\va{14},5}, \klk{\allowbreak17\allowbreak}{\allowbreak 12,3,16\allowbreak}{\allowbreak 16,10,9\allowbreak}{\allowbreak 9,15,20\allowbreak}{\allowbreak \allowbreak 20,\va{15},\va{16},\allowbreak 12\allowbreak }, \klk{\allowbreak15\allowbreak}{\allowbreak 14,8,9\allowbreak}{\allowbreak 9,17,20\allowbreak}{\allowbreak 20,\va{17},\va{13}\allowbreak}{\allowbreak \va{13},\va{14},5,14}, \klk{\allowbreak5\allowbreak}{\allowbreak 14,19,4\allowbreak}{\allowbreak 4,1,\allowbreak 6\allowbreak}{\allowbreak 6,\va{18},\va{14}\allowbreak}{\allowbreak \va{14},\allowbreak \va{13},15,14}, \klk{\allowbreak12\allowbreak}{\allowbreak 17,16,3\allowbreak}{\allowbreak 3,2,11\allowbreak}{\allowbreak 11,\va{19},\allowbreak\va{16}\allowbreak}{\allowbreak \va{16},\va{15},20,17\allowbreak } which indicate $ (\va{17},\va{18})\allowbreak=(18,12) $, $ (\va{13},\va{14})=(13,11) $, $ (\va{15},\va{16})=(7,6) $, $ \va{19}=5 $. Now we compute following link subsequently: \klk{\allowbreak6\allowbreak}{\allowbreak 7,8,1\allowbreak}{\allowbreak 1,4,5\allowbreak}{\allowbreak 5,\allowbreak 11,12\allowbreak}{\allowbreak 12,17,20,7}, \klk{\allowbreak7\allowbreak}{\allowbreak 6,1,8\allowbreak}{\allowbreak 8,14,19\allowbreak}{\allowbreak 19,18,\allowbreak 20\allowbreak}{\allowbreak 20,17,12,6}, \klk{\allowbreak11\allowbreak}{\allowbreak 5,6,12\allowbreak}{\allowbreak 12,3,2\allowbreak}{\allowbreak 2,\allowbreak 10,13\allowbreak}{\allowbreak 13,15,14,5}, \klk{\allowbreak13\allowbreak}{\allowbreak 11,2,10\allowbreak}{\allowbreak \allowbreak 10,16,18\allowbreak}{\allowbreak 18,20,15\allowbreak}{\allowbreak 15,14,5,11},  \klk{\allowbreak18\allowbreak}{\allowbreak 16,10,13\allowbreak}{\allowbreak 13,15,20\allowbreak}{\allowbreak 20,7,19\allowbreak}{\allowbreak 19,4,3,16}, \klk{\allowbreak\allowbreak 19\allowbreak}{\allowbreak 18,20,7\allowbreak}{\allowbreak 7,8,14\allowbreak}{\allowbreak 14,5,4\allowbreak}{\allowbreak \allowbreak 4,3,16,18} and \klk{\allowbreak20\allowbreak}{\allowbreak 7,19,18\allowbreak}{\allowbreak 18,13,15\allowbreak}{\allowbreak 15,9,\allowbreak 17\allowbreak}{\allowbreak 17,12,6,7}. Let this be the semi-equivelar map ${ \mathcal{K}_2} $. $ (\va{8},\va{10})=(18,17) $ implies \klk{\allowbreak16\allowbreak}{\allowbreak 17,9,10\allowbreak}{\allowbreak 10,13,\allowbreak 18\allowbreak}{\allowbreak 18,12,3\allowbreak}{\allowbreak \allowbreak 3,4,\va{11},17} which indicate $ \va{11}=19 $, $ \va{12}=14 $. Now \klk{\allowbreak19\allowbreak}{\allowbreak 4,5,14\allowbreak}{\allowbreak 14,8,7\allowbreak}{\allowbreak \allowbreak 7,20,17\allowbreak}{\allowbreak 17,6,3,4}, \klk{\allowbreak17\allowbreak}{\allowbreak 16,10,9\allowbreak}{\allowbreak 9,15,20\allowbreak}{\allowbreak 20,7,19\allowbreak}{\allowbreak 19,4,3,16}, \klk{\allowbreak7\allowbreak}{\allowbreak \allowbreak 6,1,8\allowbreak}{\allowbreak \allowbreak 8,14,19\allowbreak}{\allowbreak 19,17,20\allowbreak}{\allowbreak 20,\va{13},\va{14},6} which conclude three \vr of the remaining \5 are 5,14,15, assume $ [5,14,15,\vb{1},\vb{2}] $ be the face. Then \klk{\allowbreak20\allowbreak}{\allowbreak 7,19,17\allowbreak}{\allowbreak 17,9,15\allowbreak \allowbreak}{\allowbreak 15,\vb{1},\va{13}\allowbreak}{\allowbreak \va{13},\va{14},6,\allowbreak 7}, \klk{\allowbreak14\allowbreak}{\allowbreak 15,9,8\allowbreak}{\allowbreak 8,7,19\allowbreak}{\allowbreak \allowbreak 19,4,5\allowbreak}{\allowbreak 5,\vb{2},\vb{1},15}, \klk{\allowbreak15\allowbreak}{\allowbreak 14,8,9\allowbreak}{\allowbreak 9,17,20\allowbreak}{\allowbreak 20,\va{13},\allowbreak \vb{1}\allowbreak}{\allowbreak \vb{1},\vb{2},5,14}, \klk{\allowbreak5\allowbreak}{\allowbreak 14,19,4\allowbreak}{\allowbreak 4,1,6\allowbreak}{\allowbreak 6,\va{14},\vb{2}\allowbreak}{\allowbreak \vb{2},\allowbreak \vb{1},15,14} and \klk{\allowbreak6\allowbreak}{\allowbreak 7,8,1\allowbreak}{\allowbreak 1,\allowbreak 4,5\allowbreak}{\allowbreak 5,\vb{2},\va{14}\allowbreak}{\allowbreak \va{14},\va{13},20,7} which indicate $ \{\vb{1},\vb{2},\allowbreak \va{13},\va{14}\}=\{11,12,13,18\} $. $ \vb{1}=11 $ implies $ \{\vb{2},\va{13}\}=\{12,13\} $, $ \va{14}=18 $. But for $ ({\vb{2},\va{13}})=(12,13) $, semi-equivelar map is isomorphic to ${ \mathcal{K}_2} $ under the map $ (1\ 11\ 18\ 17\ 8\ 5\ 3\ 10\ 15\ 19)$ $(2\ 13\ 20\ \allowbreak  7\ 6\ 12\ 16\ 9\ 14\ 4) $. If $ (\vb{2},\va{13})=(13,12) $ then the semi-equivelar map is isomorphic to the semi-equivelar map ${ \mathcal{K}_1} $ under the map $ (1\ 10)$ $(3\ 11)$ $(4\ 13)$ $(5\ 18)(6\ 16)$ $(7\ 17)$ $(8\ 9)$ $(14\ 15)$ $(19\ 20) $. $ \vb{1}=12 $ implies $ \va{14}=13 $ and $ \{\vb{2},\va{13}\}=\{11,18\} $. For $ (\vb{2},\va{13})=(11,18) $, the semi-equivelar map is isomorphic to the semi-equivelar map ${ \mathcal{K}_2} $ under the map $ (1\ 14\ 6\ 19)$ $(2\ 15\ 12\ 17)$ $(3\ 9\ 11\ 20)$ $(4\ 8\ 5\ 7)$ $(10\ 13\ 18\ 16) $. If $ (\vb{2},\va{13})=(18,11) $ then the semi-equivelar map is isomorphic to the semi-equivelar map ${ \mathcal{K}_1} $ under the map $ (1\ 8\ 9\ 10\ 2)$ $(3\ 6\ 14\ 17\ 13)$ $(4\ 7\ 15\ 16\ 11)$ $(5\ 19\ 20\ \allowbreak  18\ 12) $. $ \vb{1}=13 $ implies $ \va{14}=12 $ and $ \{\vb{2},\va{13}\}=\{11,18\} $. For $ (\vb{2},\va{13})=(11,18) $, \klk{\allowbreak11\allowbreak}{\allowbreak 5,6,12\allowbreak}{\allowbreak 12,3,2\allowbreak}{\allowbreak 2,10,13\allowbreak}{\allowbreak \allowbreak 13,15,14,5}, \klk{\allowbreak12\allowbreak}{\allowbreak 6,5,11\allowbreak}{\allowbreak 11,2,3\allowbreak}{\allowbreak 3,16,18\allowbreak}{\allowbreak 18,20,7,6}, \klk{\allowbreak13\allowbreak}{\allowbreak 15,20,11\allowbreak \allowbreak}{\allowbreak 11,2,\allowbreak 10\allowbreak}{\allowbreak 10,16,18\allowbreak}{\allowbreak 18,5,14,15} and \klk{\allowbreak18\allowbreak}{\allowbreak 13,10,16\allowbreak}{\allowbreak 16,3,12\allowbreak}{\allowbreak \allowbreak 12,6,5\allowbreak}{\allowbreak 5,14,\allowbreak 15,13}. Let this be the semi-equivelar map ${ \mathcal{K}_3} $. For $ (\vb{2},\va{13})=(18,11) $, then semi-equivelar map will isomorphic to the semi-equivelar map ${ \mathcal{K}_2} $ under the map $ (1\ 8)$ $(2\ 9)$ $(3\ 15)$ $(4\ 14)$ $(5\ 19)$ $(6\ 7)$ $(11\ 17)$ $(12\ 20)\allowbreak (13\ 16) $. $ \vb{1}=18 $ indicate $ \va{14}=11 $ and $ \{\vb{2},\va{13}\}=\{12,13\} $. Under the map $ (1\ 2\ 10\ 9\ 8)$ $(3\ 13\ \allowbreak  17\ 14\ 6)$ $(4\ 11\ 16\ 15\ 7)$ $(5\ 12\ \allowbreak  18\ 20\ 19) $, the semi-equivelar map for $ (\vb{2},\va{13})=(12,13) $ is isomorphic to the semi-equivelar map ${ \mathcal{K}_3} $ and for $ (\vb{2},\va{13})=(13,12) $, is isomorphic to the semi-equivelar map ${ \mathcal{K}_2}$.}

\textnormal{$\va{9}=17 $ indicate $ \{\va{8},\va{10}\}=\{16,20\} $. For $ (\va{8},\va{10})=(16,20) $, This semi-equivelar map is isomorphic to the semi-equivelar map ${ \mathcal{K}_2} $ under the map $ (1\ 5)$ $(2\ 14)$ $(3\ 19)$ $(7\ 12)$ $(8\ 11)$ $(9\ 13)$ $(10\ 15)$ $(16\ 20)$ $(17\ 18) $. For $ (\va{8},\va{10})=(20,16) $, the semi-equivelar map is isomorphic to the semi-equivelar map ${ \mathcal{K}_1} $ under the map $ (1\ 9\ 2\ 8\ 10)$ $(3\ 14\ \allowbreak  13\ 6\ 17)$ $(4\ 15\ 11\ 7\ 16)\allowbreak (5\ 20\ 12\ 19\ 18) $.}			

\textnormal{$\va{9}=18 $ indicates $ \va{14}\in\{14,15,19,20\} $. For any {$ \va{14} $}, semi-equivelar maps will isomorphic to one of ${ \mathcal{K}_1} $, ${ \mathcal{K}_2} $ and ${ \mathcal{K}_3} $ under one of the map of $ (1\ 18\ 6\ 13\ 5\ 15\ 12\ 8\ 16)$ $(2\ 19\ 9\ 3\ 7\ 10\ 4\ 20\ 11\ 14\ 17) $, $ (1\ 11\ 4\ 12)$ $(2\ 5\ 3\ 6)$ $(7\ 10\ 14\ 18)$ $(8\ 13\ 19\ 16)$ $(9\ 15\ 20\ 17) $, $ (1\ 16\ 5\ 9\ 19\ 15\ 7\ 13)$ $(2\ 3\ 12\ 11)$ $(4\ 17\ \allowbreak 14\ 20\ 8\ 18\ 6\ 10) $, $ (1\ 10)$ $(3\ 11)$ $(4\ 13)$ $(5\ 18)$ $(6\ 16)$ $(7\ 17)$ $(8\ 9)$ $(14\ 15)$ $(19\ 20) $, $ (1\ 3\ 19\ 6\ 2\ 14\ 5\ 11\ 8\ \allowbreak 4\ 12\ 7)$ $(9\ 18\ 17\ 13)$ $(10\ 15\ 16\ 20) $, $ (1\ 15\ 2\ 14\ 10\ 5\ 17\ 12\ 7\ 18)$ $(3\ 8\ 13\ 4\ 9\ 11\ 19\ 16\ 6\ 20) $, $ (1\ 2\ 10\ 9\ 8)$ $(3\ 13\ 17\ 14\ 6)$ $(4\ 11\ 16\ 15\ 7)$ $(5\ 12\ 18\ 20\ 19) $, $ (1\ 13\ 6\ 18)$ $(2\ 11\ 12\ 3)$ $(4\ 10\ 5\ 16)$ $(7\ 20)$ $(8\ \allowbreak 15)$ $(9\ 14)$ $(17\ 19) $, $ (1\ 2)$ $(3\ 4)$ $(5\ 12)$ $(6\ 11)$ $(7\ 13)$ $(8\ 10)$ $(14\ 16)$ $(15\ 17)$ $(18\ 19) $, $ (1\ 7)$ $(2\ 20\ 3\ 17)$ $(4\ \allowbreak 19)$ $(5\ 14)$ $(6\ 8)$ $(9\ 11\ 15\ 12)$ $(10\ 13\ 18\ 16) $, $ (1\ 14\ 20\ 16\ 2\ 5\ 7\ 9\ 13\ 12)$ $(3\ 4\ 19\ 17\ 10\ 11\ 6\ 8\ 15\ 18) $, $ (1\ 8\ 9\ 10\ 2)$ $(3\ 6\ 14\ 17\ 13)$ $(4\ 7\ 15\ 16\ 11)$ $(5\ 19\ 20\ 18\ 12) $, $ (2\ 8)$ $(3\ 7)$ $(4\ 6)$ $(9\ 10)$ $(11\ 14)$ $(12\ 19)$ $(13\ \allowbreak 15)$ $(16\ 17)$ $(18\ 20) $, $ (1\ 11\ 4\ 12)$ $(2\ 5\ 3\ 6)$ $(7\ 10\ 14\ 18)$ $(8\ 13\ 19\ 16)$ $(9\ 15\ 20\ 17) $, $ (1\ 7)$ $(2\ 20\ 3\ 17)$ $(4\ \allowbreak 19)$ $(5\ 14)$ $(6\ 8)$ $(9\ 11\ 15\ 12)$ $(10\ 13\ 18\ 16) $.}

\textnormal{For $\va{9}=19,20 $, the semi-equivelar map is isomorphic to the semi-equivelar map for $ \va{9}=18 $ under the map $ (1\ 2)$ $(3\ 4)$ $(5\ 12)$ $(6\ \allowbreak  11)$ $(7\ 13)$ $(8\ 10)$ $(14\ 16)$ $(15\ 17)$ $(18\ 19) $ and $ (1\ 2\ 10\ 9\ 8)$ $(3\ 13\ 17\ 14\ 6)$ $(4\ 11\ 16\ \allowbreak 15\ 7)$ $(5\ 12\ 18\ 20\ 19) $ respectively.}
\end{Subsubcase}
\begin{Subsubcase}\klk{3\allowbreak}{\allowbreak 11,12,2\allowbreak}{\allowbreak 2,1,4\allowbreak}{\allowbreak 4,\va{8},\va{9}\allowbreak}{\allowbreak \va{9},\va{10},\va{11},12}. \textnormal{Then the semi-equivelar map is isomorphic to ${ \mathcal{K}_2} $ under one of the map $ (1\ 10\ 8\ 2\ 9)$ $(3\ 17\ 6\ 13\ 14)$ $(4\ 16\ 7\ 11\ 15)$ $(5\ 18\ 19\ \allowbreak 12\ 20) $, $ (1\ 13\ 8\ 11\ 9\ 5\ 6\ 19\ 3\ 20)$ $(2\ 15\ 6\ 10\ 14\ 12\ 17\ 4\ 18\ 7) $, $ (1\ 9)$ $(2\ 10)$ $(3\ 16)$ $(4\ 17)$ $(5\ 20)$ $(6\ 15)\allowbreak (7\ 14)$ $(11\ 13)$ $(12\ \allowbreak  18) $, $ (1\ 15\ 4\ 20)$ $(2\ 13)$ $(3\ 18)$ $(5\ 17\ 6\ 9)$ $(7\ 8\ 14\ 19)$ $(10\ 11)$ $(12\ 16) $.}\end{Subsubcase}\end{Subcase}
\end{Case}\vspace{-7mm}
\end{proof}

\begin{lemma}\label{lemma2}$ Aut( \mathcal{K}_1)=\langle\alpha_1,\alpha_2\rangle\cong D_2 $, $ Aut( \mathcal{K}_2)=\langle\beta\rangle\cong\mathbb{Z}_2 $, $ Aut( \mathcal{K}_3)=\langle\gamma_1,\gamma_2\rangle\cong D_4 $, where $ \alpha_1=(1\ 3)$ $(2\ 4)$ $(5\ 11)$ $(6\ 12)$ $(7\ 19)$ $(8\ 14)$ $(9\ 15)$ $(13\ 16)$ $(10\ 18)$ $(17\ 20) $, $ \alpha_2=(1\ 6)$ $(2\ 11)$ $(3\ 12)$ $(4\ 5)$ $(7\ 8)$ $(9\ 20)$ $(10\ 13)$ $(14\ 19)$ $(15\ 17)$ $(16\ 18) $, $ \beta=(1\ 4)$ $(2\ 3)$ $(5\ 6)$ $(7\ 14)$ $(8\ 19)$ $(9\ 18)$ $(10\ 16)$ $(11\ 12)$ $(13\ 17)$ $(15\ 20) $, $ \gamma_1=(1\ 8)$ $(2\ 9)$ $(3\ 15)$ $(4\ 14)$ $(5\ 19)$ $(6\ 7)$ $(11\ 17)$ $(12\ 20)$ $(13\ 16) $, $ \gamma_2=(1\ 5)$ $(2\ 11)$ $(3\ 12)$ $(4\ 6)$ $(7\ 19)$ $(8\ 14)$ $(9\ 15)$ $(10\ 13)$ $(16\ 18)$ $(17\ 20) $.
\end{lemma}

\begin{proof} For a polyhedral complex $ K $ we define the graph $ G_{i}(K) $ (introduced in \cite{datta2006}) as follows: $ V(G_{i}(K))=V(K) $ and $ [u,v]\in EG(G_{i}(K))  $ if $ |N(u)\cap N(v)|=i $, where $ u,v\in V(K) $ and $ N(v)=\{u\in V(K): u\in V(lk(v)) \} $. Then ${ EG(G_6( \mathcal{K}_1))} =\{[1,11],$ $[2,6],$ $[3,5],$ $[4,12]\} $, ${ EG(G_6( \mathcal{K}_2))}=\{[2,16],$ $[3,10],$ $[9,13],$ $[10,15],$ $[16,20],$ $[17,18]\} $, ${EG(G_6( \mathcal{K}_3))}=\{[2,16],$ $[3,10],$ $[9,13],$ $[10,15],$ $[11,18],$ $[12,13],$ $[16,20],$ $[17,18]\} $. If $ f:K\rightarrow K $ is an automorphism map of $ K $, then $ f $ is automorphism of $ G_i(K) $ as well. For this proof we denote $ lk_K(v) $ as $ lk(v) $ of semi-equivelar map $ K $. It is clear that $ \alpha_1,\alpha_2 $ are automorphism {maps} of $  \mathcal{K}_1 $, $ \beta $ is an automorphism map of $  \mathcal{K}_2 $ and $ \gamma_1,\gamma_2 $ are automorphism {maps} of $  \mathcal{K}_3 $.
	
	Let $ \phi: \mathcal{K}_1\rightarrow  \mathcal{K}_1 $ be an automorphism map of $  \mathcal{K}_1 $, then from {$ EG(G_6(\mathcal{K}_1)) $, we have} $ \{\phi(1),\phi(2),\phi(3),\phi(4),\phi(5),\phi(6),$ $\phi(11),\phi(12)\}=\{1,2,3,4,5,6,11,12\} $. If $ \phi(1)=1 $, then $ \phi=Identity $ which clearly is an automorphism map. If $ \phi(1)=2 $ then from ${EG(G_6( \mathcal{K}_1))} $, $ lk_{ \mathcal{K}_1}(1) $, $ lk_{ \mathcal{K}_1}(2)$, $lk_{ \mathcal{K}_1}(5) $, $lk_{ \mathcal{K}_1}(12) $, $ lk_{ \mathcal{K}_1}(9)$, we conclude $ \phi=(1\ 2)$ $(3\ 4)$ $(5\ 12)$ $(6\ 11)$ {$(7\ 13)$} $(8\ 10)$ $(14\ 16)$ $(15\ 17)$ $(18\ 19) $, but $\phi([5,12,19,17,16])=[12,5,18,15,14] $. $ [5,12,19,17,16] $ is a face but $[12,5,18,15,14]$ is not a face, therefore it is not automorphism map. If $ \phi(1)=3 $ then from ${EG(G_6( \mathcal{K}_1))} $ and link of $  \mathcal{K}_1 $, we get $ \phi=\alpha_1 $. If $ \phi(1)=4 $, then from ${EG( G_6( \mathcal{K}_1))} $, $ lk_{ \mathcal{K}_1}(1) $, $ lk_{ \mathcal{K}_1}(4) $, $ lk_{ \mathcal{K}_1}(7) $, $ lk_{ \mathcal{K}_1}(11) $, $ lk_{ \mathcal{K}_1}(12) $, $ lk_{ \mathcal{K}_1}(16) $, we get $ \phi=(1\ 4)$ $(2\ 3)$ $(5\ 6)$ $(7\ 16)$ $(8\ 18)$ $(9\ 15\ 20\ 17)$ $(10\ 14\ 13\ 19)$ $(11\ 12) $, but $ \phi([13,18,15,20])=[19,8,20,17] $. $ [13,18,15,20] $ is a face but $ [19,8,20,17] $ is not a face i.e. it is not an automorphism. Similarly if $ \phi(1)=5 $ then $ \phi=(1\ 5)$ $(3\ 11)$ $(2\ 12)$ $(4\ 6)$ $(7\ 18)$ $(8\ 16)$ $(9\ 17)$ $(10\ 19)$ $(13\ 14)$ $(15\ 20)=\phi_1  $ (say), $ \phi(1)=11 $ implies $ \phi=(1\ 11)$ $(2\ 6)$ $(3\ 5)$ $(4\ 12)$ $(7\ 10)$ $(8\ 13)$ $(9\ 20)$ $(14\ 18)$ $(16\ 19)=\phi_2 $ (say) and $ \phi(1)=6 $ implies $ \phi=\alpha_2 $, $ \phi(1)=12 $ implies $ \phi=\alpha_1\circ\alpha_2 $. Now $ \phi_1([9,10,16,17])=[17,19,8,9] $ and $ \phi_2([5, 12, 19, 17, 16])=[3, 4, 16, 17, 19] $, however $ [9,10,16,17] $, $ [5, 12, 19, 17, 16] $ are faces, but $ [17,19,8,9] $, $ [3, 4, 16, 17, 19] $ are not faces, which indicate $ \phi_1 $, $ \phi_2 $ are not automorphism map. Therefore $ Aut( \mathcal{K}_1)=\langle\alpha_1,\alpha_2\rangle $ and by GAP\cite{GAP4}, we have $ \langle\alpha_1,\alpha_2\rangle\cong D_2 $ (Dihedral Group of order 4).
	
	Let $ \phi: \mathcal{K}_2\rightarrow  \mathcal{K}_2 $ be an automorphism map of $  \mathcal{K}_2 $. Then from ${EG( G_6( \mathcal{K}_2))} $ we have $ \{\phi(10),\phi(16)\}=\{10,16\} $. From link of $  \mathcal{K}_2 $ and ${EG( G_6( \mathcal{K}_2))} $, we obtain: for $ \phi(10)=10 $, $ \phi=Identity $ and for $ \phi(10)=16 $, $ \phi=\beta $. Therefore $ Aut( \mathcal{K}_2)=\langle\beta\rangle\cong \mathbb{Z}_2 $.
	
	We have $ \gamma_1 $, $ \gamma_2 $ are automorphism map of $  \mathcal{K}_3 $. Therefore $ \gamma_1\circ\gamma_2=(1\ 19\ 6\ 14)$ $(2\ 17\ 12\ 15)$ $(3\ 20\ 11\ 9)$ $(4\ 7\ 5\ 8)$ $(10\ 16)$ $(18\ 13) $, $ (\gamma_1\circ\gamma_2)^2=(1\ 6)$ $(2\ 12)$ $(3\ 11)$ $(4\ 5)$ $(7\ 8)$ $(9\ 20)$ $(10\ 18)$ $(13\ 16)$ $(14\ 19)$ $(15\ 17) $, $ (\gamma_1\circ\gamma_2)^3=(1\ 14\ 6\ 19)$ $(2\ 15\ 12\ 17)$ $(3\ 9\ 11\ 20)$ $(4\ 8\ 5\ 7)$ $(10\ 13\ 18\ 16) $, $ \gamma_2\circ\gamma_1\circ\gamma_2=(1\ 7)$ $(2\ 20)$ $(3\ 17)$ $(4\ 19)$ $(5\ 14)$ $(6\ 8)$ $(9\ 12)$ $(10\ 18)$ $(11\ 15) $, $ \gamma_1\circ\gamma_2\circ\gamma_1=(1\ 4)$ $(2\ 3)$ $(5\ 6)$ $(7\ 14)$ $(8\ 19)$ $(9\ 17)$ $(10\ 16)$ $(11\ 12)$ $(13\ 18)$ $(15\ 20) $ are automorphism maps as composition of two automorphism map is an automorphism map. From ${ EG(G_6( \mathcal{K}_3 ))} $ we have these are all possible maps, therefore $ Aut( \mathcal{K}_3 ) =\langle\gamma_1,\gamma_2\rangle $ and by GAP\cite{GAP4}, we get $ \langle\gamma_1,\gamma_2\rangle \cong D_4$ (Dihedral Group of order 8).
\end{proof}

\begin{lemma}\label{lem7} If type is  $[6^2, 7^1], [3^1, 4^1, 7^1, 4^1],$ $[4^1, 6^1, 14^1]$ or $[4^1, 8^1, 10^1]$ then there exists a semi-equivelar map on the surface of Euler char.  $-1 $.
\end{lemma}
\begin{proof}
The proof follows from Example \ref{eg:8maps-torus}.
\end{proof}

\begin{lemma}\label{lem8}
	Let  $X $ be a semi-equivelar map of type  $[p_1^{n_1}, \dots, p_k^{n_k}] $ on the surface of Euler char.  $-1 $. Then,  $[p_1^{n_1}, \dots, p_k^{n_k}] \neq [3^4, 7^1]$.
\end{lemma}

\begin{proof}
Let $X$ be a map of type $[3^4,7^1]$ on the surface of \Echar{-1}. Then, by Euler equation, the number of vertices $|V(X)|=42$. 
\begin{figure}[H]
	\begin{center}
		\psscalebox{0.9 0.9} 
		{
			\begin{pspicture}(0,-6.6499205)(8.360267,1.1703591)
			\pspolygon[linecolor=black, linewidth=0.02](3.9802167,-3.8698602)(2.9602168,-4.42986)(2.8802166,-5.36986)(3.4602168,-6.04986)(4.4202166,-6.04986)(5.0402164,-5.38986)(4.9002166,-4.44986)
			\pspolygon[linecolor=black, linewidth=0.02](4.3002167,-3.3698602)(4.3402166,-2.2498603)(5.4402165,-1.7298602)(6.3202167,-2.3298602)(6.3802166,-3.1698601)(6.0002165,-3.7698603)(5.2402167,-3.9898603)
			\pspolygon[linecolor=black, linewidth=0.02](5.7202168,-1.1698602)(5.7602167,-0.049860228)(6.600217,0.49013978)(7.5002165,0.17013977)(7.8402166,-0.60986024)(7.5202165,-1.4298602)(6.580217,-1.7498603)
			\pspolygon[linecolor=black, linewidth=0.02](5.160217,0.010139771)(5.100217,-1.1498603)(4.0202165,-1.6898602)(3.2202168,-1.1698602)(3.2602167,0.09013977)(3.8202167,0.51013976)(4.580217,0.49013978)
			\pspolygon[linecolor=black, linewidth=0.02](2.6602166,0.09013977)(2.6002166,-1.1298603)(1.5402167,-1.6098602)(0.7802167,-1.1498603)(0.58021665,-0.36986023)(1.0202167,0.43013978)(1.9402167,0.6101398)
			\pspolygon[linecolor=black, linewidth=0.02](1.7602167,-2.1698601)(2.9002166,-1.6498603)(3.7202168,-2.2098603)(3.7202168,-3.3898602)(2.7002168,-3.8698602)(1.9802166,-3.6098602)(1.6402167,-2.9698603)
			\pspolygon[linecolor=black, linewidth=0.02](3.7002168,-3.3698602)(3.9802167,-3.8698602)(4.3202167,-3.3698602)
			\pspolygon[linecolor=black, linewidth=0.02](4.0202165,-1.6698602)(3.7202168,-2.2098603)(4.3402166,-2.2498603)
			\pspolygon[linecolor=black, linewidth=0.02](5.100217,-1.1298603)(5.7402167,-1.2098602)(5.4002166,-1.7298602)
			\pspolygon[linecolor=black, linewidth=0.02](2.5802166,-1.1098602)(3.2202168,-1.1898602)(2.9002166,-1.6498603)
			\psline[linecolor=black, linewidth=0.02](2.7202168,-3.8498602)(2.4002166,-4.36986)(2.3002167,-5.2098603)(2.4602168,-5.7898602)(2.9002166,-6.2898602)(3.5402167,-6.6098604)(4.3402166,-6.6298604)(5.0002165,-6.2898602)(5.3602166,-5.84986)(5.600217,-5.2298603)(5.5002165,-4.46986)(5.2602167,-3.9898603)
			\psline[linecolor=black, linewidth=0.02](5.5002165,-4.42986)(6.0402164,-4.30986)(6.4802165,-3.92986)(6.7802167,-3.5498602)(6.9402165,-2.8898602)(6.9002166,-2.3498602)(6.580217,-1.7498603)
			\psline[linecolor=black, linewidth=0.02](6.8802166,-2.3298602)(7.640217,-2.0498602)(8.120216,-1.5498602)(8.340217,-0.92986023)(8.320217,-0.22986023)(8.0002165,0.33013976)(7.5402164,0.7501398)(6.8202167,0.9901398)(6.180217,0.89013976)(5.5002165,0.45013976)(5.180217,-0.00986023)
			\psline[linecolor=black, linewidth=0.02](5.7602167,-0.049860228)(5.4802165,0.47013977)(5.0202165,0.8501398)(4.4802165,1.0301398)(3.9202166,1.0301398)(3.3802166,0.8701398)(2.9802167,0.53013974)(2.6602166,0.09013977)
			\psline[linecolor=black, linewidth=0.02](3.2402167,0.09013977)(2.9802167,0.5501398)(2.4002166,0.9901398)(1.8002167,1.1501398)(1.0802166,1.0301398)(0.6002167,0.7501398)(0.16021667,0.15013976)(0.020216675,-0.48986024)(0.14021668,-1.0898602)(0.48021668,-1.6498603)(1.2002167,-2.0898602)(1.7602167,-2.1698601)
			\psline[linecolor=black, linewidth=0.02](1.5202167,-1.5898602)(1.2202166,-2.1098602)(1.1002166,-2.7498603)(1.1802167,-3.2498603)(1.4402167,-3.7498603)(1.8202167,-4.1698604)(2.3802166,-4.36986)(2.9802167,-4.42986)
			\psline[linecolor=black, linewidth=0.02](6.560217,-1.7498603)(6.3402166,-2.3498602)(6.9002166,-2.3298602)
			\psline[linecolor=black, linewidth=0.02](5.2602167,-3.9898603)(4.9202166,-4.46986)(5.5002165,-4.42986)
			\psline[linecolor=black, linewidth=0.02](2.7202168,-3.8498602)(2.9602168,-4.40986)
			\psline[linecolor=black, linewidth=0.02](5.180217,-0.00986023)(5.7602167,-0.02986023)
			\psline[linecolor=black, linewidth=0.02](2.6602166,0.09013977)(3.2602167,0.09013977)
			\psline[linecolor=black, linewidth=0.02](1.5402167,-1.6098602)(1.7602167,-2.1498601)
			\psline[linecolor=black, linewidth=0.02](2.3802166,1.0101398)(1.9402167,0.6101398)(1.8002167,1.1501398)
			\psline[linecolor=black, linewidth=0.02](1.0802166,1.0301398)(1.0402167,0.43013978)(0.58021665,0.77013975)
			\psline[linecolor=black, linewidth=0.02](0.16021667,0.13013977)(0.58021665,-0.36986023)(0.020216675,-0.52986026)
			\psline[linecolor=black, linewidth=0.02](0.14021668,-1.0898602)(0.7802167,-1.1698602)(0.50021666,-1.6698602)
			\psline[linecolor=black, linewidth=0.02](3.3802166,0.8701398)(3.8002167,0.49013978)(3.9202166,1.0501398)
			\psline[linecolor=black, linewidth=0.02](4.4802165,1.0101398)(4.580217,0.47013977)(5.0002165,0.8501398)
			\psline[linecolor=black, linewidth=0.02](6.140217,0.9101398)(6.600217,0.47013977)(6.8202167,0.97013974)
			\psline[linecolor=black, linewidth=0.02](7.5202165,0.7501398)(7.5002165,0.17013977)(8.020217,0.31013978)
			\psline[linecolor=black, linewidth=0.02](8.300217,-0.20986024)(7.8402166,-0.60986024)(8.320217,-0.92986023)
			\psline[linecolor=black, linewidth=0.02](8.120216,-1.5498602)(7.5202165,-1.4498602)(7.620217,-2.0898602)
			\psline[linecolor=black, linewidth=0.02](1.1002166,-2.7498603)(1.6402167,-2.9898603)(1.2002167,-3.2898602)
			\psline[linecolor=black, linewidth=0.02](1.4202167,-3.7498603)(2.0002167,-3.6298602)(1.8402166,-4.1698604)
			\psline[linecolor=black, linewidth=0.02](2.3002167,-5.2098603)(2.8802166,-5.38986)(2.4602168,-5.80986)
			\psline[linecolor=black, linewidth=0.02](2.8802166,-6.2898602)(3.4602168,-6.04986)(3.5402167,-6.5898604)
			\psline[linecolor=black, linewidth=0.02](4.3602166,-6.6098604)(4.4202166,-6.04986)(5.0002165,-6.2898602)
			\psline[linecolor=black, linewidth=0.02](5.3602166,-5.82986)(5.0202165,-5.36986)(5.600217,-5.2298603)
			\psline[linecolor=black, linewidth=0.02](6.0202165,-4.30986)(6.0002165,-3.7498603)(6.4802165,-3.9098601)
			\psline[linecolor=black, linewidth=0.02](6.7802167,-3.5098603)(6.3602166,-3.1698601)(6.9202166,-2.92986)
			\rput[bl](3.7802167,-4.32986){\scriptsize{29}}
			\rput[bl](4.4402165,-4.7298603){\scriptsize{30}}
			\rput[bl](4.560217,-5.46986){\scriptsize{31}}
			\rput[bl](4.120217,-5.86986){\scriptsize{32}}
			\rput[bl](3.4602168,-5.90986){\scriptsize{33}}
			\rput[bl](3.0402167,-5.44986){\scriptsize{34}}
			\rput[bl](3.0802166,-4.6698604){\scriptsize{35}}
			\rput[bl](4.4602165,-3.3498602){\scriptsize{7}}
			\rput[bl](4.4602165,-2.5498602){\scriptsize{1}}
			\rput[bl](5.2602167,-2.1898603){\scriptsize{2}}
			\rput[bl](6.0002165,-2.5898602){\scriptsize{3}}
			\rput[bl](6.060217,-3.1898603){\scriptsize{4}}
			\rput[bl](5.7802167,-3.7298603){\scriptsize{5}}
			\rput[bl](5.160217,-3.7898602){\scriptsize{6}}
			\rput[bl](3.3402166,-3.3698602){\scriptsize{8}}
			\rput[bl](3.4002166,-2.4898603){\scriptsize{9}}
			\rput[bl](2.6002166,-2.0898602){\scriptsize{12}}
			\rput[bl](1.8402166,-2.5098603){\scriptsize{13}}
			\rput[bl](1.8602166,-3.0498602){\scriptsize{14}}
			\rput[bl](2.0602167,-3.5498602){\scriptsize{15}}
			\rput[bl](2.5602167,-3.7298603){\scriptsize{16}}
			\rput[bl](3.9202166,-1.4698602){\scriptsize{10}}
			\rput[bl](4.640217,-1.0698602){\scriptsize{11}}
			\rput[bl](4.660217,-0.28986022){\scriptsize{17}}
			\rput[bl](4.2402167,0.13013977){\scriptsize{18}}
			\rput[bl](3.7202168,0.13013977){\scriptsize{19}}
			\rput[bl](3.3802166,-0.18986022){\scriptsize{20}}
			\rput[bl](3.3802166,-1.1098602){\scriptsize{21}}
			\rput[bl](5.8602166,-1.1698602){\scriptsize{22}}
			\rput[bl](5.8402166,-0.34986022){\scriptsize{23}}
			\rput[bl](6.4202166,0.05013977){\scriptsize{24}}
			\rput[bl](7.0402164,-0.20986024){\scriptsize{25}}
			\rput[bl](7.3002167,-0.78986025){\scriptsize{26}}
			\rput[bl](7.160217,-1.3298602){\scriptsize{27}}
			\rput[bl](6.560217,-1.5498602){\scriptsize{28}}
			\rput[bl](2.0802166,-1.1098602){\scriptsize{36}}
			\rput[bl](2.2002168,-0.22986023){\scriptsize{37}}
			\rput[bl](1.7202166,0.19013977){\scriptsize{38}}
			\rput[bl](1.0402167,0.010139771){\scriptsize{39}}
			\rput[bl](0.7402167,-0.48986024){\scriptsize{40}}
			\rput[bl](0.9202167,-1.0698602){\scriptsize{41}}
			\rput[bl](1.4402167,-1.4298602){\scriptsize{42}}
			\psline[linecolor=black, linewidth=0.02](3.7202168,-2.2098603)(4.2802167,-3.3498602)
			\psline[linecolor=black, linewidth=0.02](4.3402166,-2.2298603)(5.100217,-1.1498603)
			\psline[linecolor=black, linewidth=0.02](5.4002166,-1.7298602)(6.600217,-1.7498603)
			\psline[linecolor=black, linewidth=0.02](6.3202167,-2.3098602)(6.9002166,-2.9098601)
			\psline[linecolor=black, linewidth=0.02](6.3802166,-3.1698601)(6.5002165,-3.8898602)
			\psline[linecolor=black, linewidth=0.02](6.0202165,-3.7498603)(5.5002165,-4.44986)
			\psline[linecolor=black, linewidth=0.02](3.9802167,-3.8498602)(5.2602167,-4.00986)
			\psline[linecolor=black, linewidth=0.02](2.9602168,-4.40986)(3.7402167,-3.3898602)
			\psline[linecolor=black, linewidth=0.02](2.4002166,-4.36986)(2.8802166,-5.38986)
			\psline[linecolor=black, linewidth=0.02](2.4602168,-5.7698603)(3.4602168,-6.04986)
			\psline[linecolor=black, linewidth=0.02](3.5202167,-6.5898604)(4.4202166,-6.04986)
			\psline[linecolor=black, linewidth=0.02](4.9602165,-6.2898602)(5.0402164,-5.38986)
			\psline[linecolor=black, linewidth=0.02](4.9002166,-4.44986)(5.600217,-5.2498603)
			\psline[linecolor=black, linewidth=0.02](1.8402166,-4.1498604)(2.7402167,-3.8698602)
			\psline[linecolor=black, linewidth=0.02](1.2202166,-3.2898602)(2.0002167,-3.6298602)
			\psline[linecolor=black, linewidth=0.02](1.2202166,-2.0898602)(1.6402167,-2.9898603)
			\psline[linecolor=black, linewidth=0.02](2.6202166,-1.1098602)(1.7602167,-2.1898603)
			\psline[linecolor=black, linewidth=0.02](2.9002166,-1.6498603)(4.0402164,-1.7098602)
			\psline[linecolor=black, linewidth=0.02](2.6402166,0.09013977)(3.2002168,-1.1498603)
			\psline[linecolor=black, linewidth=0.02](3.3802166,0.8701398)(3.2602167,0.11013977)
			\psline[linecolor=black, linewidth=0.02](3.7802167,0.53013974)(4.5002165,1.0101398)
			\psline[linecolor=black, linewidth=0.02](4.560217,0.49013978)(5.4802165,0.45013976)
			\psline[linecolor=black, linewidth=0.02](5.140217,0.03013977)(5.7202168,-1.1898602)
			\psline[linecolor=black, linewidth=0.02](5.7402167,-0.00986023)(6.180217,0.8701398)
			\psline[linecolor=black, linewidth=0.02](6.600217,0.51013976)(7.5202165,0.7301398)
			\psline[linecolor=black, linewidth=0.02](7.5202165,0.17013977)(8.300217,-0.20986024)
			\psline[linecolor=black, linewidth=0.02](7.8402166,-0.60986024)(8.100217,-1.5298603)
			\psline[linecolor=black, linewidth=0.02](7.5002165,-1.4298602)(6.9002166,-2.3298602)
			\rput[bl](7.0002165,-2.6298602){\scriptsize{$ a_1 $}}
			\rput[bl](6.9802165,-3.1698601){\scriptsize{$ a_2 $}}
			\rput[bl](6.8202167,-3.7498603){\scriptsize{$ a_3 $}}
			\rput[bl](6.4602165,-4.1698604){\scriptsize{$ a_4 $}}
			\rput[bl](5.9202166,-4.6098604){\scriptsize{$ a_5 $}}
			\rput[bl](5.560217,-4.80986){\scriptsize{$ a_6 $}}
			\rput[bl](4.9002166,-3.0298603){{$ F_{7,1} $}}
			\rput[bl](2.5602167,-2.7898602){{$ F_{7,2} $}}
			\rput[bl](4.160217,-0.7498602){{$ F_{7,3} $}}
			\rput[bl](6.3802166,-0.7498602){{$ F_{7,4} $}}
			\rput[bl](3.5602167,-5.0698605){{$ F_{7,5} $}}
			\rput[bl](1.4202167,-0.52986026){{$ F_{7,6} $}}
			\psline[linecolor=black, linewidth=0.02](0.5202167,-1.6498603)(1.5202167,-1.5898602)
			\psline[linecolor=black, linewidth=0.02](0.040216673,-0.5098602)(0.7802167,-1.1698602)
			\psline[linecolor=black, linewidth=0.02](0.6002167,0.7501398)(0.58021665,-0.34986022)
			\psline[linecolor=black, linewidth=0.02](1.0402167,0.47013977)(1.8002167,1.1301398)
			\psline[linecolor=black, linewidth=0.02](1.9202167,0.6101398)(2.9802167,0.53013974)
			\end{pspicture}
		}
	\end{center}
\caption{Type $ [3^4,7^1] $}
\label{sem_3471}
\end{figure}
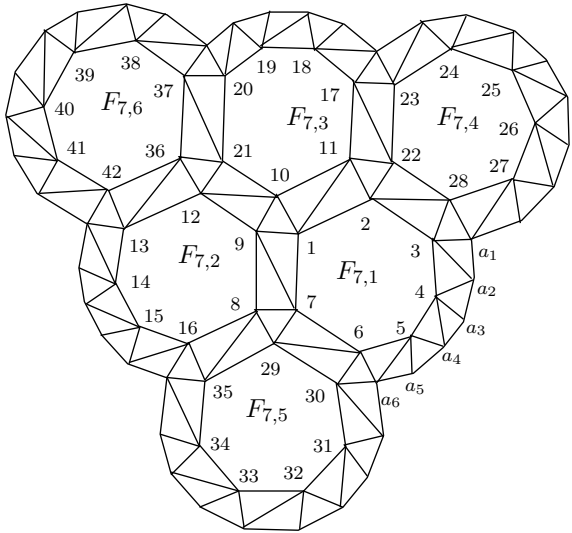
Let $lk(v)=C_{9}([v_1,v_2,v_3,v_4,v_5,v_6],v_7,v_8,v_9)$ denote link of a vertex $v\in V(X)$  where $[*]$ indicates that $v$ is appears in the 7-gon $[v,v_1,v_2,v_3,v_4,v_5,v_6]$, and $[v,v_7,v_6]$, $[v,v_7,v_8]$, $[v,v_8,v_9]$, $[v,v_1,v_9]$ form 3-gons. Without loss of generality, let $ V(X)=\{1,2,\dots, 42\} $, and \lkk{1}{2,3,4,5,6,7}{8}{9}{10}. Clearly, $ X $ has six 7-gons and all are disjoints, namely, $ F_{7,1} $, $ F_{7,2} $, $ F_{7,3} $, $ F_{7,4} $, $ F_{7,5} $, and $ F_{7,6} $. Without loss of generality, consider $ [8,9] $ which is adjacent to a 7-gon, hence, let $ F_{7,1}=[1,2,3,4,5,6,7] $, $ F_{7,2}=[8,9,12,13,14,15,16] $, $ F_{7,3}=[10,11,17,18,19,20,21] $, $ F_{7,4}=[22,23,24,25,26,27,28] $, $ F_{7,5}=[29,30,31,32,33,34,35] $, $ F_{7,6}=[36,37,38,39,40,41,42] $. Thus, we have the polyhedral complex of $ X $ illustrated in Figure \ref{sem_3471}.
It is clear that, one of $ [a_1,a_2] $, $ [a_3,a_4] $, $ [a_5,a_6] $ is adjacent to $ F_{7,6} $. If $ [a_1,a_2] $ adjacent to $ F_{7,6} $, then it is sufficient to check $ a_1=38 $ or $ a_1 \neq 38 $. If $ [a_5,a_6] $ is adjacent to $ F_{7,6} $, then this case is isomorphic to the case : $ [a_1,a_2] $ is adjacent to $ F_{7,6} $, under the map $ (2\ 7)(3\ 6)(4\ 5)(8\ 11) (9\ 10) (12\ 21) (13\ 20) (14\ 19) (15\ 18) (16\ 17) (22\ 29) (23\ 35) (24\ 34) (25\ 33) (26\ \allowbreak 32) (27\ 31) (28\ 30) (37\ 42) (38\ 41) (39\ 40) $. We see that if $ a_1=38 $, then $ [a_3,a_4] $ can not be adjacent to $ F_{7,2} $, $ F_{7,3} $, $ F_{7,4} $, and $ F_{7,6} $, i.e. $ [a_1,a_2] $ and $ [a_5,a_6] $ can not be adjacent to $ F_{7,6} $. If $ [a_3,a_4] $ is adjacent to $ F_{7,6} $, then it is sufficient to check whether $ a_3=38 $ or not. If $ a_3=38 $, it is easy to check that $ a_6\notin V(F_{7,2}), V(F_{7,3}), V(F_{7,4}) $, where $ V(F_{7,i}) $ indicate the vertex set of the 7-gon $ F_{7,i} $. Therefore, the map $ X $ of type $ [3^4,7^1] $ does not exist on the surface of \Echar{-1}.
\end{proof}

\begin{lemma}\label{lem9}
	Let  $X $ be a semi-equivelar map of type  $[p_1^{n_1}, \dots, p_k^{n_k}] $ on the surface of Euler char.  $-1 $. Then,  $[p_1^{n_1}, \dots, p_k^{n_k}] \neq [3^1, 7^1, 3^1, 7^1], [3^1, 14^2]$.
\end{lemma}
\begin{proof}
Let $ X $ be a map of type $ [3^1,7^1,3^1,7^1] $ on the surface of \Echar{-1}, then by Euler equation, we have, $ |V(X)|=21 $. Denote link of a vertex $ v\in V() $ as $ lk(v)=C_{12}([v_1,v_2,v_3,v_4,v_5,v_6],[v_7,v_8,v_9,v_{10},v_{11},v_{12}]) $, where $ [*] $ is only to indicate $ [v,v_1,v_2,v_3,v_4,\allowbreak v_5,v_6] $, $ [v,v_7,v_8,v_9,v_{10},v_{11},v_{12}] $ form 7-gons, and $ [v,v_6,v_7] $, $ [v,v_1,v_{12}] $ form 3-gons. Then it is clear that to complete link of three vertices, we need more than 21 vertices, which is a contradiction. Therefore, there does not exist any map $ X $ of type $ [3^1,7^1,3^1,7^1] $ on the surface of \Echar{-1}. Similarly, if $ X $ is a map of type $ [3^1,14^2] $, then to complete links of four vertices, it needs more than 42 vertices, which is not possible. Therefore, there dose not exist any map of type $ [3^1,14^2] $ on the surface of \Echar{-1}. This completes the proof.
\end{proof}

\begin{proof}[Proof of Theorem \ref{theo:sf}]
Let $X $ be an $n $-vertex  map on the surface of Euler char. $\chi = -1 $ of type  $[p_1^{n_1}, \dots, $ $p_{\ell}^{n_{\ell}}] $.  Then  by Lemma \ref{lem1}, $(n, [p_1^{n_1}, \dots, p_{\ell}^{n_{\ell}}]) $ $= (12, [3^5, 4^1]), (42, [3^4, 7^1]), $ $ (20, [3^2, $ $4^1, 3^1, 5^1]), (12, $ $[3^1, 4^1, 3^1, 4^2]), $ $ (24, [3^4, 8^1]), $ $  (42, [3^1, 4^1, 7^1, 4^1]), $ $(24, [3^1, 4^1, 8^1, 4^1]), $ $ (15,[3^1, 5^3]) $, $(12, [3^1, 6^1, $ $4^1, 6^1]) $, $(21,[3^1, 7^1, 3^1, 7^1]), $  $(84, [4^1, $ $ 6^1, 14^1]), $ $(48, [4^1, 6^1, 16^1]), (40, $ $[4^1,8^1, 10^1]), (24, [6^2,  8^1]), $ $(42, [3^1, 14^2]), (42, [7^1, 6^2]), (20, $ $[4^3, $ $5^1]) $.

Clearly, $[p_1^{n_1}, \dots, p_{\ell}^{n_{\ell}}] \neq  [3^1, 6^1, 4^1, 6^1]$, $[3^2, 4^1, 3^1, 6^1], [4^3, 6^1] $ and $[4^1, 8^1, 12^1]$ by Prop. \ref{prop20}, $[p_1^{n_1},\allowbreak \dots, p_{\ell}^{n_{\ell}}] \neq [3^1, 5^3] $ by Prop. \ref{prop19}, and  $[p_1^{n_1}, \dots, p_{\ell}^{n_{\ell}}] \neq [3^2, 4^1, 3^1, 5^1] $ by Lemma \ref{lem6}.

If $[p_1^{n_1}, \dots, p_{\ell}^{n_{\ell}}] = [3^1, 4^1, 3^1, 4^2], [3^5, 4^1], [3^1, 4^1, $ $8^1, 4^1], [6^2, 8^1] $ or $[4^1, 6^1, 16^1] $ then we know the list of maps from Prop. \ref{prop14}, \ref{prop15} $\&$ \ref{prop16}.
If $[p_1^{n_1}, \dots, p_{\ell}^{n_{\ell}}] $ = $[4^3, 5^1] $ then also we know the list of maps from Lemma \ref{lemma1}.

If $[p_1^{n_1}, \dots, p_{\ell}^{n_{\ell}}] = [6^2, 7^1] $ then there exists a map $ \mathcal{K}_{5}$ by Lemma \ref{lem7}. 

If $[p_1^{n_1}, \dots, p_{\ell}^{n_{\ell}}] = [3^1, 4^1, 7^1, 4^1] $ then there exists a map $ \mathcal{K}_{4}$ by Lemma \ref{lem7}. 

If $[p_1^{n_1}, \dots, p_{\ell}^{n_{\ell}}] = [4^1, 6^1, 14^1]$ then there exists a map $ \mathcal{K}_{6}$ by Lemma \ref{lem7}.

If $[p_1^{n_1}, \dots, p_{\ell}^{n_{\ell}}] = [4^1, 8^1, 10^1] $ then there exists a map $ \mathcal{K}_{7}$ by Lemma \ref{lem7}.

Clearly, $[p_1^{n_1}, \dots, p_{\ell}^{n_{\ell}}] \neq [3^4, 7^1] $ by Lemma \ref{lem8} and $\neq [3^1, 7^1, 3^1, 7^1], [3^1, 14^2]$ by Lemma \ref{lem9}.

This completes the proof.
\end{proof}

\begin{proof}[Proof of Theorem \ref{theo:sf1}] Let $K$ be a semi-equivelar map of type $ [4^3,5^1] $ on the surface of \Echar{-1}. Then, $K \cong  \mathcal{K}_1 $, $  \mathcal{K}_2 $ or $  \mathcal{K}_3 $ from Lemma \ref{lemma1}, and $ \mathcal{K}_1 \ncong  \mathcal{K}_2 \ncong  \mathcal{K}_3 $ by Lemma \ref{lemma2}. This completes the proof.
\end{proof}

\begin{proof}[Proof of Theorem \ref{theo:sf2}] By Theorem \ref{theo:sf1}, $ \mathcal{K}_1,  \mathcal{K}_2,  \mathcal{K}_3$ are the maps of type $[4^3, 5^1]$ on the surface of Euler char. -1. We also know from Lemma \ref{lemma2} that $Aut( \mathcal{K}_1)=\langle\alpha_1,\alpha_2\rangle\cong D_2 $, $ Aut( \mathcal{K}_2)=\langle\beta\rangle\cong\mathbb{Z}_2 $, $ Aut( \mathcal{K}_3)=\langle\gamma_1,\gamma_2\rangle\cong D_4 $, where $ \alpha_1=(1\ 3)$ $(2\ 4)$ $(5\ 11)$ $(6\ 12)$ $(7\ 19)$ $(8\ 14)$ $(9\ 15)$ $(13\ 16)$ $(10\ 18)$ $(17\ 20) $, $ \alpha_2=(1\ 6)$ $(2\ 11)$ $(3\ 12)$ $(4\ 5)$ $(7\ 8)$ $(9\ 20)$ $(10\ 13)$ $(14\ 19)$ $(15\ 17)$ $(16\ 18) $, $ \beta=(1\ 4)$ $(2\ 3)$ $(5\ 6)$ $(7\ 14)$ $(8\ 19)$ $(9\ 18)$ $(10\ 16)$ $(11\ 12)$ $(13\ 17)$ $(15\ 20) $, $ \gamma_1=(1\ 8)$ $(2\ 9)$ $(3\ 15)$ $(4\ 14)$ $(5\ 19)$ $(6\ 7)$ $(11\ 17)$ $(12\ 20)$ $(13\ 16) $, $ \gamma_2=(1\ 5)$ $(2\ 11)$ $(3\ 12)$ $(4\ 6)$ $(7\ 19)$ $(8\ 14)$ $(9\ 15)$ $(10\ 13)$ $(16\ 18)$ $(17\ 20) $. So, for each $i$, the number of vertex orbits is more than one in $ \mathcal{K}_i$ under $Aut( \mathcal{K}_i)$. Therefore, $ \mathcal{K}_1,  \mathcal{K}_2,  \mathcal{K}_3$ are not vertex-transitive.

We know that if a map $ \mathcal{K} $ is vertex-transitive, then the graph $ G_i(\mathcal{K}) $ is vertex-transitive for all $ i $ and hence, the degree of each vertex of the graph $ G_i(\mathcal{K}) $ is same, for all $ i $. Now we can see that, in $ G_4(\mathcal{K}_{4}) $, degree of vertices $1, 4, 9, 15, 19$ is $4$, and degree of vertices $ 6, 13, 17, 22, 28, 29, 35, 36, 37 $ is $ 0 $; in $ G_6(\mathcal{K}_{5}) $, degree of vertices $1, 2, 4, 5, 9, 10, 11, 12,\allowbreak 15, 16, 19,\allowbreak 20$ is $5$, and degree of vertices $ 3, 8, 21, 23, 26, 31, 34, 39, 42 $ is 4; in $ G_6(\mathcal{K}_{6}) $, degree of vertices $24, 33, 41, 66, 72, 80$ is $5$, and degree of the vertex $ 6 $ is $ 0 $; in $ G_9(\mathcal{K}_{7}) $, degree of vertices $1, 9$ is $3$, and degree of vertices $ 23, 25 $ is $ 1 $. Hence, $ \mathcal{K}_{4} $, $ \mathcal{K}_{5} $, $\mathcal{K}_{6}$, and $ \mathcal{K}_{7} $ are not vertex-transitive.
\end{proof}


{\small

}

\end{document}